\documentclass[reqno,10pt]{article}

\usepackage{amsmath, amsthm, amsfonts, amssymb,mathrsfs,color}
\usepackage{authblk}

\usepackage{charter}

\usepackage
[
a4paper,
hmargin=1.8cm,
vmargin=1.8cm
]{geometry}

\usepackage{tensor}

\usepackage{hyperref}
\hypersetup{
	colorlinks,
	citecolor=blue,
	filecolor=black,
	linkcolor=blue,
	urlcolor=blue
}

\usepackage{url}

\usepackage{enumitem}

\usepackage[normalem]{ulem}
\usepackage[numbers,sort&compress]{natbib}

\usepackage{indentfirst}

\def\R{\mathbb R}

\def\E{\mathbb E}

\def\p{\mathbb P}

\def\F{\mathcal F}

\def\D{\mathcal D}

\def\Q{\mathcal Q}

\def\d{{\,\rm{d}}}

\def\I{{\mathbf I}}

\def\H{\mathcal H}

\def\pas{{\mathbb P}\text{-}{\rm a.s.}}
\def\pv{\text{p.v.}}

\numberwithin{equation}{section}

\newtheorem{Theorem}{Theorem}[section]

\newtheorem{Lemma}{Lemma}[section]

\newtheorem{Definition}{Definition}[section]

\newtheorem{Hypothesis}{Assumption}[section]

\theoremstyle{definition}
\newtheorem{Example}{Example}[section]
\newtheorem{Remark}{Remark}[section]

\newcommand{\IP}[1]{\left\langle#1\right\rangle }
\newcommand{\bIP}[1]{\big\langle#1\big\rangle }

\begin{document}
 
	\allowdisplaybreaks

	\title{\textbf{{Pseudo-differential noise and nonlocal singularity formation in the stochastic C\'ordoba--C\'ordoba--Fontelos equation}}}

	\author{\textbf{Diego Alonso-Or\'an}}
	\affil{{\small Departamento de An\'alisis Matem\'atico, Universidad de La Laguna, Astrof\'isico Francisco S\'anchez s/n, 38271 La Laguna, Spain. \href{mailto:dalonsoo@ull.edu.es}{dalonsoo@ull.edu.es}}}
	
	\author{\textbf{Rafael Granero-Belinch\'on}}
	\affil{{\small Departamento de Matem\'aticas, Estad\'istica y Computaci\'on, Universidad de Cantabria. Avda. Los
			Castros s/n, Santander, Spain. \href{mailto:rafael.granero@unican.es}{rafael.granero@unican.es}}}

	\author{\textbf{Yingting Miao}}
	\affil{{\small School of Mathematics and Physics, Xi'an Jiaotong-Liverpool University, 
			Suzhou 215123, China. \href{mailto:Yingting.Miao@xjtlu.edu.cn}{Yingting.Miao@xjtlu.edu.cn}}}
	
	\author{\textbf{Hao Tang}}
	\affil{{\small Center for Applied Mathematics, Tianjin University, Tianjin 300072, China. \href{mailto:haotang@tju.edu.cn}{haotang@tju.edu.cn}}}

	\date{\today}
	
	\maketitle

	\begin{abstract}
		We study the stochastic C\'ordoba--C\'ordoba--Fontelos equation driven by multiplicative Stratonovich noise. The noise amplitude is allowed to be a pseudo-differential operator whose leading part is nearly skew-adjoint. This class contains classical transport noise and also permits genuinely nonlocal perturbations. We first develop a local-in-time theory for maximal classical solutions in Sobolev spaces, proving existence, uniqueness, and a blow-up criterion. 
		We then consider the special case of Stratonovich transport. For sufficiently
		large initial nonlocal steepness at a global maximum, we prove finite-time
		blow-up with arbitrarily high prescribed probability and obtain an
		explicit upper bound on the lifespan. Finally, on the event that the
		nonlocal steepness at the transported maximum diverges, we establish a
		conditional Type-I upper bound. When the terminal Ces\`aro average of the
		normalised nonlocal energy converges, we further identify the exact
		leading-order blow-up rate in terms of its limiting value.
	\end{abstract}

	

	\tableofcontents

	\section{Introduction}\label{Section:introduction}

	We consider the following nonlocal transport equation on the real line:
	\begin{equation}\label{eq:deterministic-CCF}
		\partial_t v+(\mathcal{H} v)\partial_xv=0,
		\quad t>0,\quad x\in\mathbb{R},
	\end{equation}
	where the Hilbert transform is defined by the principal-value integral
	\[
	\mathcal{H}f(x)
	=
	\frac{1}{\pi}\pv\int_{\mathbb{R}}\frac{f(y)}{y-x} \,\mathrm{d}y.
	\]
	Equation~\eqref{eq:deterministic-CCF} was introduced by C\'ordoba,
	C\'ordoba, and Fontelos as a one-dimensional model retaining several
	nonlocal mechanisms that arise in active-scalar and fluid-interface
	dynamics, see\cite{Cordoba-Cordoba-Fontelos-2005-Annals}.  It is now  commonly
	referred to as the C\'ordoba--C\'ordoba--Fontelos (CCF) equation.
	
	\subsection{Motivation: robustness of nonlocal characteristic-compression under noise}
	
	Several structural features distinguish the CCF equation from local
	one-dimensional transport models. Let $X(t,x)$ denote the characteristic
	flow associated with a sufficiently regular solution $v$:
	\begin{equation}\label{eq:deterministic-characterize}
		\frac{\mathrm{d}}{\mathrm{d}t}X(t,x)
		=
		\mathcal{H} v(t,X(t,x)),
		\quad
		X(0,x)=x.
	\end{equation}
	As long as the flow remains well-defined and invertible,
	\[
	v(t,X(t,x))=v_0(x).
	\]
	Thus, the range of $v$ is preserved along the characteristics. The
	loss-of-regularity mechanism is therefore associated with the compression and
	degeneration of the characteristic map, rather than with the growth of the
	scalar amplitude. This is analogous to the inviscid Burgers equation.
	In the Burgers equation, however, the transport velocity is locally determined by the
	solution itself, whereas in \eqref{eq:deterministic-CCF}, it is generated
	nonlocally via the zeroth-order singular integral \(\mathcal{H}\).
	
	The connection with nonlocal vortex-stretching models becomes apparent
	after spatially differentiating the equation. Setting $g=\partial_xv$, we obtain
	\[
	\partial_tg+(\mathcal{H}v)\partial_xg
	=
	-(\mathcal{H}g)g.
	\]
	Equivalently, if $H_0=-\mathcal{H}$ denotes the opposite Hilbert-transform
	convention, then
	\[
	\partial_tg-(H_0v)\partial_xg=gH_0g.
	\]
	The differentiated equation therefore contains the stretching term of the
	Constantin--Lax--Majda model
	\cite{Constantin-Lax-Majda-1985-CPAM,Okamoto-etal-2008-Nonlinearity},
	coupled with a nonlocal transport drift. It is also structurally related to
	the De Gregorio model \cite{DeGregorio-1990-JSP,DeGregorio-1996-MMAS}, although the relation
	between the transported quantity and the velocity is different.
	
	The CCF equation \eqref{eq:deterministic-CCF} also provides a useful
	one-dimensional comparison with the surface quasi-geostrophic (SQG)
	equation \cite{Constantin-etal-1994-Nonlinearity,Constantin-Wu-1999-SIMA}.
	In both models, the transport velocity is recovered from the advected
	scalar through a zeroth-order singular integral. The distinction between
	the two models is most transparent at the level of their characteristic
	flows. Let \(X(t,x)\) denote the flow satisfying
	\eqref{eq:deterministic-characterize}, and  define  its spatial Jacobian by
	\(\mathfrak J_X(t,x) \triangleq \partial_xX(t,x)\). Differentiating the flow equation
	with respect to the initial label \(x\) and  using  the identity
	$\partial_x \mathcal{H} = -\Lambda$, we obtain
	\begin{equation}\label{eq:compressible-of-Hv}
		\partial_t\mathfrak J_X(t,x)
		=
		-\Lambda v(t,X(t,x))\mathfrak J_X(t,x),
		\implies
		\mathfrak J_X(t,x)
		=
		\exp\left(
		-\int_0^t
		\Lambda v(t',X(t',x))\,\mathrm{d}t'
		\right).
	\end{equation}
	Here, $\Lambda$  denotes the fractional Laplacian
	$$\Lambda \triangleq (-\partial_x^2)^{1/2},$$
	and \(\mathfrak J_X(t,x)\) measures the local separation of neighbouring
	characteristics. If \(x_0\) is a global maximum point of \(v_0\), then
	\(X(t,x_0)\) remains a global maximum point of \(v(t,\cdot)\) for as
	long as the classical flow exists. The maximum principle for the fractional Laplacian implies
	\[
	\Lambda v(t,X(t,x_0))\ge 0,
	\qquad
	\text{which yields} \quad \partial_t\mathfrak J_X(t,x_0)\le 0.
	\]
	Thus, neighbouring CCF characteristics are locally drawn together near
	the advected maximum. This exact relation details the sense in which the CCF dynamics
	exhibits a \emph{nonlocal characteristic-compression mechanism}.
	
	For the SQG equation, by contrast, the velocity field is divergence-free, and its
	characteristic flow \(X_{\mathrm{SQG}}\) is area-preserving:
	\[
	\det \nabla X_{\mathrm{SQG}}(t,x)=1.
	\]
	Consequently, any contraction in one direction must be compensated by an expansion in a transverse direction. In the one-dimensional
	CCF model, \(\mathfrak J_X=\partial_xX\) is the full Jacobian of the flow. Its decrease directly represents a loss of separation between  neighbouring
	characteristics, with no transverse dimension available to compensate for the compression.
 For the blow-up of solutions in related nonlocal models, we refer to \cite{Li-Rodrigo-2008-Adv,Li-Rodrigo-2020-Adv,Dong-2014-Adv}.

	The significance of the CCF equation therefore lies not so much in the bare
	fact that finite-time singularities may occur as in the clarity
	of this nonlocal characteristic-compression mechanism and its
	structural links to Euler vorticity dynamics, Birkhoff--Rott
	vortex-sheet dynamics, the Constantin--Lax--Majda/De Gregorio
	models, and the critical SQG dynamics. In this sense, the CCF equation provides a
	\emph{canonical minimal setting} for studying the finite-time loss of
	\(C^1\) regularity driven by the convergence of  neighbouring
	characteristics, despite the conservation of the scalar range.
	
	It is precisely this structural simplicity that makes the CCF equation a natural benchmark for testing the robustness of this characteristic-compression mechanism under
	random perturbations. In realistic physical systems, reduced descriptions
	of fluid flows leave certain spatial and temporal scales
	unresolved, rendering the system subject to model uncertainty. Such effects are often
	represented through stochastic perturbations. This leads to the following
	robustness question:
	\begin{quote}
		\textbf{Robustness question:} \emph{Which features of the deterministic nonlocal characteristic-compression mechanism persist under random perturbations, and how does the noise alter the analytic and singularity-forming structure of the CCF dynamics?}
	\end{quote}
	
	Stochastic versions of the CCF equation have been studied in several
	settings. Alonso, Rohde, and Tang considered a stochastic CCF equation on
	the torus driven by Lie-transport noise and established local pathwise
	well-posedness together with a blow-up criterion, see
	\cite{Alonso-Rohde-Tang-2021-JNLS}. Alonso, Miao, and Tang studied It\^o
	multiplicative noise on the real line, see
	\cite{Alonso-Miao-Tang-2022-JDE}. In addition to local pathwise
	well-posedness and a blow-up criterion, they proved almost-sure global
	existence for certain sufficiently strong nonlinear noises under suitable
	structural and growth conditions. They also showed that, for a class of
	linear It\^o multiplicative noises and suitable initial data, finite-time
	singularity formation persists with positive probability.
	
	The present work extends the related results on the stochastic CCF equation in
	\cite{Alonso-Rohde-Tang-2021-JNLS} from local first-order differential
	operators to a class of spatially nonlocal pseudo-differential operators.
	It also complements the finite-time singularity result for linear It\^o
	multiplicative noise in \cite{Alonso-Miao-Tang-2022-JDE} by treating
	Stratonovich transport noise. Our two principal objectives are to establish
	a local theory for general admissible pseudo-differential noise and, in the
	transport-noise case, to study the corresponding nonlocal characteristic
	contraction  
	mechanism.
	
	\subsection{Stratonovich pseudo-differential noise}\label{Section:OPS noise}
	
	Let $
	(\Omega,\mathcal F,(\mathcal F_t)_{t\ge0},\mathbb P)$
	be a filtered probability space satisfying the usual conditions. Let $W$
	be a standard real-valued $(\mathcal F_t)_{t\ge0}$-Brownian motion. We write
	$\circ \d W(t)$ for Stratonovich integration.
	
	A classical source of transport noise is a stochastic Lagrangian flow of
	the form
	\[
	\mathrm{d}X(t,x)
	=
	h(t,X(t,x)) \d t
	+
	\sigma(t,X(t,x))\circ{\rm d}W(t),
	\quad
	X(0,x)=x;
	\]
	see, for example,
	\cite{Mikulevicius-Rozovskii-2004-SIAM,
		Mikulevicius-Rozovskii-2005-AoP,Holm-2015-ProcA}. If $V$ is transported as a
	scalar, the conservation law
	\[
	V(t,X(t,x))=V_0(x)
	\]
	formally leads to
	\[
	\mathrm{d} V+h\partial_xV \d t
	+\sigma\partial_xV\circ{\rm d}W(t)=0.
	\]
	If $\rho$ is transported as a density, the conserved quantity is instead
	\[
	\rho(t,X(t,x))\partial_xX(t,x)=\rho_0(x),
	\]
	and the corresponding stochastic continuity equation is
	\[
	\mathrm{d}\rho+\partial_x(h\rho) \d t
	+\partial_x(\sigma\rho)\circ{\rm d}W(t)=0.
	\]
	The precise first-order noise operator therefore depends on the physical
	quantity being transported.
	
	Abstracting this structure, we 
	refer to noise of the form
	\begin{equation*}
		\mathcal Qv \circ{\rm d}W(t),
		\quad
		\mathcal Q=a(x)\partial_x+b(x),
	\end{equation*}
	as \emph{transport-type noise}. Following the contributions in
	\cite{Mikulevicius-Rozovskii-2004-SIAM,
		Mikulevicius-Rozovskii-2005-AoP,Holm-2015-ProcA}, stochastic partial
	differential equations (SPDEs) driven by transport-type noise have been
	studied extensively, and substantial progress has been achieved.
	For incompressible
	fluid models, see, for example,
	\cite{Flandoli-Luo-2020-AoP,
		Flandoli-Luo-2021-PTRF,
		Crisan-Flandoli-Holm-2019-JNLS,
		Flandoli-Lisei-2004-SAA}.
	For shallow-water models, see
	\cite{Holden-Karlsen-Pang-2021-JDE,
		Holden-Karlsen-Pang-2023-DCDS,
		Albeverio-etal-2021-JDE,
		Galimberti-etal-2024-JDE,Crisan-Holm-2018-PhyD,Alonso-Rohde-Tang-2021-JNLS};
	for compressible fluid models, see
	\cite{Breit-etal-2022-SIMA,
		Karlsen-Tang-Wang-2026-arXiv}.
	
	In the context of advanced turbulence
	modelling, however, unresolved complex interactions frequently couple distant spatial
	regions or transfer energy across diverse Fourier scales. This provides an
	analytical motivation
	for allowing
	random perturbations to act through spatially nonlocal operators.
	This perspective aligns with contemporary approaches in anomalous transport and turbulence phenomenology, where local Eulerian dynamics are influenced by the global spatial domain, see, for
	example,
	\cite{Majda-Gershgorin-2013-PTRSL,
		Bernard-Erinin-2018-JFM,
		Hamba-2022-JFM}.
	
	Motivated by these considerations, we study the stochastic Cauchy problem
	\begin{equation}\label{eq:sccf}
		\mathrm{d} v+(\mathcal Hv)\partial_xv \d t
		=
		\mathcal Qv\circ{\rm d}W(t),
		\quad
		v(0)=v_0,
		\quad
		x\in\mathbb R,
	\end{equation}
	where  $\mathcal Q$ is \textit{a deterministic linear
		pseudo-differential operator}
	satisfying the structural assumptions stated
	in Section~\ref{sec:notation-background-results}. The admissible operators have a leading part that is
	skew-adjoint up to a lower-order perturbation. The class contains suitable
	local first-order operators of the form
	$\mathcal Q=a(x)\partial_x+b(x)$ and also genuinely nonlocal operators.

	We must emphasise that a
	general pseudo-differential noise is not generated by a pointwise Lagrangian
	stochastic flow.  Accordingly, when $\mathcal Q$ is genuinely
	nonlocal, equation~\eqref{eq:sccf} should be viewed as an analytic nonlocal
	extension of transport-type noise rather than as a classical Lagrangian
	transport model. In this case the equation couples two different nonlocal
	mechanisms: the Hilbert transform generates the nonlinear transport velocity
	and the characteristic-flow relation~\eqref{eq:compressible-of-Hv}, whereas
	$\mathcal Q$ determines the spatial action of the stochastic perturbation.
	
	Because $\mathcal Q$ may be unbounded, the Stratonovich notation in
	\eqref{eq:sccf} requires a precise interpretation. Indeed, noticing that
	\begin{equation*}
		\mathcal Qv\circ{\rm d}W(t)
		=
		\mathcal Qv \d W(t)
		+\frac12\mathcal Q^2v \d t,
	\end{equation*}
	then
	\eqref{eq:sccf} can be interpreted as
	\begin{equation}\label{Target problem SCCF}
		v(t)-v_0
		+
		\int_0^t
		\left[
		\mathcal Hv\,\partial_xv
		-
		\frac12\mathcal Q^2v
		\right](t')\d t'
		=
		\int_0^t \mathcal Qv(t')\d W(t'),
		\quad
		v(0)=v_0 .
	\end{equation}

	If $\mathcal Q$ has order $\delta_{\mathcal Q}>0$, then its
	pseudo-differential mapping properties give, formally,
	\[
	\mathcal Q:H^s\longrightarrow H^{s-\delta_{\mathcal Q}},
	\quad
	\mathcal Q^2:H^s\longrightarrow H^{s-2\delta_{\mathcal Q}},
	\]
	whereas
	\[
	(\mathcal Hv)\partial_xv\in H^{s-1}
	\]
	for $v\in H^s$. Thus the drift and diffusion terms naturally
	live in lower Sobolev spaces $H^{s-\max\{1, 2\delta_{\mathcal Q}\}}$, even though the solution is constructed in
	$H^s$. This derivative loss prevents a direct application of the standard
	theory for stochastic differential equations (SDEs) taking values in a
	separable Hilbert space.  For instance,  applying It\^{o}'s formula to \eqref{Target problem SCCF} generates singular quantities such as
	\begin{equation}\label{Q singular terms}
		\langle \mathcal{Q} v,v \rangle_{H^s}, \quad 
		\langle \mathcal{Q}^2 v,v \rangle_{H^s}, \quad 
		\text{and} \quad 
		\langle \mathcal{Q} v, \mathcal{Q} v \rangle_{H^s}.
	\end{equation}
	Controlling the terms in \eqref{Q singular terms} purely by $\|v\|_{H^s}^2$ is far from straightforward, especially when $\Q$ is of positive order. See also Remark \ref{Remark-local-theory-method} for further discussion.

	To the best of our knowledge, the systematic Sobolev theory essential for analysing nonlinear fluid SPDEs driven by positive-order Stratonovich pseudo-differential noise was first introduced by Tang and Wang in \cite{Tang-Wang-2022-arXiv}, and further developed in \cite{Tang-2023-JFA}.
	Subsequent advances and broader applications can be found in \cite{Alonso-Pang-Tang-2026-JLMS, Karlsen-Tang-Wang-2026-arXiv,Miao-2025-ZAMP}.

	\subsection{Nonlocality in the singularity formation}
	
	Stochastic perturbations, including transport noise, may affect singularity
	formation in qualitatively different ways. In some systems they suppress
	exceptional collisions or extend the lifespan with high probability; in
	others the deterministic breakdown mechanism persists with positive
	probability or almost surely. We recall a few representative examples.
	
	For the finite-dimensional point-vortex dynamics associated with the
	two-dimensional Euler equation, Flandoli, Gubinelli, and Priola proved that
	a generic stochastic perturbation compatible with the Eulerian description
	prevents vortex coalescence and yields global well-posedness for every
	initial configuration, see\cite{Flandoli-Gubinelli-Priola-2011-SPA}. In a
	high-frequency scaling regime, transport noise was shown in
	\cite{Flandoli-Galeati-Luo-2021-CPDE} to give long-time existence with high
	probability for several nonlinear PDEs that may otherwise develop
	finite-time singularities. For the three-dimensional Navier--Stokes
	equation, sufficiently strong high-mode transport noise yields vorticity
	bounds and high-probability well-posedness, see
	\cite{Flandoli-Luo-2021-PTRF}.
	
	Singularity formation may nevertheless persist under stochastic
	perturbations. For the inviscid stochastic Burgers equation with affine
	transport coefficient,
	\[
	\mathrm{d}u+u\partial_xu \d t
	=
	\xi(x)\partial_xu\circ{\rm d}W(t),
	\quad
	\xi(x)=a_{\mathrm{aff}}x+b_{\mathrm{aff}},
	\quad
	a_{\mathrm{aff}},\ b_{\mathrm{aff}}\in\mathbb R,
	\]
	it was proved in \cite{Alonso-etal-2019-NODEA} that, when the initial
	profile has a negative slope, two stochastic characteristics intersect in
	finite time almost surely. Wave breaking also persists in stochastic
	Camassa--Holm-type equations. Positive-probability wave breaking for the
	Camassa--Holm equation with Lie-transport noise was established in
	\cite{Crisan-Holm-2018-PhyD}. In a different setting, Rohde and Tang proved
	positive-probability wave breaking for a class of stochastic
	Dullin--Gottwald--Holm equations, including the Camassa--Holm equation,
	driven by non-autonomous linear It\^o multiplicative noise, see
	\cite{Rohde-Tang-2021-NoDEA}.
	
	A common feature of the preceding Burgers and Camassa--Holm analyses is the
	reduction of singularity formation to a scalar steepness observable tied to
	the Lagrangian or extremal geometry of the solution. Depending on the
	model, this observable may describe the separation of stochastic
	characteristics, the slope at a moving inflection point, or the minimum
	spatial slope.
	
	For the stochastic CCF equation, the local-in-time theory established in
	this paper (see Theorem \ref{Thm-local-solution}) gives the blow-up criterion
	\begin{equation*}
		\limsup_{t\uparrow\tau^*}
		\left(
		\|\partial_xv(t)\|_{L^\infty}
		+
		\|\Lambda v(t)\|_{L^\infty}
		\right)
		=\infty
		\quad
		\text{a.s. on}\ \{\tau^*<\infty\}.
	\end{equation*} 
	This criterion alone does not determine whether breakdown is caused by the
	local slope $\partial_xv$, the nonlocal velocity strain
	$\Lambda v=-\partial_x\mathcal Hv$, or both. Our singularity argument
	isolates the second quantity at a transported maximum point. Since the ability to follow such a maximum point is essential to our
	argument, and a
	general pseudo-differential perturbation need not admit a pointwise
	stochastic-flow representation, the finite-time singularity analysis is
	restricted to the scalar transport-noise equation
	\begin{equation}\label{eq:sccf-transport}
		\d v+(\mathcal Hv)\partial_xv \d t
		=
		\sigma(x)\partial_xv\circ{\rm d}W(t),
		\quad
		v(0)=v_0.
	\end{equation}
	For this noise, scalar values are preserved along stochastic
	characteristics.
	
	Let $x_0$ be a global  maximiser of the initial data,
	\[
	v_0(x_0)=\max_{x\in\mathbb R}v_0(x),
	\]
	and let $z_0(t)$ solve
	\begin{equation*}
		\d z_0(t)
		=
		(\mathcal Hv)(t,z_0(t)) \d t
		-
		\sigma(z_0(t))\circ{\rm d}W(t),
		\quad
		z_0(0)=x_0.
	\end{equation*}
	Under the assumptions of the blow-up theorem, the stochastic flow remains
	a diffeomorphism for $t<\tau^*$, and hence
	\[
	v(t,z_0(t))
	=
	v_0(x_0)
	=
	\max_{x\in\mathbb R}v(t,x),
	\quad
	t<\tau^*
	\quad\pas
	\]
	Thus $z_0(t)$ is the stochastic image of a fixed initial   maximiser, rather
	than a  maximiser selected independently at each time.
	
	We study the nonlocal quantity
	\[
	M(t)=(\Lambda v)(t,z_0(t)).
	\]
	For nonconstant $v_0$, since the scalar values are preserved by the stochastic flow,
	$v$ remains nonconstant for $t<\tau^*$,  the singular-integral
	representation of $\Lambda$ implies
	\[
	M(t)>0,
	\quad
	t<\tau^*
	\quad\pas
	\]
	Applying $\Lambda$ to \eqref{eq:sccf-transport} and evaluating the resulting
	stochastic field along $z_0(t)$ via the It\^o--Wentzell formula gives
	\begin{equation}\label{eq:M-Stratonovich}
		\d M(t)
		=
		-[\Lambda,\mathcal Hv]\partial_xv(t,z_0(t)) \d t
		+
		[\Lambda,\sigma]\partial_xv(t,z_0(t))\circ{\rm d}W(t),
	\end{equation}
	where, for a scalar coefficient $g$,
	\[
	[\Lambda,g]f
	=
	\Lambda(gf)-g\Lambda f.
	\]
	The drift retains the deterministic nonlocal compression mechanism. At the
	transported maximum point it satisfies a coercive estimate whose leading
	term is quadratic in $M(t)$, producing the Riccati-type growth underlying
	the blow-up argument.
	
	The stochastic term introduces a distinct difficulty. Unlike the local
	slope equation in Burgers-type models, transport noise acting on
	$\Lambda v$ produces a fractional commutator. Under the assumptions on
	$\sigma$, we prove the pointwise decomposition (see \eqref{eq:Gamma-decomp} below)
	\begin{equation}\label{eq:Lambda-sigma-decomposition}
		[\Lambda,\sigma]\partial_xv
		=
		\sigma'\Lambda v+\mathcal R_\sigma[v],
	\end{equation}
	where the lower-order remainder satisfies (see \eqref{eq:R-L-infty} below)
	\begin{equation*}
		\|\mathcal R_\sigma[v(t)]\|_{L^\infty}
		\le
		C_\sigma\|v_0\|_{L^\infty},
		\quad
		t<\tau^*
		\quad\pas
	\end{equation*}
	Because the stochastic coefficient in \eqref{eq:M-Stratonovich} depends on
	both the random field $v$ and the random evaluation point $z_0(t)$, the
	Stratonovich--It\^o correction must be computed for the coupled
	field--characteristic system. After localisation, we apply It\^o's formula
	to the reciprocal process
	\[
	Y(t)=\frac{1}{M(t)}.
	\]
	The coupled correction, the commutator decomposition
	\eqref{eq:Lambda-sigma-decomposition}, and the coercive drift estimate yield
	a closed one-sided inequality for $Y$, corresponding to a stochastic
	Riccati mechanism for $M$.
	
	Consequently,  for every $\varepsilon\in(0,1)$, there exists
	\[
	C^*=C^*(\sigma,\|v_0\|_{L^\infty},\varepsilon)>0
	\]
	such that, whenever $\Lambda v_0(x_0)>C^*$,
	\begin{equation*}
		\mathbb P\left(\tau^*\le\frac{8}{\Lambda v_0(x_0)}\right)
		\ge
		1-\varepsilon.
	\end{equation*}
	See Theorem \ref{Thm-blowup} and Remark \ref{Remark-blowup-novelty}.

	The same reciprocal process also gives a conditional rate theorem.  On the
	event that $M(t)$ itself diverges as $t\uparrow\tau^*$, we discover that it is
	the normalised nonlocal energy
	\[
	\eta(t)
	=
	\frac{
		\mathcal E[v(t,\cdot)](z_0(t))
	}{
		M(t)^2
	}
	\ge0
	\]
	that measures the contribution of the global scalar profile to the
	compression observed at the transported maximum. We first prove the
	conditional Type-I bound
	\[
	\limsup_{t\uparrow\tau^*}
	M(t)(\tau^*-t)
	\le2.
	\]
	If the terminal Ces\`aro average of $\eta$ has a strictly positive lower
	limit, then this universal upper bound is strictly improved. If, more
	precisely, the Ces\`aro average converges to a finite limit $\eta_*$, then
	the exact first-order asymptotics are given by
	\[
	M(t)(\tau^*-t)
	\longrightarrow
	\frac1{\frac12+\eta_*}.
	\]
	Thus the explicit stochastic integrating-factor contribution is
	asymptotically neutral at leading order. Any residual dependence of the
	first-order blow-up coefficient on the noise is encoded indirectly through
	the terminal Ces\`aro behaviour of the normalized nonlocal energy. When
	this average converges, its limit completely determines the asymptotic
	coefficient; see Theorem~\ref{Thm-blowup-rate} and Remark \ref{Remark-blowup-rate-novelty}.
	
	The use of a transported random   maximiser, the pointwise
	commutator decomposition, and the Stratonovich--It\^o correction
	for the coupled field--characteristic system constitute the main new
	ingredients  of our proof.

	\section{Notation, background and main results}
	\label{sec:notation-background-results}
	
	\subsection{Function spaces and related notation}
	
	Let $\mathscr S(\mathbb R;\mathbb R)$ denote the Schwartz space of
	real-valued rapidly decreasing $C^\infty$ functions on $\mathbb R$, and let
	$\mathscr S'(\mathbb R;\mathbb R)$ denote the space of real-valued tempered
	distributions. For $1\le p<\infty$, $L^p(\mathbb R;\mathbb R)$ is the
	standard Lebesgue space of real-valued measurable functions, and
	$L^\infty(\mathbb R;\mathbb R)$ is the space of essentially bounded
	real-valued functions. The $L^2$ inner product is
	\[
	\langle f,g\rangle_{L^2}
	\triangleq
	\int_{\mathbb R}f(x)g(x) \d x.
	\]
	
	Let $\mathrm i=\sqrt{-1}$. We use the Fourier transform convention
	\[
	(\mathscr Ff)(\xi)
	\triangleq
	\int_{\mathbb R}f(x){\rm e}^{-\mathrm ix\xi} \d x,
	\quad
	(\mathscr F^{-1}f)(x)
	\triangleq
	\frac{1}{2\pi}\int_{\mathbb R}f(\xi){\rm e}^{\mathrm ix\xi} \d\xi.
	\]
	For $s\in\mathbb R$, the Bessel potential operator
	$\mathcal D^s=(\I-\partial_x^2)^{s/2}$ is defined by
	\begin{equation*}
		\mathscr F(\mathcal D^sf)(\xi)
		\triangleq
		(1+|\xi|^2)^{s/2}\mathscr Ff(\xi).
	\end{equation*}
	The real Sobolev space $H^s(\mathbb R;\mathbb R)$ is
	\[
	H^s(\mathbb R;\mathbb R)
	\triangleq
	\left\{
	f\in\mathscr S'(\mathbb R;\mathbb R):
	\mathcal D^sf\in L^2(\mathbb R;\mathbb R)
	\right\},
	\]
	with norm and inner product
	\[
	\|f\|_{H^s}
	\triangleq
	\|\mathcal D^sf\|_{L^2},
	\quad
	\langle f,g\rangle_{H^s}
	\triangleq
	\langle\mathcal D^sf,\mathcal D^sg\rangle_{L^2}.
	\]
	For $s\ge0$, this is equivalently the completion of
	$C_0^\infty(\mathbb R;\mathbb R)$ with respect to the $H^s$ norm.
	
	The Hilbert transform $\mathcal H$ and the operator
	$\Lambda=(-\partial_x^2)^{1/2}$ are defined by
	\begin{equation*}
		[\mathscr F(\mathcal Hf)](\xi)
		=
		\mathrm i\,\mathrm{sgn}(\xi)[\mathscr Ff](\xi),
		\quad
		[\mathscr F(\Lambda f)](\xi)
		=
		|\xi|[\mathscr Ff](\xi).
	\end{equation*}
	For sufficiently regular functions, their singular-integral
	representations are
	\begin{equation}\label{eq: Hilbert Lambda kernel}
		\mathcal Hf(x)
		=
		\frac{1}{\pi}\,\pv\int_{\mathbb R}\frac{f(y)}{y-x} \d y,
		\quad
		\Lambda f(x)
		=
		\frac{1}{\pi}\,\pv\int_{\mathbb R}
		\frac{f(x)-f(y)}{(x-y)^2} \d y.
	\end{equation}
	With these conventions,
	\[
	\Lambda=-\mathcal H\partial_x=-\partial_x\mathcal H.
	\]
	Both $\mathcal H$ and $\Lambda$ commute with Bessel potential operators and
	with spatial derivatives on their natural domains.
	
	For $k\in\mathbb N_0\triangleq\mathbb N\cup\{0\}$, we use the standard
	space $W^{k,\infty}(\mathbb R;\mathbb R)$ with norm
	\[
	\|f\|_{W^{k,\infty}}
	\triangleq
	\sum_{\ell=0}^k\|\partial_x^\ell f\|_{L^\infty}.
	\]
	When no ambiguity can arise, we write
	\begin{equation*}
		H^s=H^s(\mathbb R;\mathbb R),
		\quad
		L^p=L^p(\mathbb R;\mathbb R),
		\quad
		W^{k,\infty}=W^{k,\infty}(\mathbb R;\mathbb R).
	\end{equation*}
	We also write $C_b^k(\mathbb R;\mathbb R)$ for the space of functions whose
	derivatives up to order $k$ are bounded and continuous.
	
	For topological spaces $E$ and $\widetilde E$,
	$C(E;\widetilde E)$ denotes the space of continuous maps from $E$ to
	$\widetilde E$. If $\widetilde E$ is a normed space,
	$C_b(E;\widetilde E)$ denotes the subspace of bounded continuous maps. If
	$E$ and $\widetilde E$ are Banach spaces,
	$\mathscr L(E;\widetilde E)$ denotes the space of bounded linear operators
	from $E$ to $\widetilde E$, equipped with the operator norm. Unless stated
	otherwise, $\mathcal A^*$ denotes the $L^2$-adjoint of an operator
	$\mathcal A$, initially defined on $\mathscr S(\mathbb R;\mathbb R)$.
	
	For $a,b\in\mathbb R$, we set
	\[
	a\wedge b\triangleq\min\{a,b\},
	\quad
	a\vee b\triangleq\max\{a,b\},
	\]
	and adopt the convention $\inf\varnothing=\infty$. For a stopping time
	$\tau$ and a Banach space $E$, the notation
	$v\in C([0,\tau);E)$ means that, for almost every $\omega$, the map
	$t\mapsto v(\omega,t)$ is continuous from $[0,\tau(\omega))$ into $E$.
	
	Throughout the paper, $C$ and $c$, possibly with subscripts, denote positive
	constants whose values may change from line to line. Their dependence is
	indicated when relevant. For nonnegative quantities $A$ and $B$, we write
	$A\lesssim B$ if $A\le CB$ for a constant independent of the parameters
	under consideration. For two linear operators $\mathcal A$ and
	$\mathcal B$, their commutator is
	\[
	[\mathcal A,\mathcal B]
	\triangleq
	\mathcal A\mathcal B-\mathcal B\mathcal A.
	\]
	
	\subsection{Structure of the pseudo-differential noise amplitude}
	
	For $m\in\mathbb R$, let
	$\mathcal S^m(\mathbb R\times\mathbb R;\mathbb C)$ denote the H\"ormander
	symbol class $S^m_{1,0}$. More precisely,
	\begin{align}
		\mathcal S^m(\mathbb R\times\mathbb R;\mathbb C)
		\triangleq
		\Bigg\{
		p\in C^\infty(\mathbb R\times\mathbb R;\mathbb C):
		|p|^{k,\ell;m}<\infty
		\text{ for all }k,\ell\in\mathbb N_0
		\Bigg\},
		\label{def Ss R}
	\end{align}
	where
	\[
	|p|^{k,\ell;m}
	\triangleq
	\max_{0\le j\le k,\,0\le r\le\ell}
	\sup_{(x,\xi)\in\mathbb R^2}
	\frac{|\partial_x^j\partial_\xi^r p(x,\xi)|}{(1+|\xi|)^{m-r}}.
	\]
	This is a Fr\'echet space with respect to the seminorms
	$\{|\cdot|^{k,\ell;m}\}_{k,\ell\in\mathbb N_0}$.
	
	The pseudo-differential operator associated with $p$ is defined, initially
	for $f\in\mathscr S(\mathbb R;\mathbb R)$, by
	\begin{equation}\label{OP define}
		[\mathrm{OP}(p)f](x)
		\triangleq
		\frac{1}{2\pi}\int_{\mathbb R}
		{\rm e}^{\mathrm ix\xi}p(x,\xi)[\mathscr Ff](\xi) \d\xi.
	\end{equation}
	We consider only operators that preserve real-valued functions. For the
	quantization in \eqref{OP define}, this is equivalent to
	\begin{equation}\label{Real symbol}
		p(x,-\xi)=\overline{p(x,\xi)}.
	\end{equation}
	We abbreviate
	\begin{equation*}
		\mathcal S^m
		\triangleq
		\left\{
		p\in\mathcal S^m(\mathbb R\times\mathbb R;\mathbb C):
		\eqref{Real symbol}\text{ holds}
		\right\},
		\quad
		\mathrm{OP}\mathcal S^m
		\triangleq
		\{\mathrm{OP}(p):p\in\mathcal S^m\}.
	\end{equation*}
	The corresponding Fourier-multiplier subclasses are
	\begin{equation*}
		\mathcal S_0^m
		\triangleq
		\{p\in\mathcal S^m:p(x,\xi)=p(\xi)\},
		\quad
		\mathrm{OP}\mathcal S_0^m
		\triangleq
		\{\mathrm{OP}(p):p\in\mathcal S_0^m\}.
	\end{equation*}
	A family in $\mathrm{OP}\mathcal S^m$ is called bounded if its symbols form
	a bounded set with respect to all seminorms in \eqref{def Ss R}. The same
	convention applies to $\mathrm{OP}\mathcal S_0^m$.
	
	We now introduce two classes of noise operators. In both classes, the
	principal part is skew-adjoint modulo an operator of order zero.
	
	\begin{Definition}\label{Ak class define}
		Let $\alpha\in[0,1]$. We define
		\[
		\mathbb A^\alpha
		\triangleq
		\left\{
		\mathcal A:\
		\text{there exist }\mathcal L\in\mathrm{OP}\mathcal S^\alpha
		\text{ and }\mathcal G\in\mathrm{OP}\mathcal S^0\text{ such that }
		\mathcal A=\mathcal L+\mathcal G
		\text{ and }
		\mathcal L+\mathcal L^*\in\mathrm{OP}\mathcal S^0
		\right\}.
		\]
		Such a pair $(\mathcal L,\mathcal G)$ is called an admissible decomposition
		of $\mathcal A$. A family $\mathscr N\subset\mathbb A^\alpha$ is called
		bounded if one can choose an admissible decomposition
		$\mathcal A=\mathcal L_{\mathcal A}+\mathcal G_{\mathcal A}$ for every
		$\mathcal A\in\mathscr N$ such that
		\[
		\{\mathcal L_{\mathcal A}:\mathcal A\in\mathscr N\}
		\text{ is bounded in }\mathrm{OP}\mathcal S^\alpha,
		\]
		and both
		\[
		\{\mathcal L_{\mathcal A}+\mathcal L_{\mathcal A}^*:
		\mathcal A\in\mathscr N\},
		\quad
		\{\mathcal G_{\mathcal A}:\mathcal A\in\mathscr N\}
		\]
		are bounded in $\mathrm{OP}\mathcal S^0$.
	\end{Definition}
	
	\begin{Definition}\label{Bk class define}
		Let $\beta\ge0$. We define
		\[
		\mathbb B^\beta
		\triangleq
		\left\{
		\mathcal B:\
		\text{there exist }\mathcal J\in\mathrm{OP}\mathcal S_0^\beta
		\text{ and }\mathcal V\in\mathrm{OP}\mathcal S_0^0\text{ such that }
		\mathcal B=\mathcal J+\mathcal V
		\text{ and }
		\mathcal J+\mathcal J^*\in\mathrm{OP}\mathcal S_0^0
		\right\}.
		\]
		Such a pair $(\mathcal J,\mathcal V)$ is called an admissible decomposition
		of $\mathcal B$. A family $\mathscr N\subset\mathbb B^\beta$ is called
		bounded if one can choose an admissible decomposition
		$\mathcal B=\mathcal J_{\mathcal B}+\mathcal V_{\mathcal B}$ for every
		$\mathcal B\in\mathscr N$ such that
		\[
		\{\mathcal J_{\mathcal B}:\mathcal B\in\mathscr N\}
		\text{ is bounded in }\mathrm{OP}\mathcal S_0^\beta,
		\]
		and both
		\[
		\{\mathcal J_{\mathcal B}+\mathcal J_{\mathcal B}^*:
		\mathcal B\in\mathscr N\},
		\quad
		\{\mathcal V_{\mathcal B}:\mathcal B\in\mathscr N\}
		\]
		are bounded in $\mathrm{OP}\mathcal S_0^0$.
	\end{Definition}
	
	\begin{Example}[Transport noise]\label{Example:transport-noise-1}
		If $\sigma\in\mathscr S(\mathbb R;\mathbb R)$, then
		\[
		\mathcal Q=\sigma(x)\partial_x\in\mathbb A^1.
		\]
		If $\sigma\in\mathbb R$ is constant, then
		\[
		\mathcal Q=\sigma\partial_x\in\mathbb B^1.
		\]
	\end{Example}
	
	\begin{proof}
		For $\sigma\in\mathscr S(\mathbb R;\mathbb R)$, the symbol of
		$\mathcal Q$ is $p(x,\xi)=\mathrm i\sigma(x)\xi$, which belongs to
		$\mathcal S^1$ and satisfies \eqref{Real symbol}. Moreover,
		\[
		(\sigma\partial_x)^*
		=
		-\sigma\partial_x-\sigma',
		\quad
		\mathcal Q+\mathcal Q^*=-\sigma'\in\mathrm{OP}\mathcal S^0.
		\]
		Thus $\mathcal Q\in\mathbb A^1$. If $\sigma$ is constant, then
		$\sigma\partial_x$ is a skew-adjoint Fourier multiplier of order one, and
		hence belongs to $\mathbb B^1$.
	\end{proof}
	
	\begin{Example}[Nonlocal noise amplitudes]
		Recall that $\D^s=(\I-\partial_x^2)^{s/2}$ for $s\in\R$.
		Let $\alpha\in[0,1]$ and $\beta\ge0$. If
		$\sigma\in\mathscr S(\mathbb R;\mathbb R)$, then
		\[
		\mathcal Q
		=
		\sigma(x)\partial_x\D^{\alpha-1}
		\in
		\mathbb A^\alpha.
		\]
		If $\sigma\in\mathbb R$, then
		\[
		\mathcal Q
		=
		\sigma\partial_x\D^{\beta-1}
		\in
		\mathbb B^\beta.
		\]
	\end{Example}
	
	\begin{proof}
		For
		$\mathcal Q=\sigma\partial_x\D^{\alpha-1}$ and
		$f,g\in\mathscr S(\mathbb R;\mathbb R)$,
		\begin{align*}
			\langle\mathcal Qf,g\rangle_{L^2}
			=-\langle f,\D^{\alpha-1}\partial_x(\sigma g)\rangle_{L^2} =\left\langle
			f,
			-\mathcal Qg
			+[\sigma,\D^{\alpha-1}]\partial_xg
			-\D^{\alpha-1}(\sigma'g)
			\right\rangle_{L^2}.
		\end{align*}
		Consequently,
		\[
		\mathcal Q^*+\mathcal Q
		=
		[\sigma,\D^{\alpha-1}]\partial_x-\D^{\alpha-1}\sigma'.
		\]
		Standard pseudo-differential calculus (see
		\cite{Taylor-1974-note,Benzoni-Gavage-Serre-2007-Book})  gives
		\[
		[\sigma,\D^{\alpha-1}]\in\mathrm{OP}\mathcal S^{\alpha-2},
		\quad
		\D^{\alpha-1}(\sigma'\cdot)\in\mathrm{OP}\mathcal S^{\alpha-1}.
		\]
		Since
		$\alpha\le1$, both terms in $\mathcal Q^*+\mathcal Q$ have order at most
		zero. Hence $\mathcal Q\in\mathbb A^\alpha$.
		
		If $\sigma\in\mathbb R$, then
		$\sigma\partial_x\D^{\beta-1}$ is a skew-adjoint Fourier
		multiplier of order $\beta$. It therefore belongs to $\mathbb B^\beta$.
	\end{proof}

	\subsection{Statement of the main results}
	
	We now formulate the structural assumption on the noise amplitude. Recall that the two classes of pseudo-differential operators $\mathbb{A}^\alpha$ and $\mathbb{B}^\beta$ are given by Definitions \ref{Ak class define} and \ref{Bk class define}, respectively.

	\begin{Hypothesis}\label{Hypo-Q}
		Let $\mathcal Q$ be a deterministic, real-preserving linear operator. Choose
		and fix one of the following admissible descriptions:
		\begin{enumerate}[label={\rm (\roman*)},leftmargin=0.79cm]
			\item $\mathcal Q=0$;
			\item there exists $\alpha\in[0,1]$ such that
			$\mathcal Q\in\mathbb A^\alpha$;
			\item there exists $\beta\ge0$ such that
			$\mathcal Q\in\mathbb B^\beta$.
		\end{enumerate}
		In cases \emph{(ii)} and \emph{(iii)}, the selected class and exponent are
		fixed throughout the paper, even if $\mathcal Q$ admits another admissible
		description.
	\end{Hypothesis}
	
	Corresponding to the description selected in Assumption~\ref{Hypo-Q}, set
	\begin{equation}\label{delta-Qi}
		\delta_{\mathcal Q}
		\triangleq
		\begin{cases}
			0,
			&\text{in case \emph{(i)}},\\[3pt]
			\alpha,
			&\text{in case \emph{(ii)}},\\[3pt]
			\beta,
			&\text{in case \emph{(iii)}}.
		\end{cases}
	\end{equation}
	Then, for every $r\in\mathbb R$,
	\[
	\mathcal Q\in\mathscr L(H^r;H^{r-\delta_{\mathcal Q}}),
	\quad
	\mathcal Q^2\in\mathscr L(H^r;H^{r-2\delta_{\mathcal Q}}).
	\]
	If
	\[
	s>\frac{3}{2}+(1\vee 2\delta_{\mathcal Q}),
	\]
	then $s-(1\vee 2\delta_{\mathcal Q})>3/2$ and hence
	\[
	H^{s-(1\vee 2\delta_{\mathcal Q})}
	\hookrightarrow
	C_b^1.
	\]
	Moreover, for $v\in H^s$,
	\[
	(\mathcal Hv)\partial_xv\in H^{s-1},
	\quad
	\mathcal Q^2v\in H^{s-2\delta_{\mathcal Q}},
	\quad
	\mathcal Qv\in H^{s-\delta_{\mathcal Q}}.
	\]
	Thus the It\^o equation \eqref{Target problem SCCF} is naturally interpreted
	in $H^{s-(1\vee 2\delta_{\mathcal Q})}$ and, by Sobolev embedding, also in
	$C_b^1(\mathbb R)$.
	
	\begin{Definition}\label{Def : solution}
		Let $s>\frac{3}{2}+\max\{2\delta_{\mathcal Q},1\}$.Let $\tau^*$ be a stopping time with $\mathbb P(\tau^*>0)=1$.
		An $H^s$-valued progressively measurable process
		\[
		v=(v(t))_{0\le t<\tau^*}
		\]
		is called a maximal $H^s$ classical solution to \eqref{Target problem SCCF} if the
		following conditions hold:
		
		\begin{itemize}[leftmargin=0.79cm]
			\setlength\itemsep{0.2em}
			
			\item  For a.e. $\omega\in\Omega$, and for every $T<\tau^*(\omega)$,
			$v\in C([0,T];H^s)$.
			
			\item $\pas$, for every $t\in[0,\tau^*)$,  \eqref{Target problem SCCF}
			holds as an identity in $C_b^1(\mathbb R)$.
			
			\item On the event $\{\tau^*<\infty\}$,
			$\limsup_{t\uparrow\tau^*}\|v(t)\|_{H^s}=\infty.$
			
		\end{itemize}
		
		If $\mathbb P(\tau^*=\infty)=1$, the maximal solution is called global.
	\end{Definition}

	Our first result establishes a local-in-time theory under
	Assumption~\ref{Hypo-Q}.
	
	\begin{Theorem}[Local-in-time theory]\label{Thm-local-solution}
		Under Assumption~\ref{Hypo-Q}, let
		\[
		s>\frac{3}{2}+\max\{2\delta_{\mathcal Q},1\}.
		\]
		For every $H^s$-valued $\mathcal F_0$-measurable random variable $v_0$,
		the Cauchy problem \eqref{Target problem SCCF} admits a pathwise unique maximal
		$H^s$ classical solution $(v,\tau^*)$ in the sense of
		Definition~\ref{Def : solution}.
		Moreover,  the maximal lifetime $\tau^*$ is independent of the Sobolev index $s$: if $v_0$
		belongs to two admissible Sobolev spaces, then the corresponding maximal
		lifetimes coincide; and
		the following blow-up criterion also holds:
		\begin{equation}\label{eq:blow-up criterion statement}
			\limsup_{t\uparrow\tau^*}
			\left(
			\|\partial_xv(t)\|_{L^\infty}
			+
			\|\Lambda v(t)\|_{L^\infty}
			\right)
			=\infty\quad
			\text{a.s. on}\ \left\{\tau^*<\infty\right\},
		\end{equation}
	\end{Theorem}

		\begin{Remark}[Methodological points in the local theory]
		\label{Remark-local-theory-method}
		The proof of Theorem~\ref{Thm-local-solution} follows a direct
		strong-solution route that differs in several respects from the
		martingale-compactness and variational constructions commonly used for
		stochastic fluid equations.
		
		\begin{enumerate}[label={\bf (\arabic*)},leftmargin=0.79cm]
			\setlength\itemsep{0.2em}

			\item \textbf{Failure of the Gelfand-triple and mild solution frameworks.}
			After conversion to It\^o form, the nonlinear drift, diffusion coefficient,
			and Stratonovich correction have the formal mapping properties
			\[
			(\mathcal Hu)\partial_xu\in H^{s-1},
			\qquad
			\mathcal Qu\in H^{s-\delta_{\mathcal Q}},
			\qquad
			\mathcal Q^2u\in H^{s-2\delta_{\mathcal Q}},
			\]
			although the solution itself is sought in $H^s$. Consequently, the
			coefficients are not invariant on the top-order state space whenever
			derivatives are lost. Moreover, 
			the essential noise contribution is the following cancellation
			mechanism (see \cite[Theorem~4.1 and Lemmas~4.1--4.2]
			{Karlsen-Tang-Wang-2026-arXiv}):
			\begin{equation}
				\left|
				\langle\mathcal Qf,f\rangle_{H^r}
				\right|
				\lesssim
				\|f\|_{H^r}^2,
				\quad
				\left|
				\langle\mathcal Q^2f,f\rangle_{H^r}
				+
				\|\mathcal Qf\|_{H^r}^2
				\right|
				\lesssim
				\|f\|_{H^r}^2,\quad f\in H^{r+2\delta_{\mathcal Q}}.\label{eq:cancel-theorem-remark}
			\end{equation}
			This means that though the operator $\tfrac{1}{2}\mathcal{Q}^2$ may be of parabolic type, the overall equation remains hyperbolic. Consequently, the available energy mechanism is not  the coercivity required by the standard variational  theory for monotone SPDEs on a Gelfand triple; see, for instance,
			\cite{Prevot-Rockner-2007-book}.  Similarly, the equation is not directly covered by classical semilinear
			mild-solution theory on a single Hilbert state space either; compare
			\cite{DaPrato-Zabczyk-2014-Book}.   More importantly, a construction carried out at
			a fixed Sobolev level in the fixed Gelfand-triple does \textit{not} by itself ensure consistency across
			different regularity indices. The independence of the maximal lifespan
			$\tau^*$ from the Sobolev index $s$ is therefore a separate
			cross-regularity consistency property that must be established. In this work, 
			the cut-off exit times coincides the maximal
			lifetimes (see Section \ref{Section: construction of the maximal solution}), which implies the the independence of $\tau^*$ from $s$.
			
			\item \textbf{Solution construction without compact embedding.}
			A common route in the stochastic-fluid literature is to construct
			regularised or Galerkin approximations, establish tightness of their laws,
			obtain a weak martingale solution through an almost-sure representation
			theorem, and then recover a pathwise solution by a
			Gy\"ongy--Krylov or Yamada--Watanabe-type argument. Representative
			examples include
			\cite{Breit-Feireisl-Hofmanova-2018-CPDE,
				Galimberti-etal-2024-JDE,GlattHoltz-Vicol-2014-AoP,Crisan-Flandoli-Holm-2019-JNLS,Alonso-Bethencourt-2020-JNLS,Alonso-Rohde-Tang-2021-JNLS}; in particular, \cite{Alonso-Rohde-Tang-2021-JNLS} develops an abstract framework for transport-noise-driven systems, while
			\cite{Galimberti-etal-2024-JDE} systematically employs tightness and
			almost-sure representations on suitable quasi-Polish path spaces.
			
			The present proof does not proceed through a martingale solution or a
			change of probability space. All approximate solutions are constructed on
			the original stochastic basis, with the same Brownian motion and initial
			data. Using  a renormalised version of \eqref{eq:cancel-theorem-remark} (see Lemma \ref{Lemma:Qn} below),
		one first proves uniform conditional $H^s$ estimates and then a
			stopped Cauchy estimate in $H^\theta$ with 
			$
			\frac32<\theta
			<
			s-\max\{2\delta_{\mathcal Q},1\}.
			$
			The stopping is subsequently removed by controlling the corresponding exit
			probabilities.  This direct construction is particularly convenient on the whole line.
			Indeed, for $s>\sigma$, the embedding
			$
			H^s(\mathbb R)\hookrightarrow H^\sigma(\mathbb R)
			$
			is continuous but not compact. Thus the tightness step available
			in many compact-domain arguments cannot be transferred directly to
			$\mathbb R$. Instead, the present proof obtains convergence directly
			in probability, and subsequently almost surely along a subsequence, in
			the lower Sobolev topology. Related regularisation and cancellation ideas 
			appear in
			\cite{Alonso-Rohde-Tang-2021-JNLS,Tang-Wang-2022-arXiv,Alonso-Miao-Tang-2022-JDE,
				Tang-2023-JFA}.
			
		 		\item \textbf{Initial data without an imposed moment condition.}
		 The approximate equation is first solved with a deterministic parameter
		 $a\in H^s$, while joint measurability of the parameterised solution map is
		 retained. The parameter is then evaluated at
		 $
		 a=v_0(\omega).
		 $
		 The relevant \textit{a priori} and Cauchy estimates are formulated
		 conditionally with respect to $\mathcal F_0$. In this way, $v_0$ is treated
		 as a frozen parameter, and the construction applies to every
		 $H^s$-valued, $\mathcal F_0$-measurable random variable without requiring,
		 for example, 	$\mathbb E\|v_0\|_{H^s}^p<\infty$
		 for any prescribed $p>0$.  The right-hand sides of the conditional
		 estimates are $\mathcal F_0$-measurable quantities depending on the
		 realised initial data and are finite almost surely. Thus no integrability
		 assumption on the law of $v_0$ is needed in order to close the local
		 construction.

		\end{enumerate}

	\end{Remark}

	We next specialise to scalar transport noise. Consider
	\[
	\d v+(\mathcal Hv)\partial_xv \d t
	=
	\sigma(x)\partial_xv\circ{\rm d}W(t),
	\quad
	v(0)=v_0.
	\]
	
	\begin{Hypothesis}\label{Hypo-sigma}
		The coefficient in \eqref{eq:sccf-transport} satisfies
		$
		\sigma\in\mathscr S(\mathbb R;\mathbb R).
		$
	\end{Hypothesis}
	
	Under Assumption~\ref{Hypo-sigma}, Example~\ref{Example:transport-noise-1}
	shows that
	\[
	\mathcal Q=\sigma(x)\partial_x\in\mathbb A^1.
	\]
	Thus $\delta_{\mathcal Q}=1$, and Theorem~\ref{Thm-local-solution} applies
	whenever $s>7/2$. Let $(v,\tau^*)$ denote the resulting maximal solution.
	For $t<\tau^*$, consider the stochastic characteristic flow
		\begin{equation}\label{eq:char-flow}
			\d\phi(t,x)
		=
		(\mathcal Hv)(t,\phi(t,x)) \d t
		-
		\sigma(\phi(t,x))\circ{\rm d}W(t),
		\quad
		\phi(0,x)=x.
	\end{equation}
	If $x_0$ is a global maximum point of $v_0$, then
		\begin{equation}\label{eq:z=phi(x0)}
		z_0(t)\triangleq\phi(t,x_0),
		\quad
		0\le t<\tau^*
	\end{equation}
	remains a global maximum point of $v(t,\cdot)$ for $0\le t<\tau^*$.
	The relevant steepness quantity is
		\begin{equation}\label{eq:M(t) define}
		M(t)\triangleq\Lambda v(t,z_0(t)),
		\quad
		0\le t<\tau^*.
	\end{equation}

\begin{Theorem}[High-probability finite-time breakdown under transport noise]
	\label{Thm-blowup}
	Suppose that Assumption~\ref{Hypo-sigma} holds. Let $s>7/2$, and let
	$v_0\in H^s$ be deterministic. Suppose that $v_0$ has a global maximum
	point $x_0\in\mathbb R$, that is,
	\[
	v_0(x_0)=\sup_{x\in\mathbb R}v_0(x).
	\]
	Let $(v,\tau^*)$ be the maximal $H^s$ solution to
	\eqref{eq:sccf-transport}. Then, for every
	$\varepsilon\in(0,1)$, there exists a constant
	\[
	C^*
	=
	C^*\bigl(
	\sigma,\|v_0\|_{L^\infty},\varepsilon
	\bigr)>0
	\]
	such that, if
	$\Lambda v_0(x_0)>C^*$,
	then
	\[
	\mathbb P\left(
	\tau^*
	\le
	\frac{8}{\Lambda v_0(x_0)}
	\right)
	\ge
	1-\varepsilon.
	\]
\end{Theorem}
\begin{Remark}[Nonlocal compression versus transport-induced rearrangement]
	\label{Remark-blowup-novelty}
	Theorem~\ref{Thm-blowup} and its proof highlight the following four
	structural features of the stochastic CCF dynamics.
	
	\begin{enumerate}[label={\bf (\arabic*)},leftmargin=0.79cm]
		\setlength\itemsep{0.2em}
		
		\item \textbf{Dominance of nonlocal compression in the steep-data
			regime.}
		For every prescribed confidence level $1-\epsilon$, Theorem
		~\ref{Thm-blowup} shows that, if the initial nonlocal steepness
	$\Lambda v_0(x_0)$
		exceeds a corresponding threshold, then
		\[
		\mathbb P\left(
		\tau^*\le\frac8{\Lambda v_0(x_0)}
		\right)
		\ge1-\epsilon.
		\]
		Thus, in the steep-data regime, breakdown occurs on the time scale
		$\frac{1}{\Lambda v_0(x_0)}$ with arbitrarily high prescribed probability.  This gives a
		rigorous sense in which the nonlocal characteristic-compression
		mechanism dominates the stochastic transport-induced rearrangement:
		on the high-probability event supplied by the theorem, the latter is
		not strong enough to prevent the Riccati-type growth of the nonlocal
		strain.
		
		This conclusion complements the finite-time singularity result for
		linear It\^o multiplicative noise in
		\cite{Alonso-Miao-Tang-2022-JDE}, but the mechanism is different.  In
		the present Stratonovich transport setting, scalar values are
		preserved exactly along the stochastic characteristic flow:
		$v(t,\phi(t,x))=v_0(x)$ and	
		$\|v(t)\|_{L^\infty}
		=
		\|v_0\|_{L^\infty},
	$ for
		$0\le t<\tau^*$.
		The singularity therefore cannot be attributed to a direct stochastic
		amplification of the scalar amplitude.  Rather, the proof forces a
		loss of regularity through nonlocal compression detected along the
		transported maximum, while the scalar range remains unchanged.
		
		\item \textbf{Local noise becomes nonlocal at the level of the
			fractional observable.}
		A recurrent strategy in stochastic Burgers- and
		Camassa--Holm-type wave-breaking arguments is to track an extremal
		local steepness variable, often
		$
		\inf_{x\in\mathbb R}\partial_xv(t,x),
		$
		or an equivalent Lagrangian slope observable; see, for example,
		\cite{Alonso-etal-2019-NODEA,Crisan-Holm-2018-PhyD,
			Rohde-Tang-2021-NoDEA}.
		
		The observable in the stochastic CCF equation is fundamentally
		different:
		\[
		M(t)=\Lambda v(t,z_0(t)),
		\]
		where $z_0(t)$ is the stochastic image of an initial global maximum of
		$v_0$.  Although
		$v_x(t,z_0(t))=0$, 
		the point $z_0(t)$ is not selected as an extremum of $\Lambda v$.
		Consequently, one cannot invoke extremum-point identities or sign
		conditions for the spatial derivatives of $\Lambda v$ in order to
		close the evolution of $M$.
		
		Moreover, although the original noise contains only the classical
		local derivative $\sigma(x)\partial_xv$, applying $\Lambda$ produces
		the fractional commutator
		\[
		[\Lambda,\sigma]\partial_xv.
		\]
		Thus, a local transport perturbation becomes nonlocal at the level of
		the observable that drives the singularity.
		
		\item \textbf{Extraction of the singular principal part and the
			coupled field--characteristic correction.}
		To analyse
		\[
	[\Lambda,\sigma]v_x(t,z_0(t)),
		\]
		we separate the leading singular part of the commutator kernel and
		control the regular remainder using
		Lemma~\ref{Lem:sigma-kernel}.  This gives
		\[
		[\Lambda,\sigma]v_x(t,z_0(t))
		=
		\sigma'(z_0(t))M(t)
		+
		\mathcal R_\sigma[v](t,z_0(t)),
		\]
		with
		\[
		\left|
		\mathcal R_\sigma[v](t,z_0(t))
		\right|
		\le
		C_\sigma\|v_0\|_{L^\infty}.
		\]
		The leading contribution is therefore extracted explicitly as
		$\sigma'(z_0)M$, whereas the genuinely nonlocal remainder is controlled
		by the exactly conserved scalar amplitude.
		
		The stochastic correction is substantially more delicate than a
		pointwise scalar Stratonovich--It\^o conversion.  First, a localised
		It\^o--Wentzell formula is used to evaluate the stochastic field
		$\Lambda v$ along the random point $z_0(t)$.  To identify the bracket
		of the resulting noise coefficient, the remainder is represented as
		an integral whose kernel depends on $z_0(t)$.  Since the
		infinite-dimensional field $v(t,\cdot)$ and the characteristic
		$z_0(t)$ are driven by the same Brownian motion, the moving kernel and
		the field have a nontrivial quadratic covariation.
		
		Applying It\^o's formula to the characteristic-dependent kernel,
		followed by the product It\^o formula, kernel integration by parts, and
		stochastic Fubini, identifies both the second-order correction
		generated by the motion of $z_0$ and the cross-variation between the
		moving kernel and the stochastic field.  The resulting equation,
		summarised in Lemma~\ref{lem:dM lemma}, is
		\[
		\mathrm dM(t)
		=
		\left[
		\frac12M(t)^2+\mathscr R_M(t)
		\right]\mathrm dt
		+
		\Gamma(t)\,\mathrm dW(t),
		\]
		where
		\[
		\mathscr R_M(t)
		\ge
		-c_1M(t)-c_2\|v_0\|_{L^\infty},
		\qquad
		|\Gamma(t)|
		\le
		c_\sigma
		\left(
		M(t)+\|v_0\|_{L^\infty}
		\right).
		\]
		It is this coupled field--characteristic calculation that converts the
		original local transport perturbation into estimates compatible with
		the nonlocal Riccati mechanism.
		
		\item \textbf{Reciprocal reduction and the removal of non-uniform upper
			localisation.}
		Even after the fractional commutator and the coupled
		Stratonovich--It\^o correction have been controlled, a direct
		one-level martingale comparison for $M$ encounters a further
		probabilistic obstruction.  Since the noise coefficient $\Gamma$ satisfies only a
		linear-growth estimate, such an argument naturally requires an upper
		localisation $M<N$.  The corresponding quadratic-variation bound
		contains the factor
		\[
		\left(
		N+\|v_0\|_{L^\infty}
		\right)^2,
		\]
		so the resulting probability estimate deteriorates as
		$N\to\infty$; see
		Remark~\ref{Remark:Y=1/M necessary}.  A direct one-level estimate can
		therefore yield a high-probability passage to a fixed upper level, but
		it cannot be passed directly to arbitrarily large levels.
		
		The strict positivity of $M$ permits the reciprocal transformation
		\[
		Y(t)=\frac1{M(t)},
		\qquad
		Y_0=\frac1{M_0}.
		\]
		This maps the entire large-$M$ regime to the small-$Y$ boundary
		region and converts the stochastic Riccati growth of $M$ into a
		one-sided boundary comparison for $Y$.  On the stopped region
		\[
		0<Y(t)<2Y_0,
		\]
		the smallness of $Y_0$ allows the lower-order drift terms to be
		absorbed into the negative Riccati contribution, while the stopped
		martingale coefficient is uniformly of order $Y_0$.  A single
		high-probability estimate then shows that, if the classical solution
		were to survive until
		\[
		8Y_0=\frac8{M_0},
		\]
		the resulting one-sided inequality would force $Y$ to become negative,
		contradicting its strict positivity.
		
		The reciprocal reduction is therefore not merely a convenient change
		of variables.  It incorporates all large-$M$ scales into a single
		stopped estimate and avoids both an iteration over successive upper
		levels and the non-uniform upper localisation that obstructs the
		direct argument.
		
	\end{enumerate}
\end{Remark}

We also have a conditional description of the
	blow-up rate along the transported maximum.  The condition below is stated
	on the event on which this particular observable is responsible for the
	breakdown; the blow-up criterion \eqref{eq:blow-up criterion statement} alone does not
	imply that $M(t)$ diverges. 
	
	We introduce
		\begin{equation}
		\mathcal E[f](z^*)
		\triangleq
		\left\|
		\frac{f(z^*)-f(\cdot)}{z^*-\cdot}
		\right\|_{\dot H^{1/2}}^2
		\ge0,\quad \quad \|g\|_{\dot H^{1/2}}\triangleq \|\Lambda^{1/2}g\|_{L^2}.
		\label{eq:E-functional}
	\end{equation}
	
	\begin{Theorem}[Conditional blow-up rate along the transported maximum]
		\label{Thm-blowup-rate}
		Suppose that Assumption~\ref{Hypo-sigma} holds. Let $s>7/2$, let
		$v_0\in H^s$ be deterministic, and suppose that $x_0$ is a global
		maximum point of $v_0$ such that
		\[
		M_0=\Lambda v_0(x_0)>0.
		\]
		Let $(v,\tau^*)$ be as in Theorem \ref{Thm-blowup}. Let $z_0$, $M$ and $\mathcal E[f](z^*)$ be given in \eqref{eq:z=phi(x0)}, \eqref{eq:M(t) define} and \eqref{eq:E-functional}, respectively.
		Define
		\begin{equation}\label{eq:Omega-M-rate-main}
			\Omega_M
			\triangleq
			\left\{
			\tau^*<\infty,
			\ \lim_{t\uparrow\tau^*}M(t)=\infty
			\right\}.
		\end{equation}
		On $\Omega_M$, set
		\begin{equation}\label{eq:eta-rate-main}
			\eta(t)
			\triangleq
			\frac{\mathcal E[v(t,\cdot)](z_0(t))}{M(t)^2},
			\qquad
			\overline\eta(t)
			\triangleq
			\frac1{\tau^*-t}
			\int_t^{\tau^*}\eta(r)\d r.
		\end{equation}
		Here $\mathcal E$ is the nonnegative functional defined in
		\eqref{eq:E-functional}.  Then the following assertions hold almost surely.
		\begin{enumerate}[label={\rm (\roman*)},leftmargin=0.79cm]
			\item On $\Omega_M$,
			\begin{equation}\label{eq:type-I-rate-main}
				\limsup_{t\uparrow\tau^*}
				M(t)(\tau^*-t)
				\le2.
			\end{equation}
			\item If, on a measurable subevent of $\Omega_M$,
			\begin{equation}\label{eq:eta-minus-main}
				\eta_{\mathrm{av}}^-
				\triangleq
				\liminf_{t\uparrow\tau^*}\overline\eta(t)>0,
			\end{equation}
			then, on that subevent,
			\begin{equation}\label{eq:strict-type-I-rate-main}
				\limsup_{t\uparrow\tau^*}
				M(t)(\tau^*-t)
				\le
				\frac1{\frac12+\eta_{\mathrm{av}}^-}
				<2.
			\end{equation}
			Here the first upper bound is interpreted as zero if
			$\eta_{\mathrm{av}}^-=+\infty$.
			\item If, on a measurable  subevent of $\Omega_M$, there exists a finite
			nonnegative random variable $\eta_*\in[0,\infty)$ such that  
			\begin{equation}\label{eq:eta-star-main}
				\lim_{t\uparrow\tau^*}\overline\eta(t)=\eta_*,
			\end{equation}
			then, on that subevent,
			\begin{equation}\label{eq:exact-rate-main}
				\lim_{t\uparrow\tau^*}
				M(t)(\tau^*-t)
				=
				\frac1{\frac12+\eta_*}.
			\end{equation}
		\end{enumerate}
	\end{Theorem}
	
\begin{Remark}[Interpretation of the conditional blow-up-rate theorem]
	\label{Remark-blowup-rate-novelty}
	The restriction to the event $\Omega_M$ is essential.  The general
	blow-up criterion \eqref{eq:blow-up criterion statement} does not imply
	that the particular observable
$
	M(t)=\Lambda v(t,z_0(t))
$
	diverges along the transported maximum.  Theorem~\ref{Thm-blowup-rate}
	therefore describes the first-order temporal blow-up profile of $M$
	within this conditional regime.   	
	The proof reveals the following structural mechanisms.
	
	\begin{enumerate}[label={\bf (\arabic*)},leftmargin=0.79cm]
		\setlength\itemsep{0.2em}
		
		\item \textbf{A conditional Riccati envelope and structural rigidity.}
		On $\Omega_M$, estimate \eqref{eq:type-I-rate-main} gives
		$
		\limsup_{t\uparrow\tau^*}
		M(t)(\tau^*-t)
		\le2.
		$
		The constant $2$ is precisely the blow-up coefficient of the
		Riccati equation
		\[
		M'=\frac12M^2.
		\]
		Thus, whenever the transported-maximum observable is responsible for
		the singularity, its first-order divergence remains inside the
		deterministic Riccati envelope.  This exhibits a form of structural
		rigidity: although the transport noise generates stochastic field
		oscillations and moves the point at which $M$ is evaluated, it does not
		produce a leading-order normalised prefactor larger than the Riccati
		value $2$.
		
		In this precise sense,the explicit Stratonovich transport noise neither
		regularizes the solution nor yields a first-order temporal profile more
		singular than that of the bare Riccati mechanism.
		
		\item \textbf{Terminal integrating factor and asymptotic filtering of
			the explicit stochastic oscillations.}
		Starting from the reciprocal equation for $Y=1/M$ (see
		\eqref{eq:Y-rate-linear-form}), the remaining difficulty is to extract
		an exact coefficient near the random terminal time.  Although the
		explicit stochastic coefficients are bounded in the terminal
		large-$M$ regime, the drift proportional to $Y$, the martingale term
		proportional to $Y$, and their quadratic-covariation contribution
		cannot be discarded separately without losing the exact first-order
		balance.
		
		To this end, we chose a terminal integrating factor $\mathscr Z$ (see 
		\eqref{eq:Z-rate-definition})  so that these terms cancel
		exactly in the product It\^o formula (see \eqref{eq:ZY-rate-differential}):
		\begin{equation*}
			\d\bigl(\mathscr Z(t)Y(t)\bigr)
			=
			-\mathscr Z(t)
			\left(
			\frac12+\eta(t)
			\right)\d t.
		\end{equation*}
	 This cancellation yields the
		terminal representation
		\[
		Y(t)
		=
		\int_t^{\tau^*}
		\frac{\mathscr Z(r)}{\mathscr Z(t)}
		\left(
		\frac12+\eta(r)
		\right)\d r\quad \text{where}\quad  \sup_{r\in[t,\tau^*]}
		\left|
		\frac{\mathscr Z(r)}{\mathscr Z(t)}-1
		\right|
		\longrightarrow0
		\qquad
		\text{as }t\uparrow\tau^*.
		\]
		The Brownian oscillations are therefore  
		absorbed into a pathwise multiplicative factor which becomes
		asymptotically flat on the shrinking terminal interval.  Consequently,
		the explicit stochastic integrating-factor contribution does not
		survive as a separate term in the first-order coefficient.  This is the
		precise sense in which the explicit multiplicative stochastic terms
		are asymptotically neutral at leading order.
		
		\item \textbf{The random nonlocal-energy footprint and the selection
			of the blow-up profile.}
		The quantity
		\[
		\eta(t)
		=
		\frac{
			\mathcal E[v(t,\cdot)](z_0(t))
		}{
			M(t)^2
		}
		\ge0
		\]
		records how the global scalar profile contributes to the compression
		observed at the transported maximum.  Although $\mathcal E$ is a
		deterministic functional of a given profile, the processes $v$ and
		$z_0$ are noise dependent; hence $\eta$ is a random record of the
		nonlocal geometry accumulated along the stochastic evolution.
		
		The theorem can therefore be read as a profile-selection principle:
		
		\begin{itemize}[leftmargin=0.79cm]
			\setlength\itemsep{0.2em}
			\item 	If
			$
			\overline\eta(t)\longrightarrow0,
			$
			then the normalised nonlocal energy is asymptotically negligible on
			the terminal scale, and the Riccati-dominated profile is recovered:
			\[
			M(t)(\tau^*-t)\longrightarrow2.
			\]
			\item 	If instead
			$
			\overline\eta(t)\longrightarrow\eta_*>0,
			$
			then a nontrivial amount of normalised nonlocal energy persists up to
			the singularity, and the first-order profile is modified to
			\[
			M(t)(\tau^*-t)
			\longrightarrow
			\frac1{\frac12+\eta_*}
			<2.
			\]
				In this regime, the global profile contributes an additional
			compressive component.  It raises the effective averaged Riccati
			coefficient from $\frac12$ to $\frac12+\eta_*$ and consequently lowers
			the normalised terminal prefactor.
			
			\item More generally, if only
			$
			\eta_{\mathrm{av}}^-
			=
			\liminf_{t\uparrow\tau^*}\overline\eta(t)>0
			$
			is known, then the nonlocal energy still forces the strict improvement
			in \eqref{eq:strict-type-I-rate-main}, but possible oscillations of
			$\overline\eta(t)$ may prevent the selection of a single asymptotic
			coefficient.  Conversely, the weaker condition
		$
			\liminf_{t\uparrow\tau^*}\overline\eta(t)=0
		$
			does not imply convergence to the coefficient $2$. 	Thus, randomness alters the exact first-order blow-up profile not
			through a surviving stochastic kick at the terminal scale, but
			\emph{implicitly}, by reshaping the global scalar profile before the
			singularity forms.  The terminal Ces\`aro behaviour of $\eta$ records
			the cumulative effect of this random profile deformation and, when it
			converges, selects the resulting first-order coefficient.
		\end{itemize}

	\end{enumerate}
\end{Remark}
	
	\subsection{Plan of the paper}
	
	The remainder of the paper is organised as follows.
	Section~\ref{Section:Preliminaries} collects the Friedrichs-mollifier
	estimates, the cancellation properties of the pseudo-differential noise,
	the deterministic energy estimate for the CCF nonlinearity, and the
	maximum-point and kernel identities used later. In
	Section~\ref{Section:LocalTheory}, we construct the unique maximal classical
	solution by a cut-off Friedrichs approximation, prove convergence in a lower
	Sobolev topology, recover continuity at the top Sobolev order, and establish
	the blow-up criterion. Section~\ref{Section:BlowUp} specialises to
	Stratonovich transport noise. There we construct the stochastic
	characteristic flow, follow an initial maximum, derive the evolution of
	$\Lambda v$ along that characteristic, including the full
	Stratonovich--It\^o correction, and prove Theorem~\ref{Thm-blowup} through
	the reciprocal process.  We then use a terminal stochastic integrating
	factor to prove the conditional rate statement in
	Theorem~\ref{Thm-blowup-rate}.

	\section{Preliminaries}
	\label{Section:Preliminaries}
	
	Recall the operator $\mathrm{OP}$ defined in \eqref{OP define}. For $n\ge1$,
	we define the Friedrichs mollifier $J_n$ by
	\begin{equation}
		J_n \triangleq \mathrm{OP}\bigl(j(\cdot/n)\bigr),
		\quad n\ge1,
		\label{Define Jn}
	\end{equation}
	where $j\in\mathscr S(\mathbb R;\mathbb R)$ is even and satisfies
	\[
	0\le j(\xi)\le1,
	\quad
	j(\xi)=1 \quad\text{for } |\xi|\le1.
	\]
	In particular, $J_n$ preserves real-valued functions.
	
	\begin{Lemma}[Standard properties of Friedrichs mollifiers,
		see \cite{Li-Liu-Tang-2021-SPA,Taylor-2011-PDEbook3}]
		\label{Lemma-Jn}
		Let $J_n$ be defined by \eqref{Define Jn}. Then the following properties
		hold.
		\begin{enumerate}[label={\bf (\arabic*)},leftmargin=0.79cm]
			\setlength\itemsep{0.2em}
			\item For every $r\in\mathbb R$, $J_n$ is self-adjoint on $H^r$:
			\[
			\langle J_nf,g\rangle_{H^r}
			=
			\langle f,J_ng\rangle_{H^r},
			\quad f,g\in H^r.
			\]
			Moreover, $J_n$ commutes with every Fourier multiplier; in particular,
			it commutes with $\mathcal D^r$, $\partial_x$, and $\mathcal H$ whenever
			the corresponding compositions are defined.
			
			\item For every $r\in\mathbb R$ and $f\in H^r$,
			\[
			\sup_{n\ge1}\|J_n\|_{\mathscr L(H^r;H^r)}\le1,
			\quad
			\lim_{n\to\infty}\|J_nf-f\|_{H^r}=0.
			\]
			In addition,
			\[
			\sup_{n\ge1}
			\|J_n\|_{\mathscr L(L^\infty;L^\infty)}<\infty.
			\]
			
			\item For every $a,b\in\mathbb R$ with $a>b$,
			\[
			\|\mathbf I-J_n\|_{\mathscr L(H^a;H^b)}
			\lesssim
			n^{-(a-b)},
			\]
			whereas, for $a<b$,
			\[
			\|J_n\|_{\mathscr L(H^a;H^b)}
			\lesssim
			n^{b-a}.
			\]
			Consequently, for each fixed $n\ge1$, $J_n$ is a smoothing operator:
			$J_n\in\mathscr L(H^a;H^b)$ for all $a,b\in\mathbb R$.
		\end{enumerate}
	\end{Lemma}
	
	The local existence proof relies on a regularising approximation of
	$\mathcal Q$. The following result is the scalar one-dimensional specialisation of \cite[Theorem~4.1 and Lemmas~4.1--4.2]
	{Karlsen-Tang-Wang-2026-arXiv}.
	
	\begin{Lemma}\label{Lemma:Qn}
		Let Assumption~\ref{Hypo-Q} hold, and let $\delta_{\mathcal Q}$ be defined
		by \eqref{delta-Qi}. For $n\ge1$, set
		\[
		\mathcal Q_n\triangleq J_n\mathcal QJ_n.
		\]
		Then the following properties hold.
		\begin{enumerate}[label={\bf (\arabic*)},leftmargin=0.79cm]
			\setlength\itemsep{0.2em}
			\item For each $n\ge1$,
			$\mathcal Q_n\in\mathrm{OP}\mathcal S^{-\infty}$.
			If the description selected in Assumption~\ref{Hypo-Q} is
			$\mathcal Q\in\mathbb A^\alpha$, then
			$\{\mathcal Q_n\}_{n\ge1}$ is bounded in $\mathbb A^\alpha$.
			If the selected description is
			$\mathcal Q\in\mathbb B^\beta$, then
			$\{\mathcal Q_n\}_{n\ge1}$ is bounded in $\mathbb B^\beta$.
			
			\item For every $r\ge0$,
			\[
			\lim_{n\to\infty}
			\|\mathcal Q_nf-\mathcal Qf\|_{H^r}=0,
			\quad f\in H^{r+\delta_{\mathcal Q}},
			\]
			and
			\[
			\lim_{n\to\infty}
			\|\mathcal Q_n^2f-\mathcal Q^2f\|_{H^r}=0,
			\quad f\in H^{r+2\delta_{\mathcal Q}}.
			\]
			
			\item Let $r,\theta\ge0$ and
			$r>\theta+\delta_{\mathcal Q}$. Then
			\[
			\|\mathcal Q_n-\mathcal Q_m\|_{
				\mathscr L(H^r;H^\theta)}
			\lesssim
			(n\wedge m)^{-(r-\theta-\delta_{\mathcal Q})}.
			\]
			Moreover, for all $f,g\in H^r$,
			\begin{align*}
				\left|
				\left\langle
				\mathcal Q_nf-\mathcal Q_mg,
				f-g
				\right\rangle_{H^\theta}
				\right| \lesssim
				(n\wedge m)^{-2(r-\theta-\delta_{\mathcal Q})}
				\bigl(\|f\|_{H^r}+\|g\|_{H^r}\bigr)^2
				+
				\|f-g\|_{H^\theta}^2.
			\end{align*}
			The implicit constants are independent of $n,m$.
			
			\item Let $r,\theta\ge0$ and
			$r>\theta+2\delta_{\mathcal Q}$. Then
			\[
			\|\mathcal Q_n^2-\mathcal Q_m^2\|_{
				\mathscr L(H^r;H^\theta)}
			\lesssim
			(n\wedge m)^{-(r-\theta-2\delta_{\mathcal Q})}.
			\]
			Moreover, for all $f,g\in H^r$,
			\begin{align*}
				\left|
				\left\langle
				\mathcal Q_n^2f-\mathcal Q_m^2g,
				f-g
				\right\rangle_{H^\theta}
				+
				\|\mathcal Q_nf-\mathcal Q_mg\|_{H^\theta}^2
				\right| \lesssim
				(n\wedge m)^{-(r-\theta-\delta_{\mathcal Q})}
				\bigl(\|f\|_{H^r}+\|g\|_{H^r}\bigr)^2
				+
				\|f-g\|_{H^\theta}^2.
			\end{align*}
			The implicit constants are independent of $n,m$.
			
			\item For every $r\ge0$ and $f\in H^r$,
			\begin{equation}
				\sup_{n\ge1}
				\left|
				\langle\mathcal Q_nf,f\rangle_{H^r}
				\right|
				\lesssim
				\|f\|_{H^r}^2,
				\quad
				\sup_{n\ge1}
				\left|
				\langle\mathcal Q_n^2f,f\rangle_{H^r}
				+
				\|\mathcal Q_nf\|_{H^r}^2
				\right|
				\lesssim
				\|f\|_{H^r}^2.
				\label{Qn cancellation-1}
			\end{equation}
			
			\item For every $r\ge0$ and $f\in H^r$,
			\begin{equation}
				\sup_{n\ge1}
				\left|
				\langle J_n\mathcal Qf,J_n f\rangle_{H^r}
				\right|
				\lesssim
				\|f\|_{H^r}^2,
				\quad
				\sup_{n\ge1}
				\left|
				\langle J_n\mathcal Q^2f,J_n f\rangle_{H^r}
				+
				\|J_n\mathcal Qf\|_{H^r}^2
				\right|
				\lesssim
				\|f\|_{H^r}^2.
				\label{Qn cancellation-2}
			\end{equation}
		\end{enumerate}
	\end{Lemma}
	
	\begin{proof}
		Properties \textbf{(1)}--\textbf{(5)} follow from
		\cite[Theorem~4.1 and Lemmas~4.1--4.2]
		{Karlsen-Tang-Wang-2026-arXiv}.
		To prove \textbf{(6)}, fix $r\ge0$ and set
		$\mathcal P_n\triangleq\mathcal D^rJ_n$. The family
		$\{\mathcal P_n\}_{n\ge1}$ is bounded in
		$\mathrm{OP}\mathcal S^r$ and, in the Fourier-multiplier case, in
		$\mathrm{OP}\mathcal S_0^r$. Applying
		\cite[Theorem~4.1]{Karlsen-Tang-Wang-2026-arXiv} to $\mathcal P_n$ gives the
		claimed estimates first for sufficiently regular $f$. Since $J_n$ is
		smoothing, all displayed expressions are well defined for $f\in H^r$;
		the estimates then extend to $H^r$ by density.
	\end{proof}
	
	\begin{Remark}
		The two uniform cancellation estimates \eqref{Qn cancellation-1} and \eqref{Qn cancellation-2} serve different purposes.
		First, recalling the Stratonovich--It\^o conversion
		\[
		\mathcal Q_nv_n\circ\mathrm dW(t)
		=
		\mathcal Q_nv_n\,\mathrm dW(t)
		+
		\frac{1}{2}\mathcal Q_n^2v_n\,\mathrm dt,
		\]
		It\^o's formula for $\|v_n\|_{H^s}^2$ produces the noise contributions
		\begin{equation*}
			2\langle\mathcal Q_nv_n,v_n\rangle_{H^s}\,\mathrm dW(t)
			+
			\left(
			\langle\mathcal Q_n^2v_n,v_n\rangle_{H^s}
			+
			\|\mathcal Q_nv_n\|_{H^s}^2
			\right)\mathrm dt.
		\end{equation*}
		Thus, \eqref{Qn cancellation-1} is the cancellation mechanism used to
		close the uniform $H^s$ estimates for the approximate solutions and to
		control the corresponding difference equations.
		
		Second, after the limiting solution $v$ has been constructed in a lower
		Sobolev topology, applying $J_n$ to the limiting It\^o equation and
		then applying It\^o's formula to $\|J_n v\|_{H^s}^2$ produces
		\[
		2\langle J_n\mathcal Qv,J_n v\rangle_{H^s}\,\mathrm dW(t)
		\quad \text{and}\quad
		\left(
		\langle J_n\mathcal Q^2v,J_n v\rangle_{H^s}
		+
		\|J_n\mathcal Qv\|_{H^s}^2
		\right)\mathrm dt.
		\]
		Estimate \eqref{Qn cancellation-2} provides bounds for these terms that
		are uniform in $n$. These uniform bounds are one ingredient in the
		proof of the continuity of $t\mapsto \|v(t)\|_{H^s}^2$.
		
		For transport-type noise involving classical derivatives, cancellation
		estimates in integer-order Sobolev spaces can be found in
		\cite{Crisan-Flandoli-Holm-2019-JNLS,Alonso-Bethencourt-2020-JNLS}, while
		fractional-order estimates are first treated in
		\cite{Alonso-Rohde-Tang-2021-JNLS}. Related cancellation estimates for
		pseudo-differential noise were first established in \cite{Tang-Wang-2022-arXiv} and then developed in
		\cite{Tang-2023-JFA,Karlsen-Tang-Wang-2026-arXiv,Alonso-Pang-Tang-2026-JLMS}.
	\end{Remark}
	
	We recall the following Kato--Ponce estimates.
	
\begin{Lemma}[Kato--Ponce estimates;
	see \cite{Kenig-Ponce-Vega-1991-JAMS}]
	\label{KP-commutator}
	Let $s>0$. For sufficiently smooth $f$ and $g$,
	\[
	\|[\mathcal D^s,f\mathbf I]g\|_{L^2}
	\lesssim_s
	\|\partial_xf\|_{L^\infty}
	\|\mathcal D^{s-1}g\|_{L^2}
	+
	\|\mathcal D^sf\|_{L^2}
	\|g\|_{L^\infty},
	\]
	and
	\[
	\|\mathcal D^s(fg)\|_{L^2}
	\lesssim_s
	\|f\|_{L^\infty}
	\|\mathcal D^sg\|_{L^2}
	+
	\|\mathcal D^sf\|_{L^2}
	\|g\|_{L^\infty}.
	\]
\end{Lemma}
	
	The next estimate is the deterministic energy estimate for the CCF
	nonlinearity. It follows from the Kato--Ponce commutator estimate in
	Lemma \ref{KP-commutator} together with the one-dimensional Hilbert-transform structure.  We refer to \cite[Lemmas~2.3 and~3.1]{Alonso-Miao-Tang-2022-JDE} for the detailed proof.
	
	\begin{Lemma}\label{Lemma : Huux}
		Let $s>5/2$. Let  $J_n$ be the Friedrichs mollifier defined in \eqref{Define Jn}.
		Then we have
		\begin{equation*}
			\left|\IP{\mathcal{H} u\partial_x u, u}_{H^{s}}\right|
			\lesssim  (\|u_x\|_{L^\infty}+\|\mathcal{H} u_x\|_{L^\infty})\|u\|^2_{H^{s}},\quad u\in H^{s+1}.
		\end{equation*}
		and
		\begin{equation*}
			\sup_{n\ge1}\left|\bIP{J_n (\mathcal{H} u\partial_x u), J_n u}_{H^{s}}\right|
			\lesssim  (\|u_x\|_{L^\infty}+\|\mathcal{H} u_x\|_{L^\infty})\|u\|^2_{H^{s}},\quad u\in H^{s}.
		\end{equation*}
		Moreover, for all $u,v\in H^s$, all $n,m\ge1$, and every
		$\theta\in(3/2,s-1)$,
		\begin{align*}
			\left|\IP{J_n \left[\mathcal{H} J_n u\partial_x J_nu\right]-J_m \left[\mathcal{H} J_m v\partial_x J_mv\right],u-v}_{H^\theta} \right|
			\lesssim  \Big(1+\|u\|^4_{H^s}+\|v\|^4_{H^s}\Big)
			\left((n\wedge m)^{-2(s-1-\theta)}+\|u-v\|^2_{H^\theta}\right).
		\end{align*}
	\end{Lemma}

	We recall the following identity from \cite[Proposition~3.5]{Silvestre-Vicol-2016-TAMS}.
	
	\begin{Lemma}\label{lem:identity}
		Let $f\in H^s(\mathbb R)$ with $s>\frac72$, and assume that $f$ attains a
		global maximum at $z^*\in\mathbb R$. Then
		\begin{equation*}
			\Lambda\bigl((\mathcal Hf)(\Lambda\mathcal Hf)\bigr)(z^*)
			+
			(\mathcal Hf)(z^*)\,\mathcal Hf_{xx}(z^*)
			=
			-\frac12\bigl(\Lambda f(z^*)\bigr)^2
			-
			\mathcal E[f](z^*),
		\end{equation*}
		where, as in \eqref{eq:E-functional},
		\begin{equation*}
			\mathcal E[f](z^*)=
			\left\|
			\frac{f(z^*)-f(\cdot)}{z^*-\cdot}
			\right\|_{\dot H^{1/2}}^2
			\ge0,\quad \quad \|g\|_{\dot H^{1/2}}= \|\Lambda^{1/2}g\|_{L^2}.
		\end{equation*}
	\end{Lemma}
	
	\begin{Remark}
		With the convention $\Lambda=-\mathcal H\partial_x$, one has
		\[
		\Lambda\mathcal H=\partial_x,
		\quad
		\Lambda f_x=-\mathcal Hf_{xx}.
		\]
		Therefore, at a maximum point $z^*$,
		\[
		-\Lambda\bigl((\mathcal Hf)f_x\bigr)(z^*)
		+
		(\mathcal Hf)(z^*)\Lambda f_x(z^*)
		=
		\frac12\bigl(\Lambda f(z^*)\bigr)^2
		+
		\mathcal E[f](z^*).
		\]
		This is exactly the sign used in the blow-up proof.
	\end{Remark}

	The following lemma is used to control the lower-order
	part of the commutator $[\Lambda,\sigma]$.
	\begin{Lemma}\label{Lem:sigma-kernel}
		Let $\sigma\in\mathscr S(\mathbb R;\mathbb R)$, and define
		\[
		\mathsf K_\sigma(z,y)
		\triangleq
		\frac{\sigma(z)-\sigma(y)}{(z-y)^2}
		-
		\frac{\sigma'(z)}{z-y}
		=
		\frac{
			\sigma(z)-\sigma(y)-\sigma'(z)(z-y)
		}{
			(z-y)^2
		},
		\qquad z\neq y.
		\]
		Then $\mathsf K_\sigma$ extends uniquely to a function in
		$C^\infty(\mathbb R^2)$, still denoted by $\mathsf K_\sigma$, and
		\[
		\mathsf K_\sigma(z,z)
		=
		-\frac12\sigma''(z),
		\qquad z\in\mathbb R.
		\]
		Moreover, the following estimates hold:
		\begin{equation}\label{eq:kernel-L2-bound}
			\sup_{z\in\mathbb R}
			\left\|
			\mathsf K_\sigma(z,\cdot)
			\right\|_{L_y^2}
			\le
			\frac43
			\left\|\sigma''\right\|_{L^2},
		\end{equation}
		\begin{equation}\label{eq:general-kernel-L1-bound}
			\sup_{z\in\mathbb R}
			\left\|
			\partial_z^a\partial_y^b
			\mathsf K_\sigma(z,\cdot)
			\right\|_{L_y^1}
			\le
			\left(
			\int_0^1
			(1-\vartheta)^{a+1}
			\vartheta^{b-1}
			\d \vartheta
			\right)
			\left\|
			\sigma^{(a+b+2)}
			\right\|_{L^1},\quad a\in\mathbb N_0,\ b\in\mathbb{N},
		\end{equation}
		and
		\begin{align}
			\sup_{z\in\mathbb R}
			\left\|
			\partial_y
			\left[
			\sigma(\cdot)
			\partial_y\mathsf K_\sigma(z,\cdot)
			\right]
			\right\|_{L_y^1} \le
			\frac12
			\|\sigma'\|_{L^\infty}
			\|\sigma^{(3)}\|_{L^1}
			+
			\frac16
			\|\sigma\|_{L^\infty}
			\|\sigma^{(4)}\|_{L^1}.
			\label{eq:weighted-kernel-derivative-L1-bound}
		\end{align}
	\end{Lemma}
	
	\begin{proof}
		Taylor's formula with integral remainder gives
		\begin{align*}
			\sigma(y)
			=
			\sigma(z)
			+
			\sigma'(z)(y-z) +
			(y-z)^2
			\int_0^1
			(1-\vartheta)
			\sigma''
			\bigl(
			(1-\vartheta)z+\vartheta y
			\bigr)
			\d \vartheta.
		\end{align*}
		Since $(z-y)^2=(y-z)^2$, it follows that, for $z\neq y$,
		\begin{equation}\label{eq:kernel-integral-representation}
			\mathsf K_\sigma(z,y)
			=
			-\int_0^1
			(1-\vartheta)
			\sigma''
			\bigl(
			(1-\vartheta)z+\vartheta y
			\bigr)
			\d \vartheta.
		\end{equation}
		The right-hand side is well defined for all
		$(z,y)\in\mathbb R^2$ and defines a smooth function. In particular,
		on the diagonal,
		\[
		\mathsf K_\sigma(z,z)
		=
		-\sigma''(z)
		\int_0^1(1-\vartheta)\,\d\vartheta
		=
		-\frac12\sigma''(z).
		\]
		This proves the existence of the smooth extension. Its uniqueness
		follows from continuity, since
		$\{(z,y)\in\mathbb R^2:z\neq y\}$ is dense in $\mathbb R^2$.
		
		We first prove \eqref{eq:kernel-L2-bound}. By Minkowski's integral
		inequality and \eqref{eq:kernel-integral-representation},
		\begin{align*}
			\left\|\mathsf K_\sigma(z,\cdot)\right\|_{L_y^2}
			\le
			\int_0^1
			(1-\vartheta)
			\left\|
			\sigma''
			\bigl(
			(1-\vartheta)z+\vartheta\,\cdot
			\bigr)
			\right\|_{L_y^2}
			\d \vartheta.
		\end{align*}
		For $\vartheta\in(0,1]$, the change of variables
		\[
		w=(1-\vartheta)z+\vartheta y
		\]
		gives
		\[
		\left\|
		\sigma''
		\bigl(
		(1-\vartheta)z+\vartheta\,\cdot
		\bigr)
		\right\|_{L_y^2}
		=
		\vartheta^{-1/2}
		\|\sigma''\|_{L^2}.
		\]
		Hence
		\begin{align*}
			\left\|\mathsf K_\sigma(z,\cdot)\right\|_{L_y^2}
			&\le
			\left(
			\int_0^1
			(1-\vartheta)\vartheta^{-1/2}
			\d \vartheta
			\right)
			\|\sigma''\|_{L^2} =
			\frac43\|\sigma''\|_{L^2}.
		\end{align*}
		The right-hand side is independent of $z$, proving
		\eqref{eq:kernel-L2-bound}.
		
		Next, differentiating \eqref{eq:kernel-integral-representation} under
		the integral sign yields, for every $a,b\in\mathbb N_0$,
		\begin{equation}\label{eq:general-kernel-derivative}
			\partial_z^a\partial_y^b\mathsf K_\sigma(z,y)
			=
			-\int_0^1
			(1-\vartheta)^{a+1}
			\vartheta^b
			\sigma^{(a+b+2)}
			\bigl(
			(1-\vartheta)z+\vartheta y
			\bigr)
			\d \vartheta.
		\end{equation}
		
		Let now $b\ge1$. Applying Minkowski's integral inequality to
		\eqref{eq:general-kernel-derivative} and using the same change of
		variables gives
		\begin{align*}
			\left\|
			\partial_z^a\partial_y^b
			\mathsf K_\sigma(z,\cdot)
			\right\|_{L_y^1} \le
			\int_0^1
			(1-\vartheta)^{a+1}
			\vartheta^b
			\left\|
			\sigma^{(a+b+2)}
			\bigl(
			(1-\vartheta)z+\vartheta\,\cdot
			\bigr)
			\right\|_{L_y^1}
			\d \vartheta =
			\left(
			\int_0^1
			(1-\vartheta)^{a+1}
			\vartheta^{b-1}
			\d \vartheta
			\right)
			\left\|
			\sigma^{(a+b+2)}
			\right\|_{L^1}.
		\end{align*}
		Again, the right-hand side is independent of $z$, which proves
		\eqref{eq:general-kernel-L1-bound}.
		
		Finally, the product rule gives
		\[
		\partial_y
		\left[
		\sigma(y)\partial_y\mathsf K_\sigma(z,y)
		\right]
		=
		\sigma'(y)\partial_y\mathsf K_\sigma(z,y)
		+
		\sigma(y)\partial_y^2\mathsf K_\sigma(z,y).
		\]
		Therefore,
		\begin{align*}
			\left\|
			\partial_y
			\left[
			\sigma(\cdot)\partial_y\mathsf K_\sigma(z,\cdot)
			\right]
			\right\|_{L_y^1} \le
			\|\sigma'\|_{L^\infty}
			\|\partial_y\mathsf K_\sigma(z,\cdot)\|_{L_y^1}
			+
			\|\sigma\|_{L^\infty}
			\|\partial_y^2\mathsf K_\sigma(z,\cdot)\|_{L_y^1}.
		\end{align*}
		Taking $(a,b)=(0,1)$ and $(a,b)=(0,2)$ in
		\eqref{eq:general-kernel-L1-bound}, respectively, yields
		\[
		\sup_{z\in\mathbb R}
		\|\partial_y\mathsf K_\sigma(z,\cdot)\|_{L_y^1}
		\le
		\frac12\|\sigma^{(3)}\|_{L^1},
		\quad
		\sup_{z\in\mathbb R}
		\|\partial_y^2\mathsf K_\sigma(z,\cdot)\|_{L_y^1}
		\le
		\frac16\|\sigma^{(4)}\|_{L^1}.
		\]
		Combining the preceding three estimates proves
		\eqref{eq:weighted-kernel-derivative-L1-bound}.
	\end{proof}
	
	\begin{Lemma}[Fractional Sobolev--Morrey estimate]
		\label{lem:fractional-Sobolev-Morrey-time}
		Let $T>0$, $1<p<\infty$, and
		\[
		0<r<1,
		\qquad
		r p>1.
		\]
		Let $f\in L^p(0,T)$ satisfy
		\[
		[f]_{W^{r,p}(0,T)}^p
		\triangleq
		\int_0^T\int_0^T
		\frac{
			|f(t)-f(t')|^p
		}{
			|t-t'|^{1+r p}
		}
		\,\mathrm dt'\,\mathrm dt
		<\infty.
		\]
		Then $f$ admits a representative, still denoted by $f$, belonging to
		\[
		C^{0,r-\frac1p}([0,T]).
		\]
		Moreover,
		\begin{equation*} 
			|f(t)-f(t')|^p
			\le
			C_{r,p}
			|t-t'|^{r p-1}
			[f]_{W^{r,p}(0,T)}^p,
			\qquad
			t,t'\in[0,T].
		\end{equation*}
	\end{Lemma}
	
	\begin{proof}
		Set
		\[
		\overline g\triangleq\int_0^1g(x)\,\d x,\quad
		g(x)\triangleq f(Tx),
		\qquad x\in(0,1).
		\]
		By Jensen's inequality,
		\begin{align*}
			\|g-\overline g\|_{L^p(0,1)}^p
			=
			\int_0^1
			\left|
			\int_0^1\bigl(g(x)-g(y)\bigr)\,\d y
			\right|^p
			\d x \le
			\int_0^1\int_0^1
			|g(x)-g(y)|^p\,\d y\,\d x \le
			[g]_{W^{r,p}(0,1)}^p,
		\end{align*}
		where the last inequality follows from $|x-y|\le1$.
		
		Applying
		\cite[Theorem~8.2]{DiNezza-Palatucci-Valdinoci-2012-BSM}
		to $g-\overline g$ on $(0,1)$ gives a representative satisfying
		\[
		|g(x)-g(y)|^p
		\le
		C_{r,p}|x-y|^{rp-1}
		[g]_{W^{r,p}(0,1)}^p.
		\]
		Moreover,
		\[
		[g]_{W^{r,p}(0,1)}^p
		=
		T^{rp-1}[f]_{W^{r,p}(0,T)}^p.
		\]
		Taking $x=t/T$ and $y=t'/T$ therefore yields
		\[
		|f(t)-f(t')|^p
		\le
		C_{r,p}|t-t'|^{rp-1}
		[f]_{W^{r,p}(0,T)}^p.
		\]
		The constant is independent of $T$.
	\end{proof}

	\section{Local-in-time theory: proof of Theorem~\ref{Thm-local-solution}}\label{Section:LocalTheory}
	In this section we provide a local-in-time theory for \eqref{Target problem SCCF} and prove Theorem \ref{Thm-local-solution}. For convenience, we restate \eqref{Target problem SCCF} as follows
	\begin{equation*}
		v(t)-v_0
		+
		\int_0^t
		\left[
		\mathcal Hv\,\partial_xv
		-
		\frac{1}{2}\mathcal Q^2v
		\right](t')\d t'
		=
		\int_0^t \mathcal Qv(t')\d W(t'),
		\quad
		v(0)=v_0.
	\end{equation*}
	Equivalently,
	\begin{equation}\label{Target problem SCCF differential}
		\d v
		=
		\Big[
		-\mathcal Hv\,\partial_xv
		+
		\frac{1}{2}\mathcal Q^2v
		\Big]\d t
		+
		\mathcal Qv\d W(t),\quad
		v(0)=v_0.
	\end{equation}
	
	The proof below draws upon the recent work \cite{Tang-Wang-2022-arXiv,Alonso-Pang-Tang-2026-JLMS}. Given that some arguments in the proof will also be used in subsequent sections, and the drift and the noise are
	singular at the $H^s$ level and the required cancellation must be preserved
	uniformly under regularisation, we opt to write down the comprehensive proof to ensure both completeness and consistency.

	\subsection{Approximation scheme}
	
	Let $\chi\in C^\infty([0,\infty);[0,1])$ such that
	$\chi=1$ on $[0,1]$ and $\chi=0$ on $[2,\infty)$. For any $R\ge1$, define $\chi_R\in C^\infty([0,\infty);[0,1])$ as
	\[
	\chi_R(r)\triangleq\chi(r/R).
	\]
	Then it is clear that $\|\chi'_R\|_{L^\infty}=\frac{1}{R}\|\chi'\|_{L^\infty}\leq \|\chi'\|_{L^\infty}$.
	
	Let $\{\mathcal Q_n\}_{n\ge1}$ be the regularising approximation of
	$\mathcal Q$ provided by Lemma~\ref{Lemma:Qn}. Let $\{J_n\}_{n\ge1}$ be
	the mollifier defined in \eqref{Define Jn}. For $n,R\ge1$, we approximate
	\eqref{Target problem SCCF differential} by  the following Stratonovich
	problem
	\begin{equation*}
		\d v+
		\chi_R\Big(
		\|\partial_xv-\partial_xv_0\|_{L^\infty}
		+
		\|\mathcal H\partial_xv-\mathcal H\partial_xv_0\|_{L^\infty}
		\Big)
		J_n\big[
		(\mathcal HJ_nv)\,\partial_xJ_nv
		]\d t =
		\mathcal Q_nv\circ{\rm d}W(t),
		\quad
		v(0)=v_0.
	\end{equation*}
	Equivalently,
	\begin{equation}\label{approximation scheme C}
		v(t)-v_0
		=
		\int_0^t
		\mathscr V_{n,R}\bigl(v_0,v(t')\bigr)\d t'
		+
		\int_0^t \mathcal Q_nv(t')\d W(t'),
	\end{equation}
	where, for $a,u\in H^s$,
	\begin{align*}
		\mathscr V_{n,R}(a,u)
		\triangleq{}&-
		\chi_R\Big(
		\|\partial_xu-\partial_xa\|_{L^\infty}
		+
		\|\mathcal H\partial_xu-\mathcal H\partial_xa\|_{L^\infty}
		\Big)
		J_n\big[
		(\mathcal HJ_nu)\,\partial_xJ_nu
		\big]
		+
		\frac12\mathcal Q_n^2u.
	\end{align*}
	
	\begin{Lemma}\label{Lemma: existence of XnR}
		Let the assumptions of Theorem~\ref{Thm-local-solution} hold. For every $n,R\ge1$ and every $H^s$-valued $\mathcal F_0$-measurable random variable $v_0$, the Cauchy problem \eqref{approximation scheme C} admits a unique global solution
		\[
		v_n^{(R)}\in C([0,\infty);H^s)\qquad\pas.
		\]
		Moreover, there is a constant $C>0$, independent of $n$, $R$ and $T$, such that, for every $R\ge1$ and $T>0$,
		\begin{align}
			\sup_{n\ge1}
			\E\left[
			\sup_{t\in[0,T]}
			\|v_n^{(R)}(t)\|_{H^s}^2
			\Big|\F_0
			\right]
			\le
			2\|v_0\|_{H^s}^2
			\exp\left\{
			C
			\Big(
			1+2R
			+\|\partial_xv_0\|_{L^\infty}
			+\|\mathcal H\partial_xv_0\|_{L^\infty}
			\Big)T
			\right\}.
			\label{Xn uniform bound}
		\end{align}
	\end{Lemma}
	
	\begin{proof}
		Fix $n,R\ge1$. Since
		$s>\frac32+\max\{2\delta_{\mathcal Q},1\}\ge\frac52$, the embeddings
		$H^s\hookrightarrow W^{1,\infty}$ and the boundedness of $\mathcal H$ on
		$H^s$ show that the maps
		\[
		(a,u)\longmapsto
		\|\partial_xu-\partial_xa\|_{L^\infty},
		\qquad
		(a,u)\longmapsto
		\|\mathcal H\partial_xu-\mathcal H\partial_xa\|_{L^\infty}
		\]
		are continuous on $H^s\times H^s$. The operators $J_n$ and
		$\mathcal Q_n$ are smoothing and bounded on $H^s$, and hence
		$\mathscr V_{n,R}:H^s\times H^s\to H^s$ is jointly continuous. For each
		fixed parameter $a\in H^s$, the map
		$u\mapsto\mathscr V_{n,R}(a,u)$ is locally Lipschitz on $H^s$.
		Moreover, on the support of the cut-off,
		\[
		\|\partial_xu\|_{L^\infty}
		+
		\|\mathcal H\partial_xu\|_{L^\infty}
		\le
		2R
		+
		\|\partial_xa\|_{L^\infty}
		+
		\|\mathcal H\partial_xa\|_{L^\infty},
		\]
		so that
		\begin{equation*} 
			\|\mathscr V_{n,R}(a,u)\|_{H^s}
			\le
			C_{n,R,a}\bigl(1+\|u\|_{H^s}\bigr).
		\end{equation*}
		The diffusion map $u\mapsto\mathcal Q_nu$ is bounded and linear on
		$H^s$.
		
		We first freeze the parameter. For every deterministic $a\in H^s$,
		standard Hilbert-space SDE theory (see, for instance, \cite{Gawarecki-Mandrekar-2011-Book,Kallianpur-Xiong-1995-book}) therefore gives a pathwise unique global
		strong solution $U_{n,R}^{a}\in C([0,\infty);H^s)$ to
		\begin{equation}\label{eq:parameterised-approximate-SDE}
			U_{n,R}^{a}(t)-a
			=
			\int_0^t
			\mathscr V_{n,R}\bigl(a,U_{n,R}^{a}(t')\bigr)\,\mathrm{d} t'
			+
			\int_0^t\mathcal Q_nU_{n,R}^{a}(t')\d W(t').
		\end{equation}  
		Using the Picard construction for the truncated coefficients, one may
		choose a non-anticipative Borel solution map
		\[
		\Phi_{n,R}:H^s\times C([0,\infty);\mathbb R)
		\longrightarrow
		C([0,\infty);H^s)
		\]
		such that
		\[
		U_{n,R}^a=\Phi_{n,R}(a,W)
		\]
		simultaneously for all $a\in H^s$ outside a single null set.
		Consequently,
		\[
		v_n^{(R)}=\Phi_{n,R}(v_0,W)
		\]
		is progressively measurable and solves
		\eqref{approximation scheme C}.
		In particular, one obtains a version such that, for every
		$T>0$,
		\[
		(a,\omega)
		\longmapsto
		U_{n,R}^{a}(\cdot,\omega)
		\]
		is $\mathcal B(H^s)\otimes\mathcal F_T$-measurable as a map into
		$C([0,T];H^s)$, and its restriction to $[0,t]$ is
		$\mathcal B(H^s)\otimes\mathcal F_t$-measurable for every $t\le T$.
		
		We may therefore substitute the $\mathcal F_0$-measurable parameter
		$a=v_0(\omega)$ and define
		\begin{equation*} 
			v_n^{(R)}(t,\omega)
			\triangleq
			U_{n,R}^{v_0(\omega)}(t,\omega).
		\end{equation*}
		The preceding joint measurability shows that $v_n^{(R)}$ is adapted and
		progressively measurable, and it has continuous $H^s$-valued paths.
		Because $W$ is a Brownian motion relative to
		$\{\mathcal F_t\}_{t\ge0}$, its future increments are independent of
		$\mathcal F_0$. Hence, under the conditional probability given
		$\mathcal F_0$, the value $a=v_0$ is frozen and the driving process remains
		Brownian. Equation \eqref{eq:parameterised-approximate-SDE} therefore
		yields \eqref{approximation scheme C}. More explicitly, applying any
		deterministic-parameter estimate to $U_{n,R}^{a}$, multiplying by
		$\mathbf 1_A$ with $A\in\mathcal F_0$, and then substituting $a=v_0$
		gives the corresponding conditional estimate by the defining property of
		conditional expectation. Pathwise uniqueness for the parameterised
		equation \eqref{eq:parameterised-approximate-SDE} also gives pathwise uniqueness for the random initial data.
		Consistency in $T$ yields the asserted the global existence.
		
		Although $v_n^{(R)}$ exists globally and has continuous $H^s$-valued paths,
		the conditional integrability of the martingale supremum is not known
		before the desired conditional second-moment estimate has been
		proved. We therefore need a localisation argument. For every integer $L\ge2$,
		define
		\begin{equation}\label{tau-L-nR-definition}
			\tau_L^{n,R}
			\triangleq
			L\wedge
			\inf\bigl\{
			t\ge0:\|v_n^{(R)}(t)\|_{H^s}\ge L
			\bigr\}.
		\end{equation}
		Since $v_n^{(R)}$ is adapted and has continuous $H^s$-valued paths,
		$\tau_L^{n,R}$ is a bounded stopping time. Moreover, for fixed $n$ and
		$R$,
		\begin{equation}\label{tau-L-nR-limit}
			\tau_L^{n,R}\uparrow\infty
			\qquad\text{almost surely as }L\to\infty.
		\end{equation}
		Indeed, every continuous path is bounded on each compact time interval.

		Fix $n,R\ge1$.
		Applying It\^o's formula to the stopped process gives,
		for every $t\ge0$,
		\begin{align}
			\|v_n^{(R)}(t\wedge\tau_L^{n,R})\|_{H^s}^2
			=\ &
			\|v_0\|_{H^s}^2
			+
			\int_0^{t\wedge\tau_L^{n,R}}
			\left[
			2\IP{\mathscr V_{n,R}(v_0,v_n^{(R)}(t')),v_n^{(R)}(t')}_{H^s}
			+
			\|\mathcal Q_nv_n^{(R)}(t')\|_{H^s}^2
			\right]\d t'
			\notag\\
			&+
			2\int_0^{t\wedge\tau_L^{n,R}}
			\IP{
				\mathcal Q_nv_n^{(R)}(t'),v_n^{(R)}(t')
			}_{H^s}
			\d W(t')
			\label{eq:Ito energy approximate}
		\end{align}
		
		We first estimate the drift integrand.
		Using the properties of $J_n$ in Lemma \ref{Lemma-Jn} and  \eqref{Qn cancellation-1},  we obtain
		\begin{align}
			&\left|2\IP{\mathscr V_{n,R}(v_0,v_n^{(R)}(t')),v_n^{(R)}(t')}_{H^s}+\|\mathcal Q_nv_n^{(R)}(t')\|_{H^s}^2\right|\notag\\
			\lesssim\ & (1+2R
			+\|\partial_xv_0\|_{L^\infty}
			+\|\mathcal H\partial_xv_0\|_{L^\infty})\|v_n^{(R)}(t')\|^2_{H^s}+\left|\IP{\mathcal Q_n^2v_n^{(R)}(t'),v_n^{(R)}(t')}_{H^s}+\|\mathcal Q_nv_n^{(R)}(t')\|_{H^s}^2\right|\notag\\
			\leq\ & C(1+2R
			+\|\partial_xv_0\|_{L^\infty}
			+\|\mathcal H\partial_xv_0\|_{L^\infty})\|v_n^{(R)}(t')\|^2_{H^s},\quad n\ge1,\label{Ito-1}
		\end{align}
		and
		\begin{equation}
			\label{Ito-2}
			\left| \IP{\mathcal Q_nv_n^{(R)}(t'),v_n^{(R)}(t')}_{H^s} \right|^2 \leq
			C \|v_n^{(R)}(t')\|_{H^s}^4,\quad n\ge1,
		\end{equation}
		where the constant $C$ is independent of $n$.

		The stopped stochastic integral in \eqref{eq:Ito energy approximate} is a square-integrable
		martingale.  Therefore, for $T>0$, the conditional Burkholder--Davis--Gundy (BDG) inequality and \eqref{Ito-2} yield
		\begin{align*}
			\E\left[
			\sup_{0\le t\le T}
			\bigg|2\int_0^{t\wedge\tau_L^{n,R}}
			\IP{
				\mathcal Q_nv_n^{(R)}(t'),v_n^{(R)}(t')
			}_{H^s}
			\d W(t')\bigg|
			\bigg|\mathcal F_0
			\right]
			\leq
			C\E\left[
			\left(
			\int_0^{T\wedge\tau_L^{n,R}}
			\|v_n^{(R)}(t')\|_{H^s}^4
			\d t'
			\right)^{1/2}
			\bigg|\mathcal F_0
			\right].
		\end{align*}
		For completeness, define the stopped martingale
		\[
		\mathfrak M_{n,R,L}(t)
		\triangleq
		2\int_0^{t\wedge\tau_L^{n,R}}
		\left\langle
		\mathcal Q_nv_n^{(R)}(r),v_n^{(R)}(r)
		\right\rangle_{H^s}
		\d W(r).
		\]
		For every $A\in\mathcal F_0$, the process
		$\mathbf 1_A\mathfrak M_{n,R,L}$ is again a square-integrable martingale.
		Applying the ordinary BDG inequality to this process and then using the
		defining property of conditional expectation gives the conditional BDG
		inequality used above.
		
		Pathwise,
		\begin{align*}
			&
			\left(
			\int_0^{T\wedge\tau_L^{n,R}}
			\|v_n^{(R)}(t')\|_{H^s}^4
			\d t'
			\right)^{1/2}
			\leq
			\left(
			\sup_{0\le t\le T}
			\|v_n^{(R)}(t\wedge\tau_L^{n,R})\|_{H^s}^2
			\right)^{1/2}
			\left(
			\int_0^{T\wedge\tau_L^{n,R}}
			\|v_n^{(R)}(t')\|_{H^s}^2
			\d t'
			\right)^{1/2}.
		\end{align*}
		Consequently, Young's inequality implies
		\begin{align}
			&\E\left[
			\sup_{0\le t\le T}
			\bigg|2\int_0^{t\wedge\tau_L^{n,R}}
			\IP{
				\mathcal Q_nv_n^{(R)}(t'),v_n^{(R)}(t')
			}_{H^s}
			\d W(t')\bigg|
			\bigg|\mathcal F_0
			\right]\notag\\
			\leq \ &
			\frac{1}{2}\E\left[\sup_{0\le t\le T}
			\|v_n^{(R)}(t\wedge\tau_L^{n,R})\|_{H^s}^2\Big|\mathcal F_0\right]+
			C\E\left[
			\int_0^{T\wedge\tau_L^{n,R}}
			\|v_n^{(R)}(t')\|_{H^s}^2
			\d t'
			\bigg|\mathcal F_0
			\right].\label{stopped-BDG-estimate}
		\end{align}

		Combining \eqref{eq:Ito energy approximate}, \eqref{Ito-1} and
		\eqref{stopped-BDG-estimate}, using the conditional Tonelli's theorem, and absorbing the first term on the right-hand side, we obtain
		\begin{align*}
			&\E\left[
			\sup_{t\in[0,T]}
			\|v_n^{(R)}(t\wedge\tau_L^{n,R})\|_{H^s}^2
			\Big|\F_0
			\right]
			-
			2\|v_0\|_{H^s}^2\\
			\le\ &
			C
			\Big(
			1+2R
			+\|\partial_xv_0\|_{L^\infty}
			+\|\mathcal H\partial_xv_0\|_{L^\infty}
			\Big)
			\int_0^T
			\E\left[
			\sup_{t'\in[0,t]}
			\|v_n^{(R)}(t'\wedge\tau_L^{n,R})\|_{H^s}^2
			\Big|\F_0
			\right]\d t.
		\end{align*}
		Since $ 1+2R
		+\|\partial_xv_0\|_{L^\infty}
		+\|\mathcal H\partial_xv_0\|_{L^\infty}$ is $\mathcal F_0$-measurable, Gronwall's
		inequality can be applied pathwise to this conditional estimate.
		Therefore,
		\begin{equation}\label{Xn uniform bound L}
			\E\left[
			\sup_{t\in[0,T]}
			\|v_n^{(R)}(t\wedge\tau_L^{n,R})\|_{H^s}^2
			\Big|\F_0
			\right]
			\le
			2\|v_0\|_{H^s}^2
			\exp\left\{
			C
			\Big(
			1+2R
			+\|\partial_xv_0\|_{L^\infty}
			+\|\mathcal H\partial_xv_0\|_{L^\infty}
			\Big)T
			\right\},
			\quad T>0,
		\end{equation}
		where the constant $C$ is independent of $n$, $R$, $L$, and $T$.
		
		Finally, fix $n$ and $R$. By \eqref{tau-L-nR-limit},
		\begin{equation*}
			\sup_{0\le t\le T}
			\|v_n^{(R)}(t\wedge\tau_L^{n,R})\|_{H^s}^2
		\to
			\sup_{0\le t\le T}
			\|v_n^{(R)}(t)\|_{H^s}^2\quad \pas
		\end{equation*}
		The conditional monotone convergence theorem applied to
		\eqref{Xn uniform bound L} gives \eqref{Xn uniform bound}.
	\end{proof}
	
	\subsection{Cauchy argument on approximate solutions}
	
	\begin{Lemma}\label{lem:Cauchy argument}
		Let the assumptions of Theorem~\ref{Thm-local-solution} hold and let
		\[
		\theta\in\left(\frac32,\,
		s-\max\{2\delta_{\mathcal Q},1\}\right).
		\]
		Let $v_n^{(R)}$ be the approximate solution  obtained in Lemma \ref{Lemma: existence of XnR}.
		Then the following properties hold almost surely:
		\begin{enumerate}[label={\bf (\arabic*)},leftmargin=0.79cm]\setlength\itemsep{0.2em}

			\item\label{Convergence of X-n TnmN}
			Let
			\[
			\tau_N^{n,m}(R)
			\triangleq
			N
			\wedge
			\inf\left\{
			t\ge0:
			\|v_n^{(R)}(t)\|_{H^s}
			\vee
			\|v_m^{(R)}(t)\|_{H^s}
			\ge N
			\right\},
			\quad n,m\ge1,\quad N\ge1.
			\]
			Then, for every $N>0$,
			\begin{equation}
				\label{Cauchy in E-M}
				\lim_{n\to\infty}
				\sup_{m\ge n}
				\E\left[
				\sup_{t\in[0,\tau_N^{n,m}(R)]}
				\|v_n^{(R)}(t)-v_m^{(R)}(t)\|_{H^\theta}^2
				\Big|\F_0
				\right]
				=
				0.
			\end{equation}
			
			\item\label{Convergence of X-n T}
			There exists an $\mathcal F_t$-progressively measurable $H^s$-valued process
			$v^{(R)}=(v^{(R)}(t))_{t\ge0}$
			satisfying, for every $T>0$,
			\begin{equation}
				\label{X Hs bound}
				\E\left[
				\sup_{t\in[0,T]}
				\|v^{(R)}(t)\|_{H^s}^2
				\Big|\F_0
				\right]
				\le
				2\|v_0\|_{H^s}^2
				\exp\left\{
				C
				\Big(
				1+2R
				+\|\partial_xv_0\|_{L^\infty}
				+\|\mathcal H\partial_xv_0\|_{L^\infty}
				\Big)T
				\right\}.
			\end{equation}
			Moreover, for a subsequence of $\{v_n^{(R)}\}_{n\ge1}$, still denoted by
			$\{v_n^{(R)}\}_{n\ge1}$, and for every $s'\in[\theta,s)$,
			\begin{equation}
				\label{Xn to X}
				\lim_{n\to\infty}
				\sup_{t\in[0,T]}
				\|v_n^{(R)}(t)-v^{(R)}(t)\|_{H^{s'}}
				=
				0,
				\quad
				T>0.
			\end{equation}
		\end{enumerate}
	\end{Lemma}
	
	\begin{proof}
		We divide the proof into two parts.

		\textit{\underline{Proof of \ref{Convergence of X-n TnmN}}.}
		For $n,m\ge1$, set
		\[
		Z_{n,m}^{(R)}(t)
		\triangleq
		v_n^{(R)}(t)-v_m^{(R)}(t).
		\]
		Subtracting the two approximate equations, we obtain
		\begin{equation*}
			Z_{n,m}^{(R)}(t)
			=
			\int_0^t
			\mathfrak i_1^{(n,m,R)}(t')\d t'
			+
			\int_0^t
			\mathfrak i_2^{(n,m,R)}(t')\d W(t'),
		\end{equation*}
		where
		\begin{align*}
			\mathfrak i_1^{(n,m,R)}
			=\ &
			-\chi_R\Big(
			\|\partial_xv_n^{(R)}-\partial_xv_0\|_{L^\infty}
			+
			\|\mathcal H\partial_xv_n^{(R)}
			-\mathcal H\partial_xv_0\|_{L^\infty}
			\Big)
			J_n\big[
			(\mathcal HJ_nv_n^{(R)})\,
			\partial_xJ_nv_n^{(R)}
			\big]
			\\
			&+
			\chi_R\Big(
			\|\partial_xv_m^{(R)}-\partial_xv_0\|_{L^\infty}
			+
			\|\mathcal H\partial_xv_m^{(R)}
			-\mathcal H\partial_xv_0\|_{L^\infty}
			\Big)
			J_m\big[
			(\mathcal HJ_mv_m^{(R)})\,
			\partial_xJ_mv_m^{(R)}
			\big]
			\\
			&+
			\frac12\big(
			\mathcal Q_n^2v_n^{(R)}
			-
			\mathcal Q_m^2v_m^{(R)}
			\big)
		\end{align*}
		and
		\[
		\mathfrak i_2^{(n,m,R)}
		\triangleq
		\mathcal Q_nv_n^{(R)}
		-
		\mathcal Q_mv_m^{(R)}.
		\]
		
		Let $T\in[0,N]$. Since
		$Z_{n,m}^{(R)}(0)=0$, It\^o's formula applied up to
		$t\wedge\tau_N^{n,m}(R)$ gives, for every $t\in[0,T]$,
		\begin{align}
			\|Z_{n,m}^{(R)}(t\wedge\tau_N^{n,m}(R))\|_{H^\theta}^2
			=\ &
			\int_0^{t\wedge\tau_N^{n,m}(R)}
			\left[
			2\IP{
				\mathfrak i_1^{(n,m,R)}(t'),
				Z_{n,m}^{(R)}(t')
			}_{H^\theta}
			+
			\|\mathfrak i_2^{(n,m,R)}(t')\|_{H^\theta}^2
			\right]\d t'
			\notag\\
			&+
			2\int_0^{t\wedge\tau_N^{n,m}(R)}
			\IP{
				\mathfrak i_2^{(n,m,R)}(t'),
				Z_{n,m}^{(R)}(t')
			}_{H^\theta}
			\d W(t').
			\label{Znm-Ito}
		\end{align}

		We first estimate the drift term. In the difference of the two cut-off
		nonlinearities, add and subtract
		\begin{align*}
			\chi_R\Big(
			\|\partial_xv_n^{(R)}-\partial_xv_0\|_{L^\infty}
			+
			\|\mathcal H\partial_xv_n^{(R)}
			-\mathcal H\partial_xv_0\|_{L^\infty}
			\Big)
			J_m\big[
			(\mathcal HJ_mv_m^{(R)})\,
			\partial_xJ_mv_m^{(R)}
			\big].
		\end{align*}
		The part with the same cut-off factor is controlled by
		Lemma~\ref{Lemma : Huux}. For the remaining part, the Lipschitz
		continuity of $\chi_R$, $\|\chi'_R\|_{L^\infty}\leq \|\chi'\|_{L^\infty}$, and
		$H^\theta\hookrightarrow W^{1,\infty}$ yield
		\begin{align*}
			&\Bigg|
			\chi_R\Big(
			\|\partial_xv_n^{(R)}-\partial_xv_0\|_{L^\infty}
			+
			\|\mathcal H\partial_xv_n^{(R)}
			-\mathcal H\partial_xv_0\|_{L^\infty}
			\Big) -
			\chi_R\Big(
			\|\partial_xv_m^{(R)}-\partial_xv_0\|_{L^\infty}
			+
			\|\mathcal H\partial_xv_m^{(R)}
			-\mathcal H\partial_xv_0\|_{L^\infty}
			\Big)
			\Bigg|
			\\
			\lesssim\ &
			\|\partial_xZ_{n,m}^{(R)}\|_{L^\infty}
			+
			\|\mathcal H\partial_xZ_{n,m}^{(R)}\|_{L^\infty}
			\lesssim
			\|Z_{n,m}^{(R)}\|_{H^\theta}.
		\end{align*}
		Since $s>\theta+1$, the Sobolev product estimate also gives, uniformly
		in $m$,
		\[
		\left\|
		J_m\big[
		(\mathcal HJ_mv_m^{(R)})\,
		\partial_xJ_mv_m^{(R)}
		\big]
		\right\|_{H^\theta}
		\lesssim
		\|v_m^{(R)}\|_{H^s}^2
		\le C_N.
		\]
		Consequently, for $0\le t\le\tau_N^{n,m}(R)$,
		\begin{align}
			&\Bigg|
			\Bigg\langle
			\chi_R\Big(
			\|\partial_xv_n^{(R)}-\partial_xv_0\|_{L^\infty}
			+
			\|\mathcal H\partial_xv_n^{(R)}
			-\mathcal H\partial_xv_0\|_{L^\infty}
			\Big)
			J_n\big[
			(\mathcal HJ_nv_n^{(R)})\,
			\partial_xJ_nv_n^{(R)}
			\big]
			\notag\\
			&\qquad-
			\chi_R\Big(
			\|\partial_xv_m^{(R)}-\partial_xv_0\|_{L^\infty}
			+
			\|\mathcal H\partial_xv_m^{(R)}
			-\mathcal H\partial_xv_0\|_{L^\infty}
			\Big)
			J_m\big[
			(\mathcal HJ_mv_m^{(R)})\,
			\partial_xJ_mv_m^{(R)}
			\big],
			Z_{n,m}^{(R)}
			\Bigg\rangle_{H^\theta}
			\Bigg|
			\notag\\
			&\qquad\le
			C_{N}
			\left[
			(n\wedge m)^{-2(s-1-\theta)}
			+
			\|Z_{n,m}^{(R)}\|_{H^\theta}^2
			\right].
			\label{cutoff-nonlinearity-difference-estimate}
		\end{align}
		
		Next, Lemma~\ref{Lemma:Qn} gives that for $0\le t\le\tau_N^{n,m}(R)$,
		\begin{align}
			\left|
			\IP{
				\mathcal Q_n^2v_n^{(R)}
				-
				\mathcal Q_m^2v_m^{(R)},
				Z_{n,m}^{(R)}
			}_{H^\theta}
			+
			\|\mathcal Q_nv_n^{(R)}
			-\mathcal Q_mv_m^{(R)}\|_{H^\theta}^2
			\right|  
			\le 
			C_N
			\left[
			(n\wedge m)^{-(s-\theta-\delta_{\mathcal Q})}
			+
			\|Z_{n,m}^{(R)}\|_{H^\theta}^2
			\right].
			\label{Qnm-combined-difference-estimate}
		\end{align}
		
		Set
		\begin{equation*}
			\lambda_{n,m}
			\triangleq
			(n\wedge m)^{-2(s-1-\theta)}
			+
			(n\wedge m)^{-(s-\theta-\delta_{\mathcal Q})}.
		\end{equation*}
		Then $\lambda_{n,m}\to0$ as $n,m\to\infty$. By the definition of
		$\mathfrak i_1^{(n,m,R)}$ and $\mathfrak i_2^{(n,m,R)}$, the nonlinear
		estimate \eqref{cutoff-nonlinearity-difference-estimate} and the combined
		cancellation estimate \eqref{Qnm-combined-difference-estimate} imply that
		\begin{align}
			2\IP{
				\mathfrak i_1^{(n,m,R)}(t),
				Z_{n,m}^{(R)}(t)
			}_{H^\theta}
			+
			\|\mathfrak i_2^{(n,m,R)}(t)\|_{H^\theta}^2
			\le
			C_{N}\lambda_{n,m}
			+
			C_{N}\|Z_{n,m}^{(R)}(t)\|_{H^\theta}^2,\quad 0\le t\le\tau_N^{n,m}(R).
			\label{drift-difference-pointwise-estimate}
		\end{align}
		Consequently,
		\begin{align}
			&\E\left[
			\sup_{0\le t\le T}
			\int_0^{t\wedge\tau_N^{n,m}(R)}
			\left(
			2\IP{
				\mathfrak i_1^{(n,m,R)}(t'),
				Z_{n,m}^{(R)}(t')
			}_{H^\theta}
			+
			\|\mathfrak i_2^{(n,m,R)}(t')\|_{H^\theta}^2
			\right)\d t'
			\Big|\F_0
			\right]
			\notag\\
			\le\ &
			C_{N}\lambda_{n,m}
			+
			C_{N}
			\int_0^T
			\E\left[
			\sup_{0\le t'\le t\wedge\tau_N^{n,m}(R)}
			\|Z_{n,m}^{(R)}(t')\|_{H^\theta}^2
			\Big|\F_0
			\right]\d t.
			\label{R-1-nmR estiamte}
		\end{align}
		Here the supremum causes no difficulty because the integrand in
		\eqref{drift-difference-pointwise-estimate} is bounded from above by a
		nonnegative function. We also used $T\le N$ to absorb the factor $T$
		into $C_{N,R}$.
		
		It remains to estimate the martingale term. Adding and subtracting
		$\mathcal Q_nv_m^{(R)}$, we have
		\[
		\mathfrak i_2^{(n,m,R)}
		=
		\mathcal Q_nZ_{n,m}^{(R)}
		+
		(\mathcal Q_n-\mathcal Q_m)v_m^{(R)}.
		\]
		The uniform cancellation estimate and the difference estimate in
		Lemma~\ref{Lemma:Qn} imply, for
		$0\le t\le\tau_N^{n,m}(R)$,
		\begin{align*}
			\left|
			\IP{
				\mathfrak i_2^{(n,m,R)},
				Z_{n,m}^{(R)}
			}_{H^\theta}
			\right|
			&\le
			C\|Z_{n,m}^{(R)}\|_{H^\theta}^2
			+
			C_N
			(n\wedge m)^{-(s-\theta-\delta_{\mathcal Q})}
			\|Z_{n,m}^{(R)}\|_{H^\theta}.
		\end{align*}
		Since $\tau_N^{n,m}(R)\le N$ and both stopped solutions are bounded in
		$H^s$ by $N$, the stopped stochastic integral below is square-integrable.
		Therefore, the conditional BDG inequality and Young's inequality give
		\begin{align}
			&2\E\left[
			\sup_{0\le t\le T}
			\left|
			\int_0^{t\wedge\tau_N^{n,m}(R)}
			\IP{
				\mathfrak i_2^{(n,m,R)}(t'),
				Z_{n,m}^{(R)}(t')
			}_{H^\theta}
			\d W(t')
			\right|
			\Big|\F_0
			\right]\notag\\
			\le \ &
			C_N\E\left[
			\left(
			\int_0^{T\wedge\tau_N^{n,m}(R)}
			\|Z_{n,m}^{(R)}(t)\|_{H^\theta}^2
			\left[
			\|Z_{n,m}^{(R)}(t)\|_{H^\theta}^2
			+
			(n\wedge m)^{-2(s-\theta-\delta_{\mathcal Q})}
			\right]\d t
			\right)^{1/2}
			\Big|\F_0
			\right]
			\notag\\
			\le \ &
			\frac14
			\E\left[
			\sup_{0\le t\le T}
			\|Z_{n,m}^{(R)}(t\wedge\tau_N^{n,m}(R))\|_{H^\theta}^2
			\Big|\F_0
			\right]+
			C_N
			\int_0^T
			\E\left[
			\sup_{0\le t'\le t\wedge\tau_N^{n,m}(R)}
			\|Z_{n,m}^{(R)}(t')\|_{H^\theta}^2
			\Big|\F_0
			\right]\d t
			+
			C_N\lambda_{n,m}.
			\label{R-2-nmR estiamte}
		\end{align}

		Taking the supremum over $t\in[0,T]$ in \eqref{Znm-Ito}, then using
		\eqref{R-1-nmR estiamte} and \eqref{R-2-nmR estiamte}, and absorbing
		the first term on the right-hand side of
		\eqref{R-2-nmR estiamte}, we find
		\begin{align*}
			\E\left[
			\sup_{0\le t\le T\wedge\tau_N^{n,m}(R)}
			\|Z_{n,m}^{(R)}(t)\|_{H^\theta}^2
			\Big|\F_0
			\right] \le
			C_{N,R}\lambda_{n,m}
			+
			C_{N,R}
			\int_0^T
			\E\left[
			\sup_{0\le t'\le t\wedge\tau_N^{n,m}(R)}
			\|Z_{n,m}^{(R)}(t')\|_{H^\theta}^2
			\Big|\F_0
			\right]\d t.
		\end{align*}
		Using Gronwall's inequality, taking $T=N$ and then using $\tau_N^{n,m}(R)\le N$,  we arrive at
		\[
		\E\left[
		\sup_{0\le t\le\tau_N^{n,m}(R)}
		\|v_n^{(R)}(t)-v_m^{(R)}(t)\|_{H^\theta}^2
		\Big|\F_0
		\right]
		\le
		C_{N}\lambda_{n,m}.
		\]
		This proves \eqref{Cauchy in E-M}.

		\textit{\underline{Proof of \ref{Convergence of X-n T}}.}
		Fix $T>0$ and let $N$ be an integer such that $N>T$. By the definition of
		$\tau_N^{n,m}(R)$,
		\begin{align*}
			\{\tau_N^{n,m}(R)<T\}
			\subset
			\left\{
			\sup_{t\in[0,T]}
			\|v_n^{(R)}(t)\|_{H^s}\ge N
			\right\} \cup
			\left\{
			\sup_{t\in[0,T]}
			\|v_m^{(R)}(t)\|_{H^s}\ge N
			\right\}.
		\end{align*}
		Consequently, conditional Chebyshev's inequality and
		\eqref{Xn uniform bound} give
		\begin{align}
			\p\bigl(\tau_N^{n,m}(R)<T\bigm|\F_0\bigr) \le
			\frac{4\|v_0\|_{H^s}^2}{N^2}
			\exp\left\{
			C\Big(
			1+2R
			+\|\partial_xv_0\|_{L^\infty}
			+\|\mathcal H\partial_xv_0\|_{L^\infty}
			\Big)T
			\right\},
			\label{exit-probability-tau-N-nm}
		\end{align}
		where $C$ is independent of $n$, $m$, and $N$.
		
		Let $\varepsilon>0$. The application of conditional
		Chebyshev's inequality yields
		\begin{align*}
			&\p\left(
			\sup_{t\in[0,T]}
			\|v_n^{(R)}(t)-v_m^{(R)}(t)\|_{H^\theta}
			>\varepsilon
			\Bigm|\F_0
			\right)\\
			\le \ &
			\p\bigl(\tau_N^{n,m}(R)<T\bigm|\F_0\bigr)
			+
			\frac{1}{\varepsilon^2}
			\E\left[
			\sup_{t\in[0,\tau_N^{n,m}(R)]}
			\|v_n^{(R)}(t)-v_m^{(R)}(t)\|_{H^\theta}^2
			\Bigm|\F_0
			\right].
		\end{align*}
		Taking the supremum over $m\ge n$, then using
		\eqref{Cauchy in E-M} and \eqref{exit-probability-tau-N-nm}, and finally letting $N\to\infty$, we obtain
		\begin{equation*}
			\lim_{n\to\infty}\sup_{m\ge n}
			\p\left(
			\sup_{t\in[0,T]}
			\|v_n^{(R)}(t)-v_m^{(R)}(t)\|_{H^\theta}
			>\varepsilon
			\Bigm|\F_0
			\right)
			=0
			\quad\pas
		\end{equation*}

		By the reverse Fatou lemma (since the probabilities are bounded by $1$), this implies
		\begin{equation}
			\lim_{n\to\infty}\sup_{m\ge n}
			\p\left(
			\sup_{t\in[0,T]}
			\|v_n^{(R)}(t)-v_m^{(R)}(t)\|_{H^\theta}
			>\varepsilon
			\right)
			=0.
			\label{unconditional-Cauchy-H-theta}
		\end{equation}
		
		We next choose one subsequence that works simultaneously on every compact
		time interval. By \eqref{unconditional-Cauchy-H-theta}, one may choose an
		increasing sequence of indices $\{n_k\}_{k\ge1}$ so rapidly that
		\begin{equation*}
			\p\left(
			\sup_{t\in[0,k]}
			\|v_{n_{k+1}}^{(R)}(t)-v_{n_k}^{(R)}(t)\|_{H^\theta}
			>2^{-k}
			\right)
			\le 2^{-k},
			\qquad k\ge1.
		\end{equation*}
		The Borel--Cantelli lemma then shows that, almost surely, for all
		sufficiently large $k$,
		\begin{equation*}
			\sup_{t\in[0,k]}
			\|v_{n_{k+1}}^{(R)}(t)-v_{n_k}^{(R)}(t)\|_{H^\theta}
			\le 2^{-k}.
		\end{equation*}
		It follows that, on a single event of probability one, the sequence
		$\{v_{n_k}^{(R)}\}_{k\ge1}$ is Cauchy in
		$C([0,T];H^\theta)$ for every $T>0$. After relabelling this subsequence as
		$\{v_n^{(R)}\}_{n\ge1}$, we have
		\begin{equation*}
			\lim_{n,m\to\infty}
			\sup_{t\in[0,T]}
			\|v_n^{(R)}(t)-v_m^{(R)}(t)\|_{H^\theta}
			=0,
			\qquad T>0
			\quad\pas
		\end{equation*}
		Since $H^\theta$ is complete, there exists an $H^\theta$-valued process
		$v^{(R)}$ with continuous paths such that
		\begin{equation*}
			v_n^{(R)}\longrightarrow v^{(R)}
			\quad\text{in }C([0,T];H^\theta),
			\qquad T>0
			\quad\pas
		\end{equation*}
		Equivalently,
		\begin{equation}
			\label{Xn to X uniform-t H-theta-limit}
			\lim_{n\to\infty}
			\sup_{t\in[0,T]}
			\|v_n^{(R)}(t)-v^{(R)}(t)\|_{H^\theta}
			=0,
			\qquad T>0
			\quad\pas
		\end{equation}
		Because each $v_n^{(R)}(t)$ is progressively
		measurable,  $v^{(R)}(t)$ is also progressively measurable as
		an $H^\theta$-valued process.  Moreover,
		by the conditional Fatou lemma and
		\eqref{Xn uniform bound}, $v^{(R)}$ is $H^s$-valued and inherits the uniform
		$H^s$ estimate:
		\begin{align*}
			\E\left[
			\liminf_{n\to\infty}
			\sup_{t\in[0,T]}
			\|v_n^{(R)}(t)\|_{H^s}^2
			\Bigm|\F_0
			\right] \le
			\liminf_{n\to\infty}
			\E\left[
			\sup_{t\in[0,T]}
			\|v_n^{(R)}(t)\|_{H^s}^2
			\Bigm|\F_0
			\right]
			<\infty
			\quad\pas
		\end{align*}
		Consequently,
		\begin{equation}
			\label{uniform_Hs_bound_liminf}
			\liminf_{n\to\infty}
			\sup_{t\in[0,T]}
			\|v_n^{(R)}(t)\|_{H^s}
			<\infty
			\quad\pas
		\end{equation}
		
		Fix an $\omega$ in the probability-one set on which
		\eqref{Xn to X uniform-t H-theta-limit} and
		\eqref{uniform_Hs_bound_liminf} hold. Choose a subsequence $\{n_k\}_{k\ge1}$, depending only
		on $T$ and $\omega$, along which the lower limit in
		\eqref{uniform_Hs_bound_liminf} is attained. Thus,
		\begin{align*}
			&
			\lim_{k\to\infty}
			\sup_{t\in[0,T]}
			\|v_{n_k}^{(R)}(\omega,t)\|_{H^s} =
			\liminf_{n\to\infty}
			\sup_{t\in[0,T]}
			\|v_n^{(R)}(\omega,t)\|_{H^s}
			<\infty.
		\end{align*}
		This single subsequence is uniformly bounded in $H^s$ for all
		$t\in[0,T]$. For any fixed $t\in[0,T]$, weak compactness in $H^s$ and
		\eqref{Xn to X uniform-t H-theta-limit} show that
		$v_{n_k}^{(R)}(\omega,t)$ converges weakly in $H^s$ to
		$v^{(R)}(\omega,t)$. Indeed, every weakly convergent subsubsequence has
		an $H^\theta$ limit, and the strong $H^\theta$ convergence identifies
		that limit uniquely with $v^{(R)}(\omega,t)$. Hence
		$v^{(R)}(\omega,t)\in H^s$, and weak lower semicontinuity gives
		\begin{align}
			\|v^{(R)}(\omega,t)\|_{H^s}
			\le
			\liminf_{k\to\infty}
			\|v_{n_k}^{(R)}(\omega,t)\|_{H^s} \le
			\liminf_{n\to\infty}
			\sup_{t\in[0,T]}
			\|v_n^{(R)}(\omega,t)\|_{H^s}.
			\label{pointwise Hs bound}
		\end{align}
		Because the subsequence above is independent of $t$, we may take the
		supremum over $t\in[0,T]$ in \eqref{pointwise Hs bound} and then use \eqref{uniform_Hs_bound_liminf} to obtain
		\begin{equation}
			\sup_{t\in[0,T]}
			\|v^{(R)}(t)\|_{H^s}^2
			\le
			\liminf_{n\to\infty}
			\sup_{t\in[0,T]}
			\|v_n^{(R)}(t)\|_{H^s}^2<\infty
			\quad\pas
			\label{limit-Hs-pathwise-bound}
		\end{equation}
		
		By \eqref{limit-Hs-pathwise-bound}, the dense embedding $H^s\hookrightarrow H^\theta$ and the fact $v^{(R)}\in C([0,T];H^\theta)$, we see that
		$v^{(R)}\in C_w([0,T];H^s)$. In particular,
		$t\mapsto\|v^{(R)}(t)\|_{H^s}$ is lower semicontinuous, so the supremum in
		\eqref{limit-Hs-pathwise-bound} is measurable. Applying the conditional
		Fatou lemma once more and then using \eqref{Xn uniform bound}, we obtain
		\eqref{X Hs bound}.
		
		Repeating the preceding argument for $T\in\mathbb N$ and intersecting the
		resulting probability-one events shows that $v^{(R)}(t)\in H^s$ for all
		$t\ge0$, almost surely.
		
		Progressive measurability of $v^{(R)}$ in the $H^\theta$ topology
		holds because it is  the almost sure limit of progressively measurable
		$H^\theta$-valued processes. Since $H^s$ is continuously embedded into
		$H^\theta$ and both spaces are separable Hilbert spaces, the Borel
		$\sigma$-algebras of $(H^s,\|\cdot\|_{H^s})$ and $(H^s,\|\cdot\|_{H^\theta})$ coincide (see, e.g., \cite[page~107]{Schwartz-1973-Book} or
		\cite[Theorem~6.8.6]{Bogachev-2007-Books}). Since $v^{(R)}(t)$ and
		$v_n^{(R)}(t)$ belong to $H^s$,  the process $v^{(R)}$  is also
		progressively measurable as an $H^s$-valued process.
		
		Finally, let $s'\in(\theta,s)$ and choose $\lambda_{\mathrm{int}}\in(0,1)$ such that
		\[
		s'=\lambda_{\mathrm{int}}\theta+(1-\lambda_{\mathrm{int}})s.
		\]
		Since $v^{(R)}\in C([0,T];H^\theta)$,
		interpolation
		gives
		\begin{align*}
			\|v^{(R)}(t)-v^{(R)}(t')\|_{H^{s'}}\le
			C\|v^{(R)}(t)-v^{(R)}(t')\|_{H^\theta}^{\lambda_{\mathrm{int}}}
			\|v^{(R)}(t)-v^{(R)}(t')\|_{H^s}^{1-\lambda_{\mathrm{int}}},
		\end{align*}
		which means that
		$v^{(R)}\in C([0,T];H^{s'})$.
		Applying the same interpolation inequality to the approximation error
		gives
		\begin{align*}
			\sup_{t\in[0,T]}
			\|v_n^{(R)}(t)-v^{(R)}(t)\|_{H^{s'}}
			\le
			C
			\left(
			\sup_{t\in[0,T]}
			\|v_n^{(R)}(t)-v^{(R)}(t)\|_{H^\theta}
			\right)^{\lambda_{\mathrm{int}}}
			\left(
			\sup_{t\in[0,T]}
			\|v_n^{(R)}(t)-v^{(R)}(t)\|_{H^s}
			\right)^{1-\lambda_{\mathrm{int}}}.
		\end{align*}
		Therefore, for every $L>0$ and $\varepsilon>0$,
		\begin{align}
			&\p\left(
			\sup_{t\in[0,T]}
			\|v_n^{(R)}(t)-v^{(R)}(t)\|_{H^{s'}}
			>\varepsilon
			\bigg|\F_0
			\right)\notag\\
			\le\ &
			\p\left(
			\sup_{t\in[0,T]}
			\|v_n^{(R)}(t)-v^{(R)}(t)\|_{H^s}
			>L
			\bigg|\F_0
			\right) +
			\p\left(
			\sup_{t\in[0,T]}
			\|v_n^{(R)}(t)-v^{(R)}(t)\|_{H^\theta}
			>
			\left(
			\frac{\varepsilon}{CL^{1-\lambda_{\mathrm{int}}}}
			\right)^{1/\lambda_{\mathrm{int}}}
			\bigg|\F_0
			\right).\label{eq:converge probabiltiy Hs' 1}
		\end{align}
		Moreover,
		\begin{align*}
			\E\left[
			\sup_{t\in[0,T]}
			\|v_n^{(R)}(t)-v^{(R)}(t)\|_{H^s}^2
			\Bigm|\F_0
			\right]
			\le\ &
			2\E\left[
			\sup_{t\in[0,T]}
			\|v_n^{(R)}(t)\|_{H^s}^2
			\Bigm|\F_0
			\right] +
			2\E\left[
			\sup_{t\in[0,T]}
			\|v^{(R)}(t)\|_{H^s}^2
			\Bigm|\F_0
			\right].
		\end{align*}
		Thus, by conditional Chebyshev's inequality,
		\eqref{Xn uniform bound}, and \eqref{X Hs bound},
		\begin{align}
			&\p\left(
			\sup_{t\in[0,T]}
			\|v_n^{(R)}(t)-v^{(R)}(t)\|_{H^s}
			>L
			\Bigm|\F_0
			\right)\notag\\
			\le\ &
			\frac{C\|v_0\|_{H^s}^2}{L^2}
			\exp\left\{
			C\Big(
			1+2R
			+\|\partial_xv_0\|_{L^\infty}
			+\|\mathcal H\partial_xv_0\|_{L^\infty}
			\Big)T
			\right\}.\label{eq:converge probabiltiy Hs' 2}
		\end{align}
		By \eqref{Xn to X uniform-t H-theta-limit} and \eqref{eq:converge probabiltiy Hs' 2},
		we can
		first send $n\to\infty$ and then send $L\to\infty$ in \eqref{eq:converge probabiltiy Hs' 1} to
		obtain
		\begin{equation*}
			\lim_{n\to\infty}
			\p\left(
			\sup_{t\in[0,T]}
			\|v_n^{(R)}(t)-v^{(R)}(t)\|_{H^{s'}}
			>\varepsilon
			\Bigm|\F_0
			\right)
			=0
			\quad\pas,
		\end{equation*}
		which gives convergence in probability in
		$C([0,T];H^{s'})$.
		
		Choose a countable dense subset of $(\theta,s)$. By a standard diagonal
		subsequence argument over all integer time horizons and all exponents in
		this dense subset, there exists a further subsequence, still denoted by
		$\{v_n^{(R)}\}_{n\ge1}$, such that, almost surely,
		\begin{equation*}
			v_n^{(R)}\longrightarrow v^{(R)}
			\quad\text{in }C([0,T];H^\zeta)
		\end{equation*}
		for every $T\in\mathbb N$ and every exponent $\zeta$ in the chosen dense
		subset. The limit is necessarily the same $v^{(R)}$, because convergence
		in $H^\zeta$ implies convergence in $H^\theta$, where the limit has already
		been identified.  Given any $s'\in(\theta,s)$, choose such a $\zeta$ with
		$s'<\zeta<s$. The continuous embedding $H^\zeta\hookrightarrow H^{s'}$ then
		establishes \eqref{Xn to X} for this $s'$; the endpoint $s'=\theta$ was
		proved in \eqref{Xn to X uniform-t H-theta-limit}. An arbitrary finite
		time horizon is contained in an integer one. This completes the proof of
		\ref{Convergence of X-n T} and hence the proof of the lemma.
	\end{proof}

	\subsection{Solving the cut-off problem}

	\begin{Lemma}
		Let the assumptions of Theorem~\ref{Thm-local-solution} hold and let $v^{(R)}$ be obtained in Lemma \ref{lem:Cauchy argument}. Define
		\begin{align*}
			\mathscr V_R(a,u)
			\triangleq-
			\chi_R\Big(
			\|\partial_xu-\partial_xa\|_{L^\infty}
			+
			\|\mathcal H\partial_xu-\mathcal H\partial_xa\|_{L^\infty}
			\Big)
			\big[
			(\mathcal H u)\,\partial_x u
			\big]
			+
			\frac12\mathcal Q^2u,\quad a,\ u\in H^s.
		\end{align*}
		Then
			\begin{equation}\label{eq : cut-off problem}
			v^{(R)}(t)-v_0
			=
			\int_0^t
			\mathscr V_R\bigl(v_0,v^{(R)}(t')\bigr)\d t'
			+
			\int_0^t \mathcal Qv^{(R)}(t')\d W(t')
		\end{equation}
		holds as an equation
		in $C([0,T];H^\theta)$ with $\theta\in
		\left(
		\frac32,\,
		s-\max\{2\delta_{\mathcal Q},1\}
		\right)$.
	Moreover, $H^s$-solution  to \eqref{eq : cut-off problem} is unique.
	\end{Lemma}
	\begin{proof}
	Choose $s_1<s$ sufficiently close to $s$ so that
	$
	s_1>\max\left\{\theta+2\delta_{\mathcal Q},\,\theta+1,\,\frac52\right\}.
	$
	By \eqref{Xn to X},
	\[
	v_n^{(R)}
	\longrightarrow
	v^{(R)}
	\quad
	\text{in }C([0,T];H^{s_1})
	\quad
	\text{a.s.}
	\]
	Using the mapping properties of $\mathcal Q$ and the approximation properties of $\mathcal Q_n$ from Lemma~\ref{Lemma:Qn}, we have
	\[
	\mathcal Q_n v_n^{(R)}
	\longrightarrow
	\mathcal Q v^{(R)}
	\quad
	\text{in }C([0,T];H^\theta)
	\]
	and
	\[
	\mathcal Q_n^2 v_n^{(R)}
	\longrightarrow
	\mathcal Q^2 v^{(R)}
	\quad
	\text{in }C([0,T];H^\theta),
	\]
	in probability, and along a further subsequence almost surely.

Consequently, by the Lenglart inequality and
\eqref{Xn uniform bound},
\[
\int_0^\cdot
\mathcal Q_n v_n^{(R)}(t')\d W(t')
\longrightarrow
\int_0^\cdot
\mathcal Q v^{(R)}(t')\d W(t')
\]
in probability in $C([0,T];H^\theta)$. If necessary, one can first localise the stochastic integrand by a sequence of stopping times and then remove the localisation, similar to \eqref{tau-L-nR-definition} and \eqref{tau-L-nR-limit}.

	The deterministic integral involving $\mathcal Q_n^2v_n^{(R)}$ converges in $C([0,T];H^\theta)$ as well.
	For the nonlinear term, by the properties of $J_n$ in Lemma~\ref{Lemma-Jn} and the convergence of $v_n^{(R)}$ in $C([0,T];H^{s_1})$ we find that
	\[
	J_n\big[
	(\mathcal HJ_nv_n^{(R)})
	\partial_xJ_nv_n^{(R)}
	\big]
	\longrightarrow
	(\mathcal Hv^{(R)})\partial_xv^{(R)}
	\quad
	\text{in }C([0,T];H^\theta).
	\]
	Since $s_1>\frac52$, we have $H^{s_1-1}\hookrightarrow L^\infty$, and therefore the maps
	\[
	u\mapsto
	\|\partial_xu-\partial_xv_0\|_{L^\infty},
	\quad
	u\mapsto
	\|\mathcal H\partial_xu-\mathcal H\partial_xv_0\|_{L^\infty}
	\]
	are continuous on $H^{s_1}$. Hence the cut-off factors converge as well. Passing to the limit in \eqref{approximation scheme C}, we obtain that \eqref{eq : cut-off problem}.
	Since \(\theta>3/2\), the identity \eqref{eq : cut-off problem} also holds in \(C([0,T];C_b^1(\mathbb R))\).
	
	The uniqueness of solutions to \eqref{eq : cut-off problem} follows from the same argument used in the proof of \eqref{Cauchy in E-M}. Without the approximation,   \cite[Theorem 4.1]{Karlsen-Tang-Wang-2026-arXiv} gives
	\begin{equation}
		\left|
		\langle\mathcal Qf,f\rangle_{H^r}
		\right|
		\lesssim
		\|f\|_{H^r}^2,
		\quad
		\left|
		\langle\mathcal Q^2f,f\rangle_{H^r}
		+
		\|\mathcal Qf\|_{H^r}^2
		\right|
		\lesssim
		\|f\|_{H^r}^2,\quad r\ge0,\quad f\in H^{r+2\delta_{\mathcal Q}}.
		\label{Q cancellation}
	\end{equation}
	
	Consider two solutions $v_1$ and $v_2$ to \eqref{eq : cut-off problem} with the same initial data, and
	\[
	Z \triangleq v_1-v_2,
	\]
	then, for
	\[
	\tau_N^{1,2}
	\triangleq
	N
	\wedge
	\inf\left\{
	t\ge0:
	\|v_1(t)\|_{H^s}
	\vee
	\|v_2(t)\|_{H^s}
	\ge N
	\right\},
	\]
	the estimates \eqref{R-1-nmR estiamte}--\eqref{R-2-nmR estiamte}, now without the approximation error $\lambda_{n,m}$ (\eqref{Q cancellation} replaces the approximate cancellation estimate),  yield
	\begin{equation}\label{eq: uniqueness estimate}
		\E\left[
		\sup_{t\in[0,T\wedge\tau_N^{1,2}]}
		\|Z(t)\|_{H^\theta}^2
		\right]
		\le
		C_N
		\int_0^T
		\E\left[
		\sup_{r\in[0,t\wedge\tau_N^{1,2}]}
		\|Z(r)\|_{H^\theta}^2
		\right]\d t.
	\end{equation}
	By Gr\"onwall's inequality, $Z=0$ on $[0,T\wedge\tau_N^{1,2}]$ almost surely. Letting $N\to\infty$ gives pathwise uniqueness for the cut-off problem.
\end{proof}

	\subsection{Construction of the maximal solution}\label{Section: construction of the maximal solution}
	
	We first notice that the same stopped $H^\theta$ difference estimate as in
	\eqref{eq: uniqueness estimate}, now without the approximation error,
	yields pathwise uniqueness for $H^s$-valued solutions to
	\eqref{Target problem SCCF}.

	Let $R\in\mathbb N$, and let $v^{(R)}$ be the unique global solution to the cut-off problem \eqref{eq : cut-off problem}. We define
	\begin{equation*}
		\tau(R) \triangleq  \inf\big\{t\ge 0: \|\partial_xv^{(R)} -\partial_xv_0\|_{L^\infty}+\|\mathcal{H}\partial_xv^{(R)} -\mathcal{H}\partial_xv_0\|_{L^\infty}\ge R\big\}.
	\end{equation*}
	Since $v^{(R)}\in C([0,\infty);H^\theta)$ and
	$\theta>3/2$, the process
	\[
	t\longmapsto
	\|\partial_xv^{(R)}(t)-\partial_xv_0\|_{L^\infty}
	+
	\|\mathcal H\partial_xv^{(R)}(t)
	-\mathcal H\partial_xv_0\|_{L^\infty}
	\]
	is continuous and adapted. Hence $\tau(R)$ is a stopping time.
	Moreover, this process vanishes at $t=0$, and therefore
	$\mathbb P(\tau(R)>0)=1$.
	Since $$\chi_R(\|\partial_xv^{(R)} (t)-\partial_xv_0\|_{L^\infty}+\|\mathcal{H}\partial_xv^{(R)}(t) -\mathcal{H}\partial_xv_0\|_{L^\infty})=1,\quad t\le \tau(R),$$  we see that \eqref{Target problem SCCF} coincides with
	\eqref{eq : cut-off problem} up to time $\tau(R)$. This together with \eqref{X Hs bound} implies that
	$v^{(R)}$ satisfies \eqref{Target problem SCCF} up to $\tau(R)$.    By the standard consistency argument based on classical solutions to  \eqref{Target problem SCCF},
	$\tau(R)$ is nondecreasing in $R$, and
	\[
	v^{(R)}(t)=v^{(R+1)}(t),
	\qquad
	0\le t\le\tau(R),
	\quad R\in\mathbb{N} 
	\quad \mathbb P\text{-a.s.}
	\]
	We set
	\begin{equation*}
		\tau^*\triangleq \lim_{R\to\infty} \tau(R),\quad  \tau(0) \triangleq 0,
	\end{equation*}
	and define
	\begin{align*}
		v(t) \triangleq \sum_{R=1}^\infty \textbf{1}_{\big[\tau(R-1), \tau(R)\big)}(t) v^{(R)}(t),\quad t\in \big[0,\tau^*\big).
	\end{align*}
	For every $T<\tau^*$, one may choose $R$ such that $T<\tau(R)$.
	The consistency of the cut-off solutions makes this definition
	unambiguous and shows that $v$ is progressively measurable.
	Moreover, for every $T<\tau^*$ one may choose $R$ such that
	$T<\tau(R)$, and then
	\[
	v(t)=v^{(R)}(t),
	\qquad 0\le t\le T.
	\]
	Consequently, $v$ satisfies \eqref{Target problem SCCF} on
	$[0,\tau^*)$ and
	$
	v\in C([0,\tau^*);H^{s'})
	$
	for every $s'<s$.
	The maximal lifetime is independent of the Sobolev index $s>\frac{3}{2}+\max\{2\delta_{\mathcal Q},1\}$.  Indeed, let
	\[
	s_2>s_1>\frac32+\max\{2\delta_{\mathcal Q},1\},
	\qquad v_0\in H^{s_2}.
	\]
	For every $R\ge1$, the $H^{s_2}$ cut-off solution is also an
	$H^{s_1}$ solution to the same cut-off problem. Pathwise uniqueness in
	$H^{s_1}$ therefore gives
	\[
	v_{s_2}^{(R)}=v_{s_1}^{(R)}.
	\]
	Since the exit functional is the same for both constructions,
	\[
	\tau_{s_2}(R)=\tau_{s_1}(R),
	\qquad R\ge1.
	\]
	Hence, we see that
	\[
	\tau_{s_2}^*=\tau_{s_1}^*
	\quad\mathbb P\text{-a.s.}
	\]
	Moreover, since
	$\Lambda=-\mathcal H\partial_x$ and $\|\partial_x v_0\|_{L^\infty}+\|\mathcal{H} \partial_x v_0\|_{L^\infty}<\infty$ $\pas$,  this construction also implies that
	\begin{equation*} 
		\limsup_{t\uparrow\tau^*}\|v(t)\|_{H^s}
		=
		\infty
		\quad \Longleftrightarrow\quad
		\limsup_{t\uparrow\tau^*}
		\left(
		\|\partial_xv(t)\|_{L^\infty}
		+
		\|\Lambda v(t)\|_{L^\infty}
		\right)
		=\infty\quad
		\text{a.s. on}\ \left\{\tau^*<\infty\right\},
	\end{equation*}
	which is \eqref{eq:blow-up criterion statement}.

	Finally, it remains to prove that
	$v\in C([0,\tau^*);H^s)$ $\pas$
	Fix
	$
	\theta\in
	\left(
	\frac32,\,
	s-\max\{2\delta_{\mathcal Q},1\}
	\right).
	$
	The construction above gives, almost surely,
	$
	v\in C([0,\tau^*);H^\theta)
	$
	and
	$\sup_{0\le t\le T}\|v(t)\|_{H^s}<\infty$ for every $T<\tau^*$.
	It follows by the standard weak-continuity argument (cf. \cite[page 263, Lemma 1.4]{Temam-1977-book}) that
	\begin{equation}\label{eq:weak-Hs-continuity}
		v\in C_w([0,\tau^*);H^s)
		\quad \mathbb P\text{-a.s.}
	\end{equation}

	For $N\in\mathbb N$, define
	\begin{equation}\label{eq:Hs-continuity-stopping-time}
		\tau_N
		\triangleq
		N\wedge
		\inf\left\{
		t\in[0,\tau^*):
		\|v(t)\|_{H^s}\ge N
		\right\}.
	\end{equation}
	Since $v$ is progressively measurable as an $H^s$-valued process,
	the d\'ebut theorem shows that $\tau_N$ is a stopping time. Moreover,
	\begin{equation}\label{eq:Hs-stopping-times-exhaust}
		\tau_N\uparrow\tau^*
		\qquad\mathbb P\text{-a.s.}
	\end{equation}
	Indeed, on $\{\tau^*<\infty\}$ this follows from the already established
	estimate
	\[
	\limsup_{t\uparrow\tau^*}
	\left(
	\|\partial_xv(t)\|_{L^\infty}
	+
	\|\Lambda v(t)\|_{L^\infty}
	\right)
	=\infty
	\]
	and the Sobolev embedding
	\[
	\|\partial_xv\|_{L^\infty}
	+
	\|\Lambda v\|_{L^\infty}
	\lesssim
	\|v\|_{H^s}.
	\]
	On $\{\tau^*=\infty\}$, it follows from the local boundedness of
	$v$ in $H^s$.

	To prove strong $H^s$ continuity, we examine the mollified solution
	$J_n v$, where $J_n$ is defined in \eqref{Define Jn}. Throughout this
	argument we use the single family of stopping times $\{\tau_N\}_{N\ge1}$
	defined in \eqref{eq:Hs-continuity-stopping-time}. For every
	$r<\tau_N$, one has $\|v(r)\|_{H^s}<N$; the value at the single endpoint
	$r=\tau_N$ does not affect any of the stopped time integrals below.
	As in \eqref{tau-L-nR-definition}--\eqref{tau-L-nR-limit}, we first work
	with the stopped equation and then let $N\to\infty$ by
	\eqref{eq:Hs-stopping-times-exhaust}.
	
	Leveraging Assumption \ref{Hypo-Q} and applying Lemmas \ref{Lemma:Qn}, \ref{Lemma-Jn},
	and \ref{Lemma : Huux}, we can identify a nondecreasing function
	$P:[0,\infty)\to
	(0,\infty)$
	such that for all $N \ge 1$ and $t \ge 0$,
	\begin{align*}
		\sup_{n \ge 1,\, \| u\|_{H^s} \le N}
		\left| 2 \langle J_n [\H u \partial_x u ], J_n u \rangle_{H^s} \right| \le P(N),\quad
		\sup_{n \ge 1,\, \| u \|_{H^s} \le N}
		\langle J_n \Q u, J_n u \rangle_{H^s}^2 \le P(N),
	\end{align*}
	and
	\begin{align*}
		\sup_{n \ge 1,\, \|u \|_{H^s} \le N}
		\bigg| \langle J_n \Q^2 u, J_n u \rangle_{H^s}
		+ \| J_n \Q u \|_{H^s}^2 \bigg|
		\le P(N).
	\end{align*}
	Applying It\^o's formula to $\| J_n v(t) \|_{H^s}^2$ and
	using the above estimates, for any $n \ge 1$, we find that
	\begin{align}\label{Ito to Jn u-1}
		\|J_n v(t\wedge\tau_N)\|_{H^s}^2
		= \ &
		\|J_n v_0\|_{H^s}^2
		-
		\int_0^{t\wedge\tau_N}
		\left[
		2 \langle J_n [\H v \partial_x v ], J_n v \rangle_{H^s}
		\right](t')\d t'
		\notag\\
		&+
		\int_0^{t\wedge\tau_N}
		\Big[
		\bIP{J_n \Q^2 v(t'), J_n v(t')}_{H^s}+
		\|J_n \mathcal Q v(t')\|_{H^s}^2
		\Big]\d t'
		\notag\\
		&+
		2\int_0^{t\wedge\tau_N}
		\bIP{
			J_n \mathcal Q v(t'),J_n v(t')
		}_{H^s}
		\d W(t')
	\end{align}
	Let $T>0$.
	Then the BDG inequality and Lemma \ref{Lemma:Qn} applied to increments of
	\eqref{Ito to Jn u-1} yield, for a constant $C_{N,T}>0$ independent of $n$,
	\begin{equation}
		\mathbb{E} \left[ \Big|\|J_n v(t \land \tau_N) \|^2_{H^s}
		- \|J_n v(t' \land \tau_N)\|^2_{H^s} \Big|^4 \right] \le C_{N,T} | t - t' |^2,
		\quad t, t' \in[0,T], \ N \ge 1.\label{eq:Jn v time difference}
	\end{equation}

	Fix
	$
	\gamma\in\left(0,\frac14\right).
	$
	For $n,N\ge1$ and $T>0$, define
	\begin{equation*} 
		\mathfrak S_{n,N,T}
		\triangleq
		\int_0^T\int_0^T
		\frac{
			\Big|\|J_n v(t \land \tau_N) \|^2_{H^s}
			- \|J_n v(t' \land \tau_N)\|^2_{H^s} \Big|^4
		}{
			|t-t'|^{2+4\gamma}
		}
		\,\mathrm dt'\,\mathrm dt.
	\end{equation*}
	By Tonelli's theorem and the preceding increment estimate,
	\begin{align}\label{eq:Expectation S-nNT}
		\mathbb E\mathfrak S_{n,N,T}
		\le
		C_{N,T}
		\int_0^T\int_0^T
		|t-t'|^{-4\gamma}
		\,\mathrm dt'\,\mathrm dt
		=
		C_{N,T}
		\frac{
			2T^{2-4\gamma}
		}{
			(1-4\gamma)(2-4\gamma)
		}.
	\end{align}
	Now we apply Lemma~\ref{lem:fractional-Sobolev-Morrey-time} with \[ p=4, \qquad r=\frac14+\gamma,\qquad f(\cdot)=\|J_n v(\cdot\wedge \tau_N)\|^2_{H^s}\] to obtain
	\[
	\Big|\|J_n v(t \land \tau_N) \|^2_{H^s}
	- \|J_n v(t' \land \tau_N)\|^2_{H^s} \Big|^4
	\le
	C_\gamma
	|t-t'|^{4\gamma}
	\mathfrak S_{n,N,T},
	\qquad
	t,t'\in[0,T].
	\]
	Here we note that $\|J_n v(t \land \tau_N) \|^2_{H^s}$ already has continuous paths, so its continuous representative
	is $\|J_n v(t \land \tau_N) \|^2_{H^s}$ itself.
	Then we arrive at
	\begin{equation}\label{eq:uniform-random-holder}
		\left|
		\|J_n v(t\wedge\tau_N)\|_{H^s}^2
		-
		\|J_n v(t'\wedge\tau_N)\|_{H^s}^2
		\right|
		\le
		\mathscr C_{n,N,T}|t-t'|^\gamma,
		\qquad
		\mathscr C_{n,N,T}
		\triangleq
		C_\gamma^{1/4}
		\mathfrak S_{n,N,T}^{1/4},\qquad
		t,t'\in[0,T],
	\end{equation}
	and \eqref{eq:Expectation S-nNT} implies
	\begin{equation}\label{eq:uniform-holder-moment}
		\sup_{n\ge1}
		\mathbb E
		\left[
		\mathscr C_{n,N,T}^4
		\right]
		<\infty.
	\end{equation}
	By Fatou's lemma and \eqref{eq:uniform-holder-moment},
	\begin{equation*}
		\mathbb E\left[\liminf_{n\to\infty}
		\mathscr C_{n,N,T}^{4}\right]
		\le
		\liminf_{n\to\infty}
		\mathbb E\left[\mathscr C_{n,N,T}^{4}\right]
		<\infty.
	\end{equation*}
		Consequently, there exists an event
	$\Omega_{N,T}$ of probability one on which
	\[
	L_{N,T}(\omega)
	\triangleq
	\liminf_{n\to\infty}
	\mathscr C_{n,N,T}(\omega)
	<\infty,
	\]
	and on which the pathwise estimate
	\eqref{eq:uniform-random-holder} holds simultaneously for every
	$n\ge1$ and every $t,t'\in[0,T]$.
	
	For every $\omega\in\Omega_{N,T}$, choose an increasing
	subsequence $\{n_k\}_{k\ge1}=\{n_k(\omega,N,T)\}_{k\ge1},$
	such that
	\begin{equation}\label{eq:pathwise-uniform-holder-subsequence}
		\mathscr C_{n_k,N,T}(\omega)
		\le
		1+L_{N,T}(\omega),
		\quad k\ge1,\quad \text{and}\quad 	\sup_{k\ge1}
		\mathscr C_{n_k,N,T}(\omega)
		\le
		1+L_{N,T}(\omega)
		<\infty.
	\end{equation}

	Let $\Omega_{\mathrm{reg}}$ be the probability-one event, obtained in
	the construction above, on which
	\begin{equation*} 
		v(t,\omega)\in H^s
		\qquad
		\text{for every }0\le t<\tau^*(\omega).
	\end{equation*}
	Fix $	\omega	\in	\Omega_{\mathrm{reg}}	\cap
	\Omega_{N,T}.$
	By the construction of $\tau_N$ (see \eqref{eq:Hs-continuity-stopping-time}), one has
	$
	\tau_N(\omega)<\tau^*(\omega).
	$
	Hence, for every $t\in[0,T]$,
	\[
	v\bigl(t\wedge\tau_N(\omega),\omega\bigr)\in H^s.
	\]
	Therefore, the deterministic approximation property in
	Lemma~\ref{Lemma-Jn} gives
	\[
	\left\|
	J_n v\bigl(t\wedge\tau_N(\omega),\omega\bigr)
	-
	v\bigl(t\wedge\tau_N(\omega),\omega\bigr)
	\right\|_{H^s}
	\longrightarrow0
	\]
	for every $t\in[0,T]$. Consequently,
	\begin{equation}\label{eq:Jn v to v}
		\left\|
		J_n v\bigl(t\wedge\tau_N(\omega),\omega\bigr)
		\right\|_{H^s}^2
		\longrightarrow
		\left\|
		v\bigl(t\wedge\tau_N(\omega),\omega\bigr)
		\right\|_{H^s}^2,
		\qquad
		t\in[0,T].
	\end{equation}
	Notice that \eqref{eq:Jn v to v} is now a deterministic statement
	on the fixed path $\omega$ and holds for every $t\in[0,T]$; no
	intersection over an uncountable collection of probability-one events
	is involved.
	
	We may therefore pass to the limit along
	$\{n_k\}_{k\ge1}$ in \eqref{eq:uniform-random-holder}. Using
	\eqref{eq:pathwise-uniform-holder-subsequence}, we obtain, for every
	$t,t'\in[0,T]$,
	\begin{align*}
		&
		\left|
		\left\|
		v\bigl(t\wedge\tau_N(\omega),\omega\bigr)
		\right\|_{H^s}^2
		-
		\left\|
		v\bigl(t'\wedge\tau_N(\omega),\omega\bigr)
		\right\|_{H^s}^2
		\right| \le
		\bigl(1+L_{N,T}(\omega)\bigr)
		|t-t'|^\gamma.
	\end{align*}
	Thus
	\[
	t\longmapsto
	\left\|v\bigl(t\wedge\tau_N(\omega),\omega\bigr)\right\|_{H^s}^2
	\]
	is H\"older continuous on $[0,T]$. Since
	$\Omega_{\mathrm{reg}}\cap\Omega_{N,T}$ has probability
	one, it follows that
	\begin{equation*} 
		t\longmapsto
		\|v(t\wedge\tau_N)\|_{H^s}
		\quad\text{is continuous on }[0,T]
		\quad\pas
	\end{equation*}
	
	Taking the intersection of the corresponding probability-one events
	over $N,T\in\mathbb N$, and using
	\eqref{eq:Hs-stopping-times-exhaust}, we conclude that
	\[
	t\longmapsto\|v(t)\|_{H^s}
	\]
	is continuous on $[0,\tau^*)$ almost surely.
	
	Finally, we combine this norm continuity with the weak continuity
	\eqref{eq:weak-Hs-continuity} to obtain
	$
	v\in C([0,\tau^*);H^s)
	$ $\pas$
	This completes the proof of Theorem~\ref{Thm-local-solution}.
	
	\begin{Remark}[On the role of
		Lemma~\ref{lem:fractional-Sobolev-Morrey-time}]
		One can use Fatou's lemma, \eqref{eq:Jn v to v} and \eqref{eq:Jn v time difference} to obtain
		\begin{equation*} 
			\mathbb{E} \left[ \Big|\| v(t \land \tau_N) \|^2_{H^s}
			- \| v(t' \land \tau_N)\|^2_{H^s} \Big|^4 \right] \le C_{N,T} | t - t' |^2,
			\quad t, t' \in[0,T], \ N \ge 1. 
		\end{equation*}
		Consequently, Kolmogorov's continuity criterion shows that
		$\| v(\cdot \land \tau_N)\|^2_{H^s}$ admits a (H\"older-)continuous 
		modification. This conclusion alone, however, does not imply that the original
		process $\|v(\cdot \land \tau_N)\|^2_{H^s}$ and its continuous modification are
		indistinguishable.  
		Although the weak $H^s$ continuity of
		$v(\cdot\wedge\tau_N)$ implies that $t\mapsto \|v(t\wedge\tau_N)\|_{H^s}$ is lower
		semicontinuous in $H^s$, this is not sufficient to identify the two
		processes.   The argument above therefore applies
		Lemma~\ref{lem:fractional-Sobolev-Morrey-time} pathwise to the
		approximating processes $\|J_n v(\cdot \land \tau_N)\|^2_{H^s}$, and combines the resulting
		H\"older estimates with the uniform moment bound for the random
		H\"older constants and a pathwise subsequence argument. Passing to
		the limit in these pathwise estimates shows directly that
		$
		t\longmapsto  
		\|v(t\wedge\tau_N)\|_{H^s}^2
		$
		itself has H\"older-continuous paths.  This avoids the preceding
		modification-identification issue.
	\end{Remark}

	\section{Nonlocal singularity formation under transport noise}\label{Section:BlowUp}
	
	In this section we prove Theorem~\ref{Thm-blowup} and  
	the conditional result on the blow-up rate, Theorem~\ref{Thm-blowup-rate}.
	
	We consider the
	transport-noise equation
	\begin{equation}\label{eq:sccf+transport noise}
		\d v+(\mathcal Hv)\,\partial_xv \d t
		=
		\sigma(x)\partial_xv\circ{\rm d}W(t),
		\quad
		v(0)=v_0,
		\quad t>0,
		\quad x\in\R.
	\end{equation}
	Throughout this section, Assumption~\ref{Hypo-sigma} is in force, and
	$(v,\tau^*)$ denotes the maximal $H^s$ solution with $s>7/2$.
	
	\begin{Remark}\label{Remark-localisation-regularity}
		We first record the localisation used in the stochastic pointwise
		computations. For $R>1+\|v_0\|_{H^s}$, let
		\begin{equation}\label{eq:tau-R-blowup}
			\tau_R
			\triangleq
			R\wedge
			\inf\bigl\{t\in[0,\tau^*):\|v(t)\|_{H^s}\ge R\bigr\}.
		\end{equation}
		Then $\tau_R\uparrow\tau^*$ almost surely as $R\to\infty$. On every
		interval $[0,T\wedge\tau_R]$, $T<\infty$, all spatial norms needed below
		are bounded. Indeed,
		\[
		v,\ \mathcal Hv\in H^s\hookrightarrow C_b^3,
		\qquad
		\Lambda v\in H^{s-1}\hookrightarrow C_b^2,
		\]
		and
		\[
		\Lambda\bigl((\mathcal Hv)v_x\bigr),\ \Lambda(\sigma v_x)
		\in H^{s-2}\hookrightarrow C_b^1,
		\qquad
		\Lambda\bigl[\sigma\partial_x(\sigma v_x)\bigr]
		\in H^{s-3}\hookrightarrow C_b.
		\]
		These are the spatial regularities required by the It\^o--Wentzell formula. All
		applications of the It\^o--Wentzell formula, the stochastic Fubini theorem,
		and pointwise It\^o calculus below are first made on
		$[0,T\wedge\tau_R]$. The bounds entering the blow-up comparison are
		independent of $R$, so we then let $R\to\infty$. We shall not repeat this
		localisation at each step.
	\end{Remark}
	\begin{Remark}\label{Remark-progressive measurability}
		Throughout this section, set
		\[
		\mathcal I_{\tau^*}
		\triangleq
		\{(\omega,t):0\le t<\tau^*(\omega)\}.
		\]
		The stochastic interval $\mathcal I_{\tau^*}$ is viewed as a progressive
		subset of $\Omega\times[0,\infty)$. Processes
		defined only on this interval are always taken in progressively measurable
		versions; whenever a globally defined process is needed, we use the extension
		by zero after $\tau^*$. As before, first entrance times below are understood with the
		convention $\inf\varnothing=\infty$.
	\end{Remark}

	\subsection{Stochastic characteristics and persistence of the maximum}
	
	For each $x\in\R$, we restate the Stratonovich characteristic flow given by \eqref{eq:char-flow}:
	\begin{equation*} 
		\d\phi(t,x)
		=
		\mathcal Hv(t,\phi(t,x)) \d t
		-
		\sigma(\phi(t,x))\circ{\rm d}W(t),
		\quad
		\phi(0,x)=x.
	\end{equation*}
	The localisation in Remark~\ref{Remark-localisation-regularity}
	allows us to choose a version of the characteristic flow for which all
	the properties below hold simultaneously up to the maximal existence
	time.
	
	\begin{Lemma}[Stochastic flow and related invariance]\label{lem:flow}
		On a common event of probability one,  the following properties hold:
		\begin{itemize}[leftmargin=0.79cm]
			\setlength\itemsep{0.2em}
			
			\item for every $t<\tau^*$, the map
			\[
			\phi(t,\cdot):\R\to\R,
			\]
			is a $C^1$-diffeomorphism. Moreover, $(t,x)\mapsto
			\phi(t,x)$ is continuous, and for every fixed $x$, the path
			$t\mapsto\phi(t,x)$ is a continuous semimartingale.
			
			\item along characteristics,
			\begin{equation}\label{eq:invariance}
				v(t,\phi(t,x))=v_0(x),
				\quad
				0\le t<\tau^*,
				\quad x\in\R.
			\end{equation}
			Equivalently, $v(t,\cdot)=v_0\circ\phi^{-1}(t,\cdot)$,
			$0\le t<\tau^*$.
		\end{itemize}

	\end{Lemma}
	
	\begin{proof}
		Let $(R_m)_{m\ge1}$ be a deterministic increasing sequence such that
		\[
		R_m>1+\|v_0\|_{H^s},
		\qquad
		R_m\uparrow\infty.
		\]
		For every $m\ge1$, define the stopped drift
		\[
		b_m(t,x)
		\triangleq
		\mathcal H
		v(t\wedge\tau_{R_m},x).
		\]
		By the definition of $\tau_{R_m}$ in \eqref{eq:tau-R-blowup} and the Sobolev embedding
		$H^s(\R)\hookrightarrow C_b^3(\R)$ for $s>7/2$, we have
		\[
		\sup_{t\ge0}
		\|b_m(t,\cdot)\|_{C_b^3}
		\lesssim
		\sup_{0\le r\le\tau_{R_m}}
		\|v(r)\|_{H^s}
		\le R_m.
		\]
		In particular, $b_m$ is globally Lipschitz in the spatial variable,
		with bounded spatial derivatives. The same properties hold for
		$\sigma$, since $\sigma\in\mathscr S(\R;\R)$.
		
		Consider the globally defined localised equation
		\begin{equation*}
			\mathrm{d}\phi_m(t,x)
			=
			b_m(t,\phi_m(t,x)) \d t
			-
			\sigma(\phi_m(t,x))\circ{\rm d}W(t),
			\qquad
			\phi_m(0,x)=x.
		\end{equation*}
		For every pair $(m,N)\in\mathbb N^2$, the standard $C^1$
		stochastic-flow theorem (see, for example,
		\cite[Theorem 4.6.5]{Kunita-1990-Book}), applied on the deterministic interval
		$[0,N]$, yields an event $\Omega_{m,N}$ of probability one on which
		$
		x\longmapsto\phi_m(t,x)
		$
		is a $C^1$-diffeomorphism for every $t\in[0,N]$, and
		$(t,x)\mapsto\phi_m(t,x)$ is jointly continuous.

		If $m,n\ge1$, then
		$\phi_m(\cdot,x)$ and $\phi_n(\cdot,x)$ solve the same equation on
		$
		0\le t\le\tau_{R_m}\wedge\tau_{R_n}.
		$
		Hence pathwise uniqueness gives, for every fixed $x\in\R$,
		\[
		\phi_m(t,x)=\phi_n(t,x),
		\qquad
		0\le t\le\tau_{R_m}\wedge\tau_{R_n},
		\]
		almost surely. We first take the intersection of the corresponding
		full-probability events over $
		m,n\in\mathbb N$ and $
		x\in\mathbb Q$,
		and then use the continuity in $x$ to extend this identity
		simultaneously to every $x\in\R$.
		
		We may therefore take
		\[
		\Omega_{*}
		\subset
		\bigcap_{m,N\ge1}\Omega_{m,N}
		\]
		to be a common event of probability one on which the above
		consistency holds and
		$
		\tau_{R_m}\uparrow\tau^*.
		$
		For $\omega\in\Omega_{*}$ and
		$0\le t<\tau^*(\omega)$, $\phi(\omega,t,x)$ is constructed by first  choosing $m$
		such that $t<\tau_{R_m}$, and then letting
		\[
		\phi(\omega,t,x)
		\triangleq
		\phi_m(\omega,t,x).
		\]
		The consistency just proved shows that this definition does not
		depend on the choice of $m$. Since $\phi_m(t,\cdot)$ is a
		$C^1$-diffeomorphism, the same is true of $\phi(t,\cdot)$.
		The joint continuity and the continuous local semimartingale
		property follow in the same way from the corresponding localised
		properties.

		Let
		\[
		\mathfrak J_\phi(t,x)
		\triangleq
		\partial_x\phi(t,x)
		\]
		be the spatial Jacobian. Fix $m,N\ge1$ and $x\in\mathbb R$.
		Differentiating \eqref{eq:char-flow} with respect to the initial
		point $x$ on the localised interval
		$[0,N\wedge\tau_{R_m}]$ gives
		\[
		\d\mathfrak J_\phi(t,x)
		=
		(\mathcal Hv)_x(t,\phi(t,x))
		\mathfrak J_\phi(t,x)\d t
		-
		\sigma'(\phi(t,x))
		\mathfrak J_\phi(t,x)\circ{\rm d}W(t).
		\]
		This is a scalar linear Stratonovich equation with
		$\mathfrak J_\phi(0,x)=1$. Its explicit solution is
		\begin{equation*}
			\mathfrak J_\phi(t,x)
			=
			\exp\!\Bigg(
			\int_0^t
			(\mathcal Hv)_x(t',\phi(t',x))\d t'
			-
			\int_0^t
			\sigma'(\phi(t',x))\circ{\rm d}W(t')
			\Bigg)
			>0,
			\qquad 0\le t\le N\wedge\tau_{R_m}.
		\end{equation*}
		Since $m$ and $N$ are arbitrary and
		$\tau_{R_m}\uparrow\tau^*$, it follows that
		$\mathfrak J_\phi(t,x)>0$ for every $t<\tau^*$ and every
		$x\in\mathbb R$. Thus the flow is orientation-preserving on its
		entire interval of existence. Together with the construction and
		consistency of the localised flows, this shows that
		$\phi(t,\cdot):\mathbb R\to\mathbb R$ is a
		$C^1$-diffeomorphism for every $t<\tau^*$.
		
		Finally, after the localisation
		described above, the It\^o--Wentzell formula gives
		\begin{align*}
			\d v(t,\phi(t,x))
			=
			[\d v(t,\cdot)]_{x=\phi(t,x)}
			+
			v_x(t,\phi(t,x))\circ\d\phi(t,x).
		\end{align*}
		By \eqref{eq:sccf+transport noise} and \eqref{eq:char-flow}, we have
		\begin{align*}
			\d v(t,\phi(t,x))=&-\mathcal Hv(t,\phi(t,x)) v_x(t,\phi(t,x)) \d t
			+
			\sigma(\phi(t,x))v_x(t,\phi(t,x))\circ{\rm d}W(t)\\
			&
			+v_x(t,\phi(t,x))
			\left[
			\mathcal Hv(t,\phi(t,x)) \d t
			-
			\sigma(\phi(t,x))\circ{\rm d}W(t)
			\right]
			=0.
		\end{align*}
		Hence $v(t,\phi(t,x))=v_0(x)$ for each fixed $x$, simultaneously for all $t<\tau^*$.
		Taking a common full-probability event first for rational
		$x$ and rational $t$, and then using the joint continuity of $v$ and
		$\phi$, yields \eqref{eq:invariance} simultaneously for all
		$x\in\R$ and $t<\tau^*$.
	\end{proof}

	\begin{Remark}
		Since $\phi(t,\cdot)$ is a bijection, \eqref{eq:invariance}
		implies
		\begin{equation}\label{eq:Linfty-invariance}
			\|v(t)\|_{L^\infty}=\|v_0\|_{L^\infty}\triangleq V_{\infty},
			\quad
			0\le t<\tau^*.
		\end{equation}
	\end{Remark}
	
	Let $x_0$ be the global maximum point appearing in  Theorem~\ref{Thm-blowup}, and 
recall \eqref{eq:z=phi(x0)}:
	\begin{equation*} 
	z_0(t)\triangleq\phi(t,x_0),
	\quad
	0\le t<\tau^*.
\end{equation*}
	By Lemma~\ref{lem:flow},
	\[
	v(t,z_0(t))
	=v_0(x_0)
	=\sup_{x\in\R}v_0(x)
	=\sup_{x\in\R}v(t,x).
	\]
	Thus $z_0(t)$ is a global maximum point of $v(t,\cdot)$ and
	\begin{equation*}
		v_x(t,z_0(t))=0,
		\quad
		0\le t<\tau^*.
	\end{equation*}
	
	We recall  the quantity that drives the blow-up mechanism introduced in \eqref{eq:M(t) define}:
	\begin{equation*} 
		M(t)\triangleq\Lambda v(t,z_0(t)),
		\quad
		0\le t<\tau^*.
	\end{equation*}
	Set
	\begin{equation*}
		M_0\triangleq M(0)=\Lambda v_0(x_0).
	\end{equation*}
	
	\begin{Lemma}[Strict positivity of $M$]\label{lem:M-positive}
		Suppose that $M_0>0$. Then
		\begin{equation*}
			M(t)>0,
			\quad
			0\le t<\tau^*,
			\quad\text{almost surely}.
		\end{equation*}
	\end{Lemma}
	
	\begin{proof}
		Fix $\omega$ in the common full-probability event constructed above and let $t<\tau^*(\omega)$. Since $z_0(t)$ is a maximum point and
		$v_x(t,z_0(t))=0$, the singular-integral representation of $\Lambda$
		in \eqref{eq: Hilbert Lambda kernel} is an absolutely convergent integral at $z_0(t)$, with no principal-value interpretation required. Then we arrive at
		\begin{equation}\label{eq:M-positive-integral}
			M(t)
			=
			\frac1\pi
			\int_\R
			\frac{v(t,z_0(t))-v(t,y)}{(z_0(t)-y)^2} \d y
			\ge0.
		\end{equation}
		The assumption $M_0>0$ implies that $v_0$ is not constant. By
		\eqref{eq:invariance}, $v(t,\cdot)=v_0\circ\phi^{-1}(t,\cdot)$ is
		also not constant. Hence there exists $y_1\in\R$ such that
		$v(t,y_1)<v(t,z_0(t))$. By continuity, this strict inequality holds
		on a set of positive measure. The integral in
		\eqref{eq:M-positive-integral} is absolutely convergent: near
		$y=z_0(t)$ this follows from $v_x(t,z_0(t))=0$ and
		$v(t,\cdot)\in C_b^2$, while at infinity it follows from the boundedness
		of $v$. Therefore the nonnegative integral is strictly positive.
	\end{proof}

	\subsection{Evolution of the nonlocal quantity and related estimates}\label{Section:blow-up-nonlocal-estimates}
	
	Set
	\[
	u\triangleq\Lambda v.
	\]
	Applying $\Lambda$ to \eqref{eq:sccf+transport noise} gives, in
	Stratonovich form,
	\begin{equation*}
		\d u
		=
		-\Lambda\bigl((\mathcal Hv)v_x\bigr) \d t
		+
		\Lambda(\sigma v_x)\circ{\rm d}W(t).
	\end{equation*}
	Applying the It\^o--Wentzell formula to
	$M(t)=u(t,z_0(t))$, and using \eqref{eq:char-flow}, yields
	\begin{align}
		\mathrm{d} M(t)
		=
		\left[
		-\Lambda\bigl((\mathcal Hv)v_x\bigr)
		+(\mathcal Hv)\Lambda v_x
		\right](t,z_0(t)) \d t +
		\left[
		\Lambda(\sigma v_x)-\sigma\Lambda v_x
		\right](t,z_0(t))\circ{\rm d}W(t).
		\label{eq:M-Strat-raw}
	\end{align}
	
	For the drift, the identities
	\[
	v_x=\mathcal H\Lambda v,
	\qquad
	\Lambda v_x=-\mathcal Hv_{xx},
	\]
	and Lemma~\ref{lem:identity}, evaluated at the maximum point $z_0(t)$,
	give
	\begin{equation}\label{eq:drift-simplified}
		\left[
		-\Lambda\bigl((\mathcal Hv)v_x\bigr)
		+(\mathcal Hv)\Lambda v_x
		\right](t,z_0(t))
		=
		\frac12M(t)^2
		+
		\mathcal E[v(t,\cdot)](z_0(t)),
	\end{equation}
	where $\mathcal E\ge0$ is defined in \eqref{eq:E-functional}.
	
	For the noise coefficient, define
	\begin{equation}\label{eq:Gamma-def}
		\Gamma(t)
		\triangleq
		[\Lambda,\sigma]v_x(t,z_0(t)).
	\end{equation}
	Using the same truncation in the two principal-value integrals, one has,
	for every $z\in\R$,
	\begin{align*}
		[\Lambda,\sigma]v_x(t,z)
		=
		\frac1\pi\,\pv\int_\R
		\frac{\sigma(z)-\sigma(y)}{(z-y)^2}v_x(t,y) \d y.
	\end{align*}
	With the kernel $\mathsf K_\sigma$ from Lemma~\ref{Lem:sigma-kernel},
	\[
	\frac{\sigma(z)-\sigma(y)}{(z-y)^2}
	=
	\frac{\sigma'(z)}{z-y}+\mathsf K_\sigma(z,y),
	\]
	and therefore
	\begin{equation}\label{eq:Gamma-decomp}
		\Gamma(t)
		=
		\sigma'(z_0(t))M(t)
		+
		\mathcal R_\sigma[v](t,z_0(t)),\quad \text{where}\quad
		\mathcal R_\sigma[v](t,z)
		\triangleq
		\frac1\pi\int_\R \mathsf K_\sigma(z,y)v_x(t,y) \d y.
	\end{equation}
	Indeed, with the convention used in \eqref{eq: Hilbert Lambda kernel},
	\[
	\frac1\pi\,\pv\int_\R\frac{v_x(t,y)}{z-y} \d y
	=-\mathcal Hv_x(t,z)
	=\Lambda v(t,z).
	\]
	Note that $\mathsf K_\sigma(z,\cdot),v_x\in L^2$ implies that $\mathsf K_\sigma(z,\cdot)v_x(\cdot)\in L^1$.
	Integrating $\mathcal R_\sigma[v](t,z)$ by parts on a finite interval and then
	letting its endpoints tend to infinity gives
	\begin{equation}\label{eq:R-IBP}
		\mathcal R_\sigma[v](t,z)
		=
		-\frac1\pi
		\int_\R
		\partial_y\mathsf K_\sigma(z,y)v(t,y) \d y.
	\end{equation}
	Indeed, for each fixed $z$, $\mathsf K_\sigma(z,y) = \mathcal{O}(\vert{}y\vert{}^{-1})$ as $\vert{}y\vert{}\to\infty$ (since $\sigma \in \mathscr{S}$). This fact, together with $v(t,\cdot)\in L^\infty$, ensures the boundary terms vanish.

	By
	Lemma~\ref{Lem:sigma-kernel} and \eqref{eq:Linfty-invariance},
	\begin{equation}\label{eq:R-L-infty}
		|\mathcal R_\sigma[v](t,z)|
		\le
		C_\sigma V_{\infty},
		\quad
		C_\sigma
		\triangleq
		\frac1\pi
		\sup_{z\in\R}
		\|\partial_y\mathsf K_\sigma(z,\cdot)\|_{L^1}.
	\end{equation}
	Consequently, after increasing the constant if necessary, there exists
	$c_\sigma>0$, depending only on $\sigma$, such that
	\begin{equation}\label{eq:Gamma-bound}
		|\Gamma(t)|
		\le
		c_\sigma\bigl(M(t)+V_{\infty}\bigr),
		\quad
		0\le t<\tau^*.
	\end{equation}
	Here we used the strict positivity of $M$ from Lemma~\ref{lem:M-positive}.
	Combining \eqref{eq:M-Strat-raw}, \eqref{eq:drift-simplified}, and
	\eqref{eq:Gamma-def}, we obtain
	\begin{equation}\label{eq:M-Strat}
		\d M(t)
		=
		\left[
		\frac12M(t)^2
		+
		\mathcal E[v(t,\cdot)](z_0(t))
		\right]\d t
		+
		\Gamma(t)\circ{\rm d}W(t).
	\end{equation}

	We next estimate the It\^o correction generated by $\Gamma$. Note that
	\[
	\Gamma(t)\circ{\rm d}W(t)
	=
	\Gamma(t) \d W(t)
	+
	\frac12 \d\langle\Gamma,W\rangle_t.
	\]
	Write
	\begin{equation*}
		\Gamma(t)=\Gamma_{\mathrm{loc}}(t)+\Gamma_{\mathrm{rem}}(t),
		\quad
		\Gamma_{\mathrm{loc}}(t)\triangleq\sigma'(z_0(t))M(t),
		\quad
		\Gamma_{\mathrm{rem}}(t)\triangleq \mathcal R_\sigma[v](t,z_0(t)).
	\end{equation*}
	The It\^o form of \eqref{eq:char-flow} is
	\begin{equation}\label{eq:z0-Ito}
		\d z_0(t)
		=
		\left[
		\mathcal Hv(t,z_0(t))
		+
		\frac12\sigma(z_0(t))\sigma'(z_0(t))
		\right]\d t
		-
		\sigma(z_0(t)) \d W(t),
	\end{equation}
	and hence
	\begin{equation*}
		\d\langle z_0,W\rangle_t
		=
		-\sigma(z_0(t)) \d t.
	\end{equation*}
	The martingale coefficient of $M$ in the It\^o form of
	\eqref{eq:M-Strat} is $\Gamma$, so
	\begin{equation*}
		\d\langle M,W\rangle_t
		=
		\Gamma(t) \d t.
	\end{equation*}
	It follows that
	\begin{align}
		\d\langle\Gamma_{\mathrm{loc}},W\rangle_t
		=
		\left[
		\sigma'(z_0(t))\Gamma(t)
		-
		\sigma(z_0(t))\sigma''(z_0(t))M(t)
		\right]\d t.
		\label{eq:Gamma-loc-W-quad}
	\end{align}
	
	To compute the bracket of $\Gamma_{\mathrm{rem}}$ with $W$, we first verify that $\Gamma_{\mathrm{rem}}$ is
	a continuous local semimartingale. By \eqref{eq:R-IBP},
	\[
	\Gamma_{\mathrm{rem}}(t)
	=
	-\frac1\pi
	\int_\R
	\partial_y\mathsf K_\sigma(z_0(t),y)v(t,y)\d y.
	\]
	
	As in Remark~\ref{Remark-localisation-regularity}, we will first carry out the analysis on
	$[0,T\wedge\tau_R]$,
	where $\tau_R$ is given in \eqref{eq:tau-R-blowup}. Since $R$ and
	$T$ are arbitrary and $\tau_R\uparrow\tau^*$, the resulting
	identities hold locally on $[0,\tau^*)$.
	
	For notational convenience, we set
	\begin{equation}\label{eq:q=Ky}
		q(z,y)
	\triangleq
	\partial_y\mathsf K_\sigma(z,y),
	\qquad
	q_z
	\triangleq
	\partial_zq,
	\qquad
	q_{zz}
	\triangleq
	\partial_z^2q
	\end{equation}
	and
	define the characteristic drift and the
	field drift, respectively, by
	\begin{align*}
		\mathfrak a(t)
		\triangleq
		\mathcal Hv(t,z_0(t))
		+
		\frac12
		\sigma(z_0(t))\sigma'(z_0(t)),
		\quad
		\mathfrak b(t,y)
		\triangleq
		-(\mathcal Hv)(t,y)v_x(t,y)
		+
		\frac12
		\sigma(y)
		\partial_y\bigl(\sigma(y)v_x(t,y)\bigr).
	\end{align*}
	Then \eqref{eq:z0-Ito} and the It\^o form of \eqref{eq:sccf+transport noise} can be
	written as
	\begin{align*}
		\mathrm{d}z_0(t)
		=
		\mathfrak a(t) \d t
		-
		\sigma(z_0(t))\d W(t),
	\end{align*}
	and
	\begin{align*}
		\mathrm{d}v(t,y)
		=
		\mathfrak b(t,y) \d t
		+
		\sigma(y)v_x(t,y)\d W(t),
	\end{align*}
	respectively.
	For every fixed $y\in\R$, It\^o's formula gives
	\begin{align*}
		\mathrm{d}q(z_0(t),y)
		=
		\Bigl[
		\mathfrak a(t)
		q_z(z_0(t),y)
		+
		\frac12
		\sigma(z_0(t))^2
		q_{zz}(z_0(t),y)
		\Bigr]\d t-
		\sigma(z_0(t))
		q_z(z_0(t),y)\d W(t).
	\end{align*}
	Hence, by the product It\^o formula,
	\begin{equation}\label{eq:Gamma-rem-integrand-Ito}
		\mathrm{d}\bigl(q(z_0(t),y)v(t,y)\bigr)
		=
		\mathscr D(t,y) \d t
		+
		\mathscr G(t,y)\d W(t),
	\end{equation}
	where
	\begin{align}
		\mathscr D(t,y)
		\triangleq \ &
		q(z_0(t),y)
		\mathfrak b(t,y)
		+
		\mathfrak a(t)
		q_z(z_0(t),y)v(t,y)  \notag\\
		&+
		\frac12
		\sigma(z_0(t))^2
		q_{zz}(z_0(t),y)v(t,y) -
		\sigma(z_0(t))\sigma(y)
		q_z(z_0(t),y)v_x(t,y),
		\label{eq:Gamma-rem-integrand-drift}
	\end{align}
	\begin{align}
		\mathscr G(t,y)
		\triangleq \ &
		-\sigma(z_0(t))
		q_z(z_0(t),y)v(t,y) +
		q(z_0(t),y)\sigma(y)v_x(t,y).
		\label{eq:Gamma-rem-integrand-martingale}
	\end{align}
	Here the term containing $q_{zz}$ is the second-order It\^o
	correction generated by the stochastic motion of $z_0$, whereas
	the last term in \eqref{eq:Gamma-rem-integrand-drift} is the
	quadratic-covariation contribution between
	$q(z_0(t),y)$ and $v(t,y)$.
	
	By Lemma~\ref{Lem:sigma-kernel}, applied with
	\[
	(a,b)=(0,1),\qquad
	(a,b)=(1,1),\qquad
	(a,b)=(2,1),
	\]
	we have
	\begin{equation}\label{eq:q-kernel-bounds-for-Gamma-rem}
		\sup_{z\in\R}
		\left(
		\|q(z,\cdot)\|_{L_y^1}
		+
		\|q_z(z,\cdot)\|_{L_y^1}
		+
		\|q_{zz}(z,\cdot)\|_{L_y^1}
		\right)
		<\infty.
	\end{equation}
	On $[0,T\land \tau_R]$, the quantities
	\[
	\|v(t)\|_{L^\infty},
	\|v_x(t)\|_{L^\infty},
	\|\mathfrak b(t,\cdot)\|_{L^\infty},
	|\mathfrak a(t)|
	\]
	are bounded by the localised Sobolev estimates in
	Remark~\ref{Remark-localisation-regularity}. Since $\sigma\in \mathscr{S}(\R; \R)$, it follows from
	\eqref{eq:Gamma-rem-integrand-drift},
	\eqref{eq:Gamma-rem-integrand-martingale}, and
	\eqref{eq:q-kernel-bounds-for-Gamma-rem} that, for some deterministic
	constant $C_{R,\sigma}<\infty$,
	\begin{equation*}
		\int_\R
		|\mathscr D(t,y)|\d y
		+
		\left(
		\int_\R
		|\mathscr G(t,y)|\d y
		\right)^2
		\le
		C_{R,\sigma},
		\qquad
		0\le t\le T\land \tau_R.
	\end{equation*}
	Thus deterministic Fubini applies to the finite-variation part of
	\eqref{eq:Gamma-rem-integrand-Ito}, while the stochastic Fubini theorem applies to its
	martingale part. Alternatively, the stochastic interchange can be
	justified by first truncating the spatial integral and then using
	the It\^o isometry and dominated convergence. We obtain
	\begin{align}
		\Gamma_{\mathrm{rem}}(t\wedge T\land \tau_R )
		=
		\Gamma_{\mathrm{rem}}(0)
		-
		\frac1\pi
		\int_0^{t\wedge T\land \tau_R }
		\int_\R
		\mathscr D(r,y)\d y\d r -
		\frac1\pi
		\int_0^{t\wedge T\land \tau_R }
		\left(
		\int_\R
		\mathscr G(r,y)\d y
		\right)\d W(r).
		\label{eq:Gamma-rem-semimartingale-decomposition}
	\end{align}
	In particular, after removing the localisation, $\Gamma_{\mathrm{rem}}$ is a continuous
	local semimartingale on $[0,\tau^*)$.
	
	Only the martingale coefficient in
	\eqref{eq:Gamma-rem-semimartingale-decomposition} contributes to the
	quadratic covariation with $W$. Therefore,
	$\d t\otimes\d\mathbb P$-almost everywhere on
	$\{t<\tau^*\}$,
	\begin{align*}
		\frac{\d}{\d t}\langle\Gamma_{\mathrm{rem}},W\rangle_t
		= 
		-\frac1\pi
		\int_\R
		\mathscr G(t,y)\d y 
		= 
		\frac1\pi
		\int_\R
		\sigma(z_0(t))
		q_z(z_0(t),y)v(t,y)\d y-
		\frac1\pi
		\int_\R
		q(z_0(t),y)\sigma(y)v_x(t,y)\d y.
	\end{align*}

	For any fixed $z\in\mathbb R$, we can infer from the estimate $\|\sigma(\cdot)q(z,\cdot)\|_{L^1}\le\|\sigma\|_{L^\infty}\|q(z,\cdot)\|_{L^1}$, \eqref{eq:q=Ky}, and \eqref{eq:weighted-kernel-derivative-L1-bound} that $\sigma(\cdot)q(z,\cdot)\in W^{1,1}(\mathbb R)$.
	Since
	$v(t,\cdot)\in W^{1,\infty}(\R)$, integration by parts on $\R$
	yields
	\[
	-\int_\R
	q(z_0(t),y)\sigma(y)v_x(t,y)\d y
	=
	\int_\R
	\partial_y
	\bigl[
	\sigma(y)q(z_0(t),y)
	\bigr]
	v(t,y)\d y.
	\]
	Recalling that
	\[
	q(z,y)
	=
	\partial_y\mathsf K_\sigma(z,y),
	\qquad
	q_z(z,y)
	=
	\partial_z\partial_y\mathsf K_\sigma(z,y),
	\]
	we conclude that
	\begin{align}
		\frac{\d}{\d t}\langle\Gamma_{\mathrm{rem}},W\rangle_t
		=
		\frac1\pi
		\int_\R
		\Bigl[
		\sigma(z_0(t))
		\partial_z\partial_y\mathsf K_\sigma(z_0(t),y)
		+
		\partial_y
		\bigl(
		\sigma(y)
		\partial_y\mathsf K_\sigma(z_0(t),y)
		\bigr)
		\Bigr]
		v(t,y)\d y.
		\label{eq:Gamma-rem-W-quad-final}
	\end{align}
	
	Combining \eqref{eq:Gamma-loc-W-quad} and
	\eqref{eq:Gamma-rem-W-quad-final}, define, for $0\le t<\tau^*$,
	\begin{equation}\label{eq:kappa-rate-define}
		\kappa_{\Gamma}(t)
		\triangleq
		\frac{\d}{\d t}\langle\Gamma,W\rangle_t=
		\sigma'(z_0(t))\Gamma(t)
		-
		\sigma(z_0(t))\sigma''(z_0(t))M(t)
		+
		\Xi(t),
	\end{equation}
	where
	\begin{align*}
		\Xi(t)
		\triangleq
		\frac1\pi
		\int_\R
		\Bigl[
		\sigma(z_0(t))
		\partial_z\partial_y\mathsf K_\sigma(z_0(t),y)
		+
		\partial_y\bigl(
		\sigma(y)\partial_y\mathsf K_\sigma(z_0(t),y)
		\bigr)
		\Bigr]
		v(t,y)\,\d y.
	\end{align*}
	By Lemma~\ref{Lem:sigma-kernel} and \eqref{eq:Linfty-invariance},
	\begin{equation}\label{eq:Xi-rate-bound}
		|\Xi(t)|\le C_{\sigma,1}V_{\infty}
	\end{equation}
	for a constant $C_{\sigma,1}$ depending only on $\sigma$.
	Together with \eqref{eq:Gamma-bound}, this implies that there exist
	constants $C_{\sigma,2},C_{\sigma,3}>0$, depending only on $\sigma$, such
	that
	\begin{equation}\label{eq:Gamma-W-quad-estimate}
		\kappa_\Gamma(t)
		\ge
		-C_{\sigma,2}M(t)-C_{\sigma,3}V_{\infty},
		\qquad 0\le t<\tau^*.
	\end{equation}
	
	\begin{Remark}\label{Remark-kappa-measurability}
		If
		\[
		A_\sigma(z,y)
		\triangleq
		\sigma(z)\partial_z\partial_y\mathsf K_\sigma(z,y)
		+
		\partial_y\bigl(\sigma(y)\partial_y\mathsf K_\sigma(z,y)\bigr),
		\]
		then Lemma~\ref{Lem:sigma-kernel}, together with the same estimates after
		one additional $z$-derivative, gives
		\[
		\sup_{z\in\R}
		\left(
		\|A_\sigma(z,\cdot)\|_{L^1}
		+
		\|\partial_zA_\sigma(z,\cdot)\|_{L^1}
		\right)<\infty.
		\]
		Consequently, the map
		$
		(z,f)\longmapsto
		\frac1\pi\int_\R A_\sigma(z,y)f(y)\,\d y
		$
		is continuous from $\R\times L^\infty(\R)$ to $\R$. Since $z_0$ has
		continuous adapted paths and
		$v\in C([0,\tau^*);H^s)\hookrightarrow C([0,\tau^*);L^\infty)$,
		the processes $\Xi$ and $\kappa_\Gamma$ are continuous and adapted, hence
		progressively measurable and locally integrable. 
			Moreover, for every
			$R,T>0$ and $t\ge0$, we have
			\begin{equation*} 
					\langle\Gamma,W\rangle_{t\wedge T\wedge\tau_R}
					=
					\int_0^{t\wedge T\wedge\tau_R}
					\kappa_\Gamma(r)\,\d r.
				\end{equation*}
			Thus \eqref{eq:kappa-rate-define} is a fixed progressively measurable
			version of the bracket density, rather than merely an unspecified
			Radon--Nikod\'ym representative.
	\end{Remark}

	Now we summarise the preceding calculation in the form used below.  
	\begin{Lemma}\label{lem:dM lemma}
		There exist constants $c_\sigma,c_1,c_2>0$, depending only on
		$\sigma$ through its $W^{2,\infty}$ norm and the kernel bounds in
		Lemma~\ref{Lem:sigma-kernel}, and a progressively measurable, locally integrable process
		$\mathscr R_M$, such that, for $0\le t<\tau^*$,
		\begin{equation}\label{eq:M-Ito-exact}
			\mathrm{d}M(t)
			=
			\left[
			\frac12M(t)^2+\mathscr R_M(t)
			\right]\d t
			+
			\Gamma(t) \d W(t),
		\end{equation}
		with
		\begin{equation}\label{eq:M-Ito-drift-bound}
			\mathscr R_M(t)\ge -c_1M(t)-c_2V_{\infty},
		\end{equation}
		and
		\begin{equation}\label{eq:Gamma-bound-summary}
			|\Gamma(t)|\le c_\sigma\bigl(M(t)+V_{\infty}\bigr).
		\end{equation}
	\end{Lemma}
	
	\begin{proof}
		Set
		\begin{equation}\label{eq:energy-rate-definition}
			\mathcal E_M(t)
			\triangleq
			\mathcal E[v(t,\cdot)](z_0(t)).
		\end{equation}
		By \eqref{eq:drift-simplified},
		\[
		\mathcal E_M(t)
		=
		\left[
		-\Lambda\bigl((\mathcal Hv)v_x\bigr)
		+(\mathcal Hv)\Lambda v_x
		\right](t,z_0(t))
		-
		\frac12M(t)^2.
		\]
		The map
		\[
		v\longmapsto
		-\Lambda\bigl((\mathcal Hv)v_x\bigr)
		+(\mathcal Hv)\Lambda v_x
		\]
		is continuous from $H^s$ to $H^{s-2}\hookrightarrow C_b(\R)$.
		Since $v\in C([0,\tau^*);H^s)$ and $z_0$ and $M$ have continuous
		adapted paths, $\mathcal E_M$ is continuous and adapted, hence
		progressively measurable and locally integrable.
		
		Define
		\[
		\mathscr R_M(t)
		\triangleq
		\mathcal E_M(t)+\frac12\kappa_\Gamma(t).
		\]
		By \eqref{eq:kappa-rate-define}, the Stratonovich equation
		\eqref{eq:M-Strat} is equivalent to \eqref{eq:M-Ito-exact} with this
		choice of $\mathscr R_M$. Since $\mathcal E_M\ge0$, estimate
		\eqref{eq:Gamma-W-quad-estimate} yields
		\eqref{eq:M-Ito-drift-bound} after renaming the constants. Estimate
		\eqref{eq:Gamma-bound-summary} is \eqref{eq:Gamma-bound}.
	\end{proof}
	
	\subsection{The reciprocal process}
	
	Set
	\begin{equation*}
		Y(t)\triangleq\frac1{M(t)},
		\quad
		Y_0\triangleq Y(0)=\frac1{M_0}.
	\end{equation*}
	By Lemma~\ref{lem:M-positive}, $Y$ is a well-defined, continuous,
	strictly positive semimartingale on the entire stochastic interval
	$[0,\tau^*)$.
	
	Applying It\^o's formula to $Y=M^{-1}$, first after stopping $M$ in a
	compact subinterval of $(0,\infty)$ and then removing this auxiliary
	localisation, and using \eqref{eq:M-Ito-exact}, we obtain
	\begin{equation}\label{eq:Y-Ito-exact}
		\mathrm{d}Y(t)
		=
		b_Y(t) \d t
		+
		\nu_Y(t)Y(t) \d W(t),
	\end{equation}
	where
	\begin{equation*}
		b_Y(t)
		\triangleq
		-\frac12
		-
		\frac{\mathscr R_M(t)}{M(t)^2}
		+
		\frac{\Gamma(t)^2}{M(t)^3},\quad
		\nu_Y(t)
		\triangleq
		-\frac{\Gamma(t)}{M(t)}.
	\end{equation*}
	By \eqref{eq:M-Ito-drift-bound},
	\[
	-\frac{\mathscr R_M(t)}{M(t)^2}
	\le
	c_1Y(t)+c_2V_{\infty}Y(t)^2.
	\]
	Moreover, \eqref{eq:Gamma-bound-summary} gives
	\begin{equation}\label{eq:nuY-bound}
		|\nu_Y(t)|
		\le
		c_\sigma\bigl(1+V_{\infty}Y(t)\bigr)
	\end{equation}
	and
	\[
	\frac{\Gamma(t)^2}{M(t)^3}
	=
	Y(t)\left(\frac{\Gamma(t)}{M(t)}\right)^2
	\le
	2c_\sigma^2Y(t)
	+
	2c_\sigma^2V_{\infty}^2Y(t)^3.
	\]
	Thus, with
	\begin{equation*}
		a_1\triangleq c_1+2c_\sigma^2,
	\end{equation*}
	we have the drift estimate
	\begin{equation}\label{eq:dY-bound-full}
		b_Y(t)
		\le
		-\frac12
		+
		a_1Y(t)
		+
		c_2V_{\infty}Y(t)^2
		+
		2c_\sigma^2V_{\infty}^2Y(t)^3.
	\end{equation}
	
	Choose $Y_0$ sufficiently small so that
	\begin{equation}\label{eq:Y0-condition}
		2a_1Y_0
		+
		4c_2V_{\infty}Y_0^2
		+
		16c_\sigma^2V_{\infty}^2Y_0^3
		\le
		\frac14.
	\end{equation}
	Using the progressive-extension convention stated in Remark \ref{Remark-progressive measurability}, set
	\[
	\mathcal O_Y
	\triangleq
	\mathcal I_{\tau^*}
	\cap
	\bigl\{(\omega,t):Y(t,\omega)\ge2Y_0\bigr\}.
	\]
	This is a progressive subset of $\Omega\times[0,\infty)$. Hence, by the
	d\'ebut theorem, its first entrance time
	\[
	\widehat\tau_Y
	\triangleq
	\inf\bigl\{t\ge0:(\omega,t)\in\mathcal O_Y\bigr\}
	\]
	is a stopping time. We set 
	\begin{equation*}
		\tau_Y
		\triangleq
		\tau^*\wedge\widehat\tau_Y.
	\end{equation*}
	Moreover,
	\begin{equation}\label{eq:Y-before-tauY}
		0<Y(t)<2Y_0,
		\quad
		0\le t<\tau_Y,
	\end{equation}
	and, if $\tau_Y<\tau^*$, continuity gives
	\begin{equation}\label{eq:Y-at-tauY}
		Y(\tau_Y)=2Y_0.
	\end{equation}
	By \eqref{eq:Y0-condition}, \eqref{eq:dY-bound-full}, and
	\eqref{eq:Y-before-tauY},
	\begin{equation}\label{eq:bY-negative}
		b_Y(t)\le-\frac14,
		\quad
		0\le t<\tau_Y.
	\end{equation}
	
	Define the globally stopped integrand
	\begin{equation*}
		\mathfrak h_Y(t)
		\triangleq
		\begin{cases}
			\nu_Y(t)Y(t),&0\le t<\tau_Y,\\
			0,&t\ge\tau_Y.
		\end{cases}
	\end{equation*}
	The process $\mathfrak h_Y$ is progressively measurable. By
	\eqref{eq:nuY-bound} and \eqref{eq:Y-before-tauY}, it is bounded by a
	deterministic constant. Hence the stochastic integral below is a
	square-integrable continuous martingale on every finite time interval:
	\begin{equation*}
		\mathcal N(t)
		\triangleq
		\int_0^t \mathfrak h_Y(r) \d W(r),
		\quad t\ge0.
	\end{equation*}
	The process $\mathcal N$ is defined on the whole probability space and
	for all $t\ge0$, including on paths for which $\tau^*$ is finite. For
	$t<\tau_Y$, integrating \eqref{eq:Y-Ito-exact} and using
	\eqref{eq:bY-negative} gives
	\begin{equation}\label{eq:Y-ineq-integrated}
		Y(t)
		\le
		Y_0-\frac14t+\mathcal N(t).
	\end{equation}
	If $\tau_Y<\tau^*$, the same inequality holds at $t=\tau_Y$ by taking
	$t\uparrow\tau_Y$ and using continuity.
	
	\subsection{Conclusion of the proof of Theorem~\ref{Thm-blowup}}
	
	On $[0,\tau_Y)$, estimates \eqref{eq:nuY-bound} and
	\eqref{eq:Y-before-tauY} imply
	\begin{equation*}
		|\nu_Y(t)|
		\le
		c_\sigma\bigl(1+2V_{\infty}Y_0\bigr)
		\triangleq
		\overline\nu_Y.
	\end{equation*}
	Set
	\begin{equation*}
		T^*\triangleq8Y_0=\frac8{M_0}.
	\end{equation*}
	Since $\mathfrak h_Y=0$ after $\tau_Y$, we have the deterministic quadratic-variation
	bound
	\begin{align*}
		\langle\mathcal N\rangle_{T^*}
		\le
		\overline\nu_Y^{\,2}(2Y_0)^2T^*
		=
		32\overline\nu_Y^{\,2}Y_0^3.
	\end{align*}
	Consequently, Markov's inequality, Doob's $L^2$ maximal inequality, and
	It\^o's isometry give
	\begin{align*}
		\mathbb P\left(
		\sup_{0\le t\le T^*}|\mathcal N(t)|>\frac{Y_0}{2}
		\right)\le
		\frac4{Y_0^2}
		\mathbb E\left[
		\sup_{0\le t\le T^*}|\mathcal N(t)|^2
		\right]
		\le
		\frac{16}{Y_0^2}
		\mathbb E\langle\mathcal N\rangle_{T^*}
		\le
		512\overline\nu_Y^{\,2}Y_0.
	\end{align*}
	Define the event
	\begin{equation*}
		\Omega_0
		\triangleq
		\left\{
		\sup_{0\le t\le T^*}|\mathcal N(t)|
		\le
		\frac{Y_0}{2}
		\right\}.
	\end{equation*}
	Then
	\begin{equation}\label{eq:good-probability}
		\mathbb P(\Omega_0)
		\ge
		1-512\overline\nu_Y^{\,2}Y_0.
	\end{equation}
	
	We claim that
	\begin{equation}\label{eq:Omega0-implies-blowup}
		\Omega_0
		\subset
		\{\tau^*\le T^*\}.
	\end{equation}
	Suppose, to the contrary, that a path belongs to $\Omega_0$ and satisfies
	$\tau^*>T^*$. We first show that $\tau_Y>T^*$. If
	$\tau_Y\le T^*$, then $\tau_Y<\tau^*$ and
	\eqref{eq:Y-at-tauY} applies. Passing to $t=\tau_Y$ in
	\eqref{eq:Y-ineq-integrated} and using the definition of $\Omega_0$, we
	obtain
	\[
	2Y_0
	=Y(\tau_Y)
	\le
	Y_0-\frac14\tau_Y+\frac{Y_0}{2}
	\le
	\frac32Y_0,
	\]
	a contradiction. Hence $\tau_Y>T^*$.
	
	We may therefore evaluate \eqref{eq:Y-ineq-integrated} at $t=T^*$.
	Again using the definition of $\Omega_0$, we find
	\[
	Y(T^*)
	\le
	Y_0-\frac14T^*+\frac{Y_0}{2}
	=
	Y_0-2Y_0+\frac{Y_0}{2}
	=
	-\frac{Y_0}{2}<0.
	\]
	This contradicts the strict positivity of $Y$ on $[0,\tau^*)$.
	Therefore \eqref{eq:Omega0-implies-blowup} holds, and
	\begin{equation}\label{eq:pre-final-probability}
		\mathbb P\left(\tau^*\le\frac8{M_0}\right)
		\ge
		\mathbb P(\Omega_0).
	\end{equation}
	
	It remains to choose the threshold $C^*$. If $Y_0\le1$, then
	\begin{equation*}
		\overline\nu_Y
		\le
		c_\sigma(1+2V_{\infty})
		\triangleq
		\overline\nu_{Y,0}.
	\end{equation*}
	A sufficient choice is
	\begin{equation}\label{eq:C-star-choice}
		C^*
		\triangleq
		\max\left\{
		1,
		\ 24a_1,
		\ \sqrt{48c_2V_{\infty}},
		\ \bigl(192c_\sigma^2V_{\infty}^2\bigr)^{1/3},
		\ \frac{512\overline\nu_{Y,0}^{\,2}}{\varepsilon}
		\right\}.
	\end{equation}
	Indeed, if $M_0>C^*$, then $Y_0=1/M_0\le1$, and the first three
	nontrivial lower bounds in \eqref{eq:C-star-choice} imply
	\[
	2a_1Y_0\le\frac1{12},
	\qquad
	4c_2V_{\infty}Y_0^2\le\frac1{12},
	\qquad
	16c_\sigma^2V_{\infty}^2Y_0^3\le\frac1{12}.
	\]
	Thus \eqref{eq:Y0-condition} holds. Moreover,
	\[
	512\overline\nu_Y^{\,2}Y_0
	\le
	\frac{512\overline\nu_{Y,0}^{\,2}}{M_0}
	\le
	\varepsilon.
	\]
	Combining this estimate with \eqref{eq:good-probability} and
	\eqref{eq:pre-final-probability}, we conclude that
	\[
	\mathbb P\left(\tau^*\le\frac8{M_0}\right)
	\ge
	1-\varepsilon.
	\]
	This completes the proof of Theorem~\ref{Thm-blowup}.

	\begin{Remark}\label{Remark:Y=1/M necessary}
		The change of variables $Y=M^{-1}$ is not logically indispensable. Its purpose is to avoid a non-uniform upper localisation in $M$ and to
		incorporate all large-$M$ scales into a single estimate.
		
		To see the issue with a one-step argument at the level of $M$, assume
		that $M_0$ is sufficiently large that
		\begin{equation}\label{eq:M-drift-coercive-large}
			\frac12m^2-c_1m-c_2V_{\infty}\ge\frac14m^2,
			\quad m\ge\frac{M_0}{2}.
		\end{equation}
		For $N>M_0$, define
		\begin{equation*}
			\tau_{N}
			\triangleq
			\inf\bigl\{t\in[0,\tau^*):M(t)\ge N\bigr\},
		\end{equation*}
		and introduce the lower-level stopping time
		\begin{equation*}
			\wp
			\triangleq
			\inf\bigl\{t\in[0,\tau^*):M(t)\le M_0/2\bigr\}.
		\end{equation*}
		The latter is different from the sign stopping time
		\[
		\tau_{\mathrm{sgn}}
		=
		\inf\bigl\{t\in[0,\tau^*):M(t)\le0\bigr\}.
		\]
		By the progressive-extension convention (see Remark \ref{Remark-progressive measurability}), $\tau_N$, $\wp$, and
		$\tau_{\mathrm{sgn}}$ are stopping times, with
		$\inf\varnothing=\infty$. Positivity of $M$ rules out
		$\tau_{\mathrm{sgn}}$, whereas
		\eqref{eq:M-drift-coercive-large} is available only while $M$ remains
		above $M_0/2$.
		
		Set
		\[
		\tau_N^{\mathrm{loc}}
		\triangleq
		\wp\wedge\tau_{N}\wedge\tau^*
		\]
		and define the globally stopped martingale
		\[
		\mathfrak M_N(t)
		\triangleq
		\int_0^t
		\mathbf 1_{\{r<\tau_N^{\mathrm{loc}}\}}
		\Gamma(r) \d W(r).
		\]
		On $[0,\tau_N^{\mathrm{loc}})$, one has $M_0/2<M<N$, and hence
		\begin{equation*}
			\langle\mathfrak M_N\rangle_T
			\le
			c_\sigma^2(N+V_{\infty})^2T.
		\end{equation*}
		Consequently, Doob's inequality gives
		\begin{equation}\label{eq:MN-Doob-direct}
			\mathbb P\left(
			\sup_{0\le t\le T}
			|\mathfrak M_N(t)|
			>
			\frac{M_0}{4}
			\right)
			\le
			Cc_\sigma^2
			\frac{(N+V_{\infty})^2T}{M_0^2}.
		\end{equation}
		
		For $T>0$, define the good event
		\[
		\Omega_{N,T}^{M}
		\triangleq
		\left\{
		\sup_{0\le t\le T}|\mathfrak M_N(t)|
		\le\frac{M_0}{4}
		\right\}.
		\]
		On $\Omega_{N,T}^{M}$, the integral form of
		\eqref{eq:M-Ito-exact}, together with
		\eqref{eq:M-Ito-drift-bound} and
		\eqref{eq:M-drift-coercive-large}, yields, for every
		$0\le t<T\wedge\tau_N^{\mathrm{loc}}$,
		\begin{equation}\label{eq:M-Osgood-direct}
			M(t)
			\ge
			\frac{3M_0}{4}
			+
			\frac14\int_0^tM(r)^2\,\d r.
		\end{equation}
		We first rule out the lower stopping time. If
		\[
		\wp<T\wedge\tau_N\wedge\tau^*,
		\]
		then, by continuity and the definition of $\wp$,
		$M(\wp)=M_0/2$. Letting $t\uparrow\wp$ in
		\eqref{eq:M-Osgood-direct}, however, gives
		$M(\wp)\ge3M_0/4$, a contradiction.
		
		Now define, for $0\le t<T\wedge\tau_N\wedge\tau^*$,
		\[
		g(t)
		\triangleq
		\frac{3M_0}{4}
		+
		\frac14\int_0^tM(r)^2\,\d r.
		\]
		Then $M(t)\ge g(t)$ and
		\[
		g'(t)=\frac14M(t)^2\ge\frac14g(t)^2,\quad  \left(\frac1{g(t)}\right)'
		\le-\frac14,
		\qquad
		\frac1{g(t)}
		\le
		\frac4{3M_0}-\frac t4.
		\]
	
		Set
		\[
		T_0\triangleq\frac{16}{3M_0}.
		\]
		Taking $T=T_0$, if $\tau_N\wedge\tau^*>T_0$, then the preceding estimate
		holds for every $t<T_0$ and forces $g(t)\to\infty$ as $t\uparrow T_0$.
		This contradicts the continuity of $M$ on $[0,T_0]$ and the bound
		$M(t)<N$ there. Hence, on $\Omega_{N,T_0}^{M}$,
		\[
		\tau_N\wedge\tau^*\le\frac{16}{3M_0}.
		\]
		Using \eqref{eq:MN-Doob-direct} with $T=T_0$ therefore gives
		\begin{equation*} 
			\mathbb P\left(
			\tau_N\wedge\tau^*
			\le
			\frac{16}{3M_0}
			\right)
			\ge
			1-A_\sigma\frac{(N+V_{\infty})^2}{M_0^3}
		\end{equation*}
		for some constant $A_\sigma>0$.
		
		Since the error term grows with
		$N$, this estimate \textbf{cannot} be passed directly to the limit
		$N\to\infty$. Choosing $N=2M_0$ gives a high-probability doubling of
		$M$, or breakdown before the doubling time, but it does not by itself
		give explosion.
	\end{Remark}

	\subsection{Conditional blow-up rate along the transported maximum}
	\label{subsec:conditional-blowup-rate}
	
	We now prove Theorem~\ref{Thm-blowup-rate}.  The restriction to the event
	$\Omega_M$ in \eqref{eq:Omega-M-rate-main} is essential at the present
	level of the theory.  Indeed, the general blow-up criterion
	\eqref{eq:blow-up criterion statement} only asserts that at least one of
	$\|v_x(t)\|_{L^\infty}$ and $\|\Lambda v(t)\|_{L^\infty}$ becomes
	unbounded. It does not by itself imply that the quantity $M(t)$ attached to
	the transported maximum diverges.  No largeness assumption on $M_0$ is
	needed in this subsection.
	
	Recall the following quantities (cf. \eqref{eq:energy-rate-definition} and \eqref{eq:kappa-rate-define})
	\begin{equation*}
		\mathcal E_M(t)
		\triangleq
		\mathcal E[v(t,\cdot)](z_0(t))
		\ge0,\qquad  \kappa_{\Gamma}(t)
		\triangleq
		\frac{\d}{\d t}\langle\Gamma,W\rangle_t,
	\end{equation*}
	As pointed out in Remark \ref{Remark-kappa-measurability} and the proof of Lemma \ref{lem:dM lemma}, $\mathcal E_M$ and $\kappa_{\Gamma}(t)$ are
	progressively measurable and locally integrable on
	$\mathcal I_{\tau^*}$. Since $M$ is continuous and strictly
	positive there, the same is true locally for
	$\eta=\mathcal E_M/M^2$.

	By \eqref{eq:kappa-rate-define},
	\eqref{eq:Xi-rate-bound}, and \eqref{eq:Gamma-bound}, there exists a
	constant $C_{\sigma,4}>0$, depending only on $\sigma$, such that
	\begin{equation}\label{eq:kappa-rate-two-sided}
		|\kappa_{\Gamma}(t)|
		\le
		C_{\sigma,4}\bigl(M(t)+V_{\infty}\bigr),
		\qquad 0\le t<\tau^*.
	\end{equation}

	As in Lemma \ref{lem:dM lemma}, the It\^o form of \eqref{eq:M-Strat} is
	\begin{equation*} 
		\d M(t)
		=
		\left[
		\frac12M(t)^2+\mathcal E_M(t)+\frac12\kappa_{\Gamma}(t)
		\right]\d t
		+
		\Gamma(t)\d W(t).
	\end{equation*}
	Since $M(t)>0$ by Lemma~\ref{lem:M-positive}, It\^o's formula applied to
	$Y=M^{-1}$ yields
	\begin{equation}\label{eq:Y-rate-exact}
		\begin{aligned}
			\d Y(t)
			={}&
			\left[
			-\frac12
			-\frac{\mathcal E_M(t)}{M(t)^2}
			-\frac12\frac{\kappa_{\Gamma}(t)}{M(t)^2}
			+\frac{\Gamma(t)^2}{M(t)^3}
			\right]\d t
			-
			\frac{\Gamma(t)}{M(t)^2}\d W(t).
		\end{aligned}
	\end{equation}
Recall the quantities $\eta(t)$ and $	\overline\eta(t)$ defined in \eqref{eq:eta-rate-main}:
	\begin{equation*}
		\eta(t)
		\triangleq
		\frac{\mathcal E_M(t)}{M(t)^2},	\qquad
		\overline\eta(t)
		\triangleq
		\frac1{\tau^*-t}
		\int_t^{\tau^*}\eta(r)\d r.
	\end{equation*}
	We also define
		\begin{equation*}
		\mu_Y(t)
		\triangleq
		-\frac12\frac{\kappa_{\Gamma}(t)}{M(t)}
		+
		\left(\frac{\Gamma(t)}{M(t)}\right)^2,
		\qquad
		\nu_Y(t)
		\triangleq
		-\frac{\Gamma(t)}{M(t)}.
	\end{equation*}
	Then \eqref{eq:Y-rate-exact} becomes the following equation with multiplicative noise:
	\begin{equation}\label{eq:Y-rate-linear-form}
		\d Y(t)
		=
		\left[
		-\left(\frac12+\eta(t)\right)
		+\mu_Y(t)Y(t)
		\right]\d t
		+
		\nu_Y(t)Y(t)\d W(t).
	\end{equation}
	Notice that the nonlocal term has not been discarded: it is retained in the
	nonnegative, scale-invariant quantity $\eta(t)=\mathcal E_M(t)/M(t)^2$.
	
	The bounds \eqref{eq:Gamma-bound} and \eqref{eq:kappa-rate-two-sided} imply
	\begin{equation}\label{eq:mu-nu-rate-bound}
		|\nu_Y(t)| \le
		c_\sigma\bigl(1+V_{\infty}Y(t)\bigr),\qquad
		|\mu_Y(t)| \le
		\frac{C_{\sigma,4}}2\bigl(1+V_{\infty}Y(t)\bigr)
		+
		c_\sigma^2\bigl(1+V_{\infty}Y(t)\bigr)^2.
	\end{equation}
	In particular, $\mu_Y$ and $\nu_Y$ are uniformly bounded on every region
	where $M\ge1/2$.  On $\Omega_M$, one has $Y(t)\to0$ and therefore there is a
	path-dependent time $t_0<\tau^*$ such that
	\begin{equation}\label{eq:M-eventually-large}
		M(t)\ge1,
		\qquad t_0\le t<\tau^*.
	\end{equation}
	
	\begin{Lemma}
		\label{lem:terminal-integrating-factor}
		There exists a strictly positive continuous adapted process
		$\mathscr Z=(\mathscr Z(t))_{t\ge0}$ such that, almost surely on
		$\Omega_M$, the random variable $\mathscr Z(\tau^*)$ is finite and
		strictly positive and, for all $t<\tau^*$ sufficiently close to
		$\tau^*$,
		\begin{equation}\label{eq:Y-terminal-representation}
			Y(t)
			=
			\int_t^{\tau^*}
			\frac{\mathscr Z(r)}{\mathscr Z(t)}
			\left(\frac12+\eta(r)\right)\d r.
		\end{equation}
		Moreover,
		\begin{equation}\label{eq:Z-ratio-to-one}
			\lim_{t\uparrow\tau^*}
			\sup_{r\in[t,\tau^*]}
			\left|
			\frac{\mathscr Z(r)}{\mathscr Z(t)}-1
			\right|
			=0.
		\end{equation}
	\end{Lemma}
	
	\begin{proof}
		Choose $\chi_\infty\in C^\infty([0,\infty);[0,1])$ such that
		\[
		\chi_\infty(m)=0\quad\text{for }0\le m\le\frac12,
		\qquad
		\chi_\infty(m)=1\quad\text{for }m\ge1.
		\]
		Using the progressive extension by zero after $\tau^*$ (see Remark \ref{Remark-progressive measurability}), define, for
		$t\ge0$,
		\[
		\widehat\mu_Y(t)
		\triangleq
		\begin{cases}
			\chi_\infty(M(t))\mu_Y(t),&t<\tau^*,\\
			0,&t\ge\tau^*,
		\end{cases}
		\qquad
		\widehat\nu_Y(t)
		\triangleq
		\begin{cases}
			\chi_\infty(M(t))\nu_Y(t),&t<\tau^*,\\
			0,&t\ge\tau^*.
		\end{cases}
		\]
		Whenever $\chi_\infty(M(t))\neq0$, one has $M(t)>1/2$ and hence
		$Y(t)<2$. Therefore \eqref{eq:mu-nu-rate-bound} gives the deterministic
		bounds
		\[
		|\widehat\nu_Y(t)|
		\le
		c_\sigma(1+2V_\infty)
		\]
		and
		\[
		|\widehat\mu_Y(t)|
		\le
		\frac{C_{\sigma,4}}2(1+2V_\infty)
		+
		c_\sigma^2(1+2V_\infty)^2.
		\]
		Thus $\widehat\mu_Y$ and $\widehat\nu_Y$ are bounded progressively
		measurable processes. Define, for $t\ge0$,   
		\begin{equation}\label{eq:Z-rate-definition}
			\mathscr Z(t)
			\triangleq
			\exp\!\left[
			-\int_0^t\widehat\nu_Y(r)\d W(r)
			+
			\int_0^t
			\left(
			\frac12\widehat\nu_Y(r)^2-\widehat\mu_Y(r)
			\right)\d r
			\right].
		\end{equation}
		It\^o's formula gives
		\begin{equation}\label{eq:Z-rate-SDE}
			\d\mathscr Z(t)
			=
			\left[
			\widehat\nu_Y(t)^2-\widehat\mu_Y(t)
			\right]\mathscr Z(t)\d t
			-
			\widehat\nu_Y(t)\mathscr Z(t)\d W(t).
		\end{equation}
		The definition \eqref{eq:Z-rate-definition} shows directly that
		$\mathscr Z$ is a globally defined, strictly positive continuous adapted
		process.
		
		Let $\Omega_{\mathrm{full}}$ be a common event of probability one on
		which all localised semimartingale identities used above, the chosen
		continuous versions, and the continuous stochastic integral in
		\eqref{eq:Z-rate-definition} hold simultaneously. Fix
		\[
		\omega\in\Omega_M\cap\Omega_{\mathrm{full}}
		\]
		and suppress $\omega$ from the notation. By \eqref{eq:M-eventually-large},
		$\widehat\mu_Y=\mu_Y$ and $\widehat\nu_Y=\nu_Y$ on
		$[t_0,\tau^*)$.  Combining \eqref{eq:Y-rate-linear-form} and
		\eqref{eq:Z-rate-SDE} with the product It\^o formula, including the
		quadratic-covariation term, gives the exact cancellation
		\begin{equation}\label{eq:ZY-rate-differential}
			\d\bigl(\mathscr Z(t)Y(t)\bigr)
			=
			-\mathscr Z(t)
			\left(\frac12+\eta(t)\right)\d t,
			\qquad t_0\le t<\tau^*.
		\end{equation}
		Indeed, the coefficient of $\mathscr ZY\d t$ is
		\[
		\mu_Y+(\nu_Y^2-\mu_Y)-\nu_Y^2=0,
		\]
		and the two stochastic terms cancel as well.
		
		For every integer $n>M_0$, let
		\begin{equation*} 
			\tau_n
			\triangleq
			\inf\bigl\{
			t\in[0,\tau^*):M(t)\ge n
			\bigr\}.
		\end{equation*}
		These are stopping times by the progressive-extension convention (see Remark \ref{Remark-progressive measurability}). On
		$\Omega_M$, continuity, $M(0)=M_0<n$, and
		$M(t)\to\infty$ as $t\uparrow\tau^*$ imply, for every integer
		$n>M_0$,
		\[
		\tau_n<\tau^*,
		\qquad
		M(\tau_n)=n,
		\qquad
		Y(\tau_n)=\frac1n.
		\]
		The sequence $\{\tau_n\}_{n\ge1}$ is nondecreasing. Let
		$\tau_\infty\triangleq\lim_n\tau_n\le\tau^*$. If
		$\tau_\infty<\tau^*$, continuity of $M$ at $\tau_\infty$ gives
		\[
		M(\tau_n)\longrightarrow M(\tau_\infty)<\infty,
		\]
		which contradicts $M(\tau_n)=n\to\infty$. Hence
		$\tau_n\uparrow\tau^*$.

		Fix $t\in[t_0,\tau^*)$ and take $n$ sufficiently large that
		$\tau_n>t$.  Integrating \eqref{eq:ZY-rate-differential} from $t$ to
		$\tau_n$ yields
		\begin{equation}\label{eq:ZY-rate-stopped}
			\mathscr Z(t)Y(t)
			=
			\frac{\mathscr Z(\tau_n)}n
			+
			\int_t^{\tau_n}
			\mathscr Z(r)
			\left(\frac12+\eta(r)\right)\d r.
		\end{equation}
		
		Since $\mathscr Z$ is globally defined, continuous, and strictly positive,
		and since $\tau^*<\infty$ on $\Omega_M$,
		\begin{equation*} 
			\mathscr Z(t)\longrightarrow
			\mathscr Z(\tau^*)\in(0,\infty)
			\qquad\text{as }t\uparrow\tau^*.
		\end{equation*}
		In particular, $\mathscr Z(\tau_n)/n\to0$.  Since the integrand in
		\eqref{eq:ZY-rate-stopped} is nonnegative, monotone convergence as
		$n\to\infty$ gives \eqref{eq:Y-terminal-representation}.  It also
		shows that
		\[
		\int_t^{\tau^*}\eta(r)\d r<\infty
		\]
		for every $t$ sufficiently close to $\tau^*$. Finally, the strict positivity and continuity of $\mathscr Z$ on the
		compact terminal interval $[t_0,\tau^*]$ imply
		\[
		\sup_{r\in[t,\tau^*]}
		\left|
		\frac{\mathscr Z(r)}{\mathscr Z(t)}-1
		\right|
		\le
		\frac{
			\sup_{r_1,r_2\in[t,\tau^*]}|\mathscr Z(r_1)-\mathscr Z(r_2)|
		}{
			\inf_{r\in[t,\tau^*]}\mathscr Z(r)
		}
		\longrightarrow0,
		\]
		which proves \eqref{eq:Z-ratio-to-one}.
	\end{proof}
	
	\begin{proof}[Proof of Theorem~\ref{Thm-blowup-rate}]
		Work on $\Omega_M$ and write
		\[
		\ell_*(t)\triangleq\tau^*-t,
		\qquad
		\overline\eta(t)
		=
		\frac1{\ell_*(t)}
		\int_t^{\tau^*}\eta(r)\d r.
		\]
		By Lemma~\ref{lem:terminal-integrating-factor},
		\begin{equation}\label{eq:Y-rate-average}
			\frac{Y(t)}{\ell_*(t)}
			=
			\frac1{\ell_*(t)}
			\int_t^{\tau^*}
			\frac{\mathscr Z(r)}{\mathscr Z(t)}
			\left(\frac12+\eta(r)\right)\d r.
		\end{equation}
		Set
		\[
		\mathfrak e_Z(t)
		\triangleq
		\sup_{r\in[t,\tau^*]}
		\left|
		\frac{\mathscr Z(r)}{\mathscr Z(t)}-1
		\right|.
		\]
		Then $\mathfrak e_Z(t)\to0$ by \eqref{eq:Z-ratio-to-one}.  For $t$
		sufficiently close to $\tau^*$, $\mathfrak e_Z(t)<1$, and
		\begin{equation}\label{eq:Y-rate-squeeze}
			\bigl(1-\mathfrak e_Z(t)\bigr)
			\left(\frac12+\overline\eta(t)\right)
			\le
			\frac{Y(t)}{\ell_*(t)}
			\le
			\bigl(1+\mathfrak e_Z(t)\bigr)
			\left(\frac12+\overline\eta(t)\right).
		\end{equation}
		
		Since $\overline\eta(t)\ge0$, the lower bound in
		\eqref{eq:Y-rate-squeeze} gives
		\[
		\liminf_{t\uparrow\tau^*}
		\frac{Y(t)}{\tau^*-t}
		\ge\frac12.
		\]
		Using $M(t)=1/Y(t)$, we obtain
		\[
		\limsup_{t\uparrow\tau^*}
		M(t)(\tau^*-t)
		\le2,
		\]
		which proves \eqref{eq:type-I-rate-main}.
		
		If \eqref{eq:eta-minus-main} holds, then the same lower bound yields
		\[
		\liminf_{t\uparrow\tau^*}
		\frac{Y(t)}{\tau^*-t}
		\ge
		\frac12+\eta_{\mathrm{av}}^-.
		\]
		Taking reciprocals proves \eqref{eq:strict-type-I-rate-main}. A strictly
		positive lower limit of the terminal Ces\`aro averages yields a strict
		improvement of the universal upper bound $2$.
		
		Finally, suppose that \eqref{eq:eta-star-main} holds with
		$\eta_*<\infty$.  From \eqref{eq:Y-rate-average},
		\begin{align*}
			\left|
			\frac{Y(t)}{\ell_*(t)}
			-
			\left(\frac12+\overline\eta(t)\right)
			\right|
			&\le
			\mathfrak e_Z(t)
			\left(\frac12+\overline\eta(t)\right)
			\longrightarrow0.
		\end{align*}
		Therefore
		\[
		\frac{Y(t)}{\tau^*-t}
		\longrightarrow
		\frac12+\eta_*.
		\]
		The limit on the right is strictly positive, so taking reciprocals
		gives \eqref{eq:exact-rate-main}.
	\end{proof}
	
	\begin{Remark}
		The strict inequality in \eqref{eq:strict-type-I-rate-main} cannot be
		concluded from $\mathcal E_M(t)\ge0$ alone.  Within the event $\Omega_M$, no additional assumption on the terminal average of
		$\eta(t)=\mathcal E_M(t)/M(t)^2$  is required for the upper bound with constant $2$.
		If
		$\overline\eta(t)\to0$, then \eqref{eq:exact-rate-main} gives the
		borderline constant
		\[
		M(t)(\tau^*-t)\longrightarrow2.
		\]
		If instead $\overline\eta(t)\to\eta_*$ for some
		$\eta_*\in(0,\infty)$, then the limiting constant is strictly smaller
		than $2$.
	\end{Remark}
	
	\begin{Remark}
		The preceding proof does not require $\sigma$ to be affine, the
		characteristic $z_0(t)$ to have a limit, or the initial profile to be
		symmetric. The explicit stochastic coefficients in the reciprocal equation
		enter the terminal representation through the multiplicative factor
		$\mathscr Z(r)/\mathscr Z(t)$, which converges uniformly to one on the
		shrinking interval $[t,\tau^*]$. Thus this explicit multiplicative factor
		does not contribute separately to the first-order asymptotic coefficient.
		This does not make the rate independent of the noise: the processes $v$,
		$z_0$, $\mathcal E_M$, and hence $\eta=\mathcal E_M/M^2$, remain path
		dependent. Any remaining dependence of the first-order coefficient on the
		noise is encoded implicitly in the terminal Ces\`aro behaviour of $\eta$.
		When that average converges to a finite limit $\eta_*$, the limiting
		coefficient is $1/(\frac12+\eta_*)$.
		
		The result remains conditional on $\Omega_M$.  Extending it to the whole
		event $\{\tau^*<\infty\}$ requires a separate single-point-dominance
		argument showing that breakdown necessarily forces
		$M(t)=\Lambda v(t,z_0(t))$ to diverge along the transported maximum.
	\end{Remark}

	\section*{Acknowledgements}

D.A.-O. is supported by Grant RYC2023-045563-I funded by
MICIU/AEI/10.13039/501100011033.  

D.A.-O. and R.G.-B. are funded by the
project ``An\'alisis Matem\'atico Aplicado y Ecuaciones Diferenciales''
(AMAED), Grant PID2022-141187NB-I00, funded by
MCIN/AEI/10.13039/501100011033/FEDER, UE. This publication is part of the
project PID2022-141187NB-I00 funded by
MCIN/AEI/10.13039/501100011033. 

Y.T. M. is funded by the National Natural Science Foundation of China under project 12526619.  Y.T. M. is also supported by the
Research Development Fund (RDF-24-01-063) from  Xi'an Jiaotong-Liverpool University.

H.T.  is supported by the National Natural Science Foundation of China under projects  24AAA00530 and 12531007.

 D.A.-O., R.G.-B., and H.T. acknowledge the hospitality of the Tianyuan Mathematics Research Center in Kunming, where part of this work was carried out during the workshop ``Nonlocality in Stochastic Evolutionary Systems.''

	\setlength{\bibsep}{1ex}

\end{document}